\documentclass[12pt]{article} 

\usepackage{amsmath,amsfonts,amssymb,mathtools,amsthm}
\usepackage{mathrsfs}

\usepackage[all]{xy}
\usepackage{psfrag}

\theoremstyle{plain}

\newtheorem{global-theorem}{Theorem}
\newtheorem{theorem}{Theorem}[section]
\newtheorem{lemma}[theorem]{Lemma}
\newtheorem{scholium}[theorem]{Scholium}

\newtheorem{corollary}[theorem]{Corollary}

\newtheorem{definition}[theorem]{Definition}
\newtheorem{proposition}[theorem]{Proposition}

\theoremstyle{definition}

\newtheorem{hypothesis}[theorem]{Hypothesis}
\newtheorem{remark}[theorem]{Remark}

\newcommand{\cc}{{\mathbb C}}
\newcommand{\pp}{{\mathbb P}}
\newcommand{\rr}{{\mathbb R}}

\newcommand{\Gm}{{\mathbb G}_{\rm m}}
\newcommand{\aaaa}{{\mathbb A}}

\newcommand{\Dd}{{\mathcal D}}
\newcommand{\Ff}{{\mathcal F}}
\newcommand{\Oo}{{\mathcal O}}

\newcommand{\Vv}{{\mathcal V}}

\newcommand{\Cc}{{\mathcal C}}

\newcommand{\Ll}{{\mathcal L}}

\newcommand{\Hh}{{\mathcal H}}

\newcommand{\srA}{{\mathscr A}}

\newcommand{\srE}{{\mathscr E}}
\newcommand{\srF}{{\mathscr F}}
\newcommand{\srD}{{\mathscr D}}
\newcommand{\srQ}{{\mathscr Q}}
\newcommand{\srX}{{\mathscr X}}
\newcommand{\srY}{{\mathscr Y}}
\newcommand{\srZ}{{\mathscr Z}}
\newcommand{\srT}{{\mathscr T}}

\newcommand{\srR}{{\mathscr R}}
\newcommand{\srK}{{\mathscr K}}
\newcommand{\srC}{{\mathscr C}}
\newcommand{\srP}{{\mathscr P}}
\newcommand{\srG}{{\mathscr G}}

\newcommand{\scD}{\textsf{\textbf{D}}}

\newcommand{\Dol}{{\text{DOL}^{\text{par}}_{L^{2}}}}
\newcommand{\End}{{\text{End}}}

\newcommand{\delbar}{\overline{\partial}}
\newcommand{\ddual}{{}^{\vee}\!}

\title{Direct Images in Non Abelian Hodge Theory}
\author{R. Donagi, T. Pantev and C. Simpson}
\date{} 

\begin{document}
\maketitle

\begin{abstract}
In this paper we explain how non-abelian Hodge theory 
allows one to
compute the $L^{2}$ cohomology or middle perversity higher direct images of harmonic bundles and twistor $D$-modules
in a purely algebraic manner. Our main result is a new algebraic
description for the fiberwise $L^{2}$ cohomology 
of a tame harmonic bundle or the corresponding flat bundle or
tame polystable parabolic Higgs bundle. 
Specifically we give a formula
for the Dolbeault version of the $L^{2}$ pushforward 
in terms of a modification of the Dolbeault complex of a 
Higgs bundle which takes into account the monodromy 
weight filtration in the normal directions
of the horizontal parabolic divisor. The parabolic structure
of the higher direct image is obtained by analyzing 
the $V$-filtration at a normal crossings point.
We prove this algebraic formula
for semistable families of curves. 
\end{abstract}

\tableofcontents

\newpage

\section{Introduction}

The Non Abelian Hodge Correspondence (NAHC) on a variety $X$ 
is the equivalence

\

\[
\boxed{
  \left(
  \text{
  \begin{minipage}{1.2in}
      semisimple local systems on $X$
  \end{minipage}
  }
  \right)
\Longleftrightarrow
\left(
  \text{
    \begin{minipage}{2in}
      polystable Higgs bundles on $X$ with vanishing Chern classes $c_1=c_2=0$
\end{minipage}
  }
  \right)
 }
\]

\

\noindent given by Hitchin's equations.
Here a local system $L$ is a representation of the fundamental group of $X$,
while a Higgs bundle on $X$ is a pair $(E,\varphi)$ where 
$E$ is a vector bundle on $X$ 
and the Higgs field 
$\varphi \in \Gamma(X, \End(E)\otimes \Omega^1_X)$
self-commutes, in the sense that  
\[
\varphi \wedge \varphi =0  \in \Gamma(X, \End(E)\otimes \Omega^2_X),
\]
so there is a {\em Dolbeault complex}:
\begin{equation} \label{Koszul}
  \text{DOL}(X,(E,\varphi)):= \left[
    E \stackrel{\varphi}{\to} E\otimes \Omega^1_X
  \stackrel{\varphi}{\to} E\otimes \Omega^2_X
  \stackrel{\varphi}{\to} \ldots
  \right].
\end{equation}
There is a richer object called a harmonic bundle, reviewed in section
\ref{sec-hb}, which naturally determines both a Higgs bundle
$(E,\varphi)$ and a local system $L$.  In fact, it determines a family
of $\lambda$-connections for all $\lambda \in \cc$, with the Higgs
bundles arising at $\lambda=0$ and the local system, or equivalently a
flat connection, at $\lambda=1$.  The NAHC (``Kobayashi-Hitchin
correspondence'' in Mochizuki's terminology) says that this sets up an
equivalence between Higgs bundles, local systems, and harmonic
bundles.  The equivalence, in the compact case, is due in one
direction to Hitchin \cite{Hitchin} for one dimensional $X$ and to
Simpson \cite{SimpsonHBLS} in higher dimension.  The other direction
is provided by Donaldson's appendix to Hitchin's paper for curves, and
Corlette \cite{Corlette} in all dimensions.

A version of the NAHC on non compact spaces $X \setminus D$, where $X$
is compact and $D \subset X$ is a normal crossing divisor, was
established in \cite{SimpsonHBNC} for one dimensional $X$, in
\cite{Biquard} for $X$ of arbitrary dimension with smooth $D$, and in
\cite{MochizukiKobayashiHitchin} in general.  The setup involves a
local system $L$ and a Higgs bundle $(E,\varphi)$ that are defined on
$X \setminus D$ and carry order of growth filtrations along the
components of $D$.  More precisely, $L$ is a polystable filtered local
system with vanishing parabolic Chern classes, and $(E,\varphi)$ is a
parabolic Higgs bundle, consisting of a locally abelian parabolic
vector bundle $E$ with vanishing parabolic Chern classes, together
with a Higgs field $\varphi$ that is logarithmic with respect to the
parabolic structure along $D$.  The theorem states that both types of
data are equivalent to a {\em tame} harmonic bundle on $X \setminus
D$.  We refer to section \ref{sec-hb} for more details.

Natural operations such as pullback and direct image
 with respect to a morphism \linebreak $f:X \to Y$
are well defined for harmonic bundles and local systems, 
and commute with the NAHC. 
Pullback for Higgs bundles is also well defined and commutes with the NAHC.
The main goal of the present work is 
to give an algebraic definition of direct image for Higgs bundles
and to show that it too commutes with the NAHC.

Consider the case that $X$ is compact and $Y$ is a point.  The direct
image of a Higgs bundle $(E,\varphi)$ on $X$ should be an object of
the bounded derived category of coherent sheaves on the point, so it
should be described by a complex of vector spaces.  For this we take
the Dolbeault complex \eqref{Koszul}.  This is the correct choice in
the sense that if $(E,\varphi)$ corresponds to $L$ by the NAHC then
there is a natural isomorphism:
\[
Rf_*L \stackrel{\cong}{\longrightarrow}
\mathbb{R}\Gamma(X,\text{DOL}(X,(E,\varphi))). 
\]
For example, the trivial local system $L=\cc$ corresponds to the
trivial vector bundle $E=\Oo$ along with the the null Higgs field
$\varphi=0$. In this case the above isomorphism is the ordinary Hodge
theorem: $Rf_*\cc$ gives the de Rham cohomology
$H^{\bullet}_{DR}(X,\cc)$, 
while
$\text{DOL}(X,(E,0)) = \oplus_{i} \Omega^{i}_{X}[-i]$ are the
holomorphic forms on $X$, which give the Dolbeault cohomology
$H^{\bullet}_{Dol}(X,\cc) = \oplus_{p,q} H^{p}(X,\Omega^{q}_{X})$.

The analogous result in the open case, when $X$ is one dimensional,
was established by Zucker \cite{Zucker} for variations of Hodge
structure.
In this work we extend these results to families of spaces and to
general local systems.

Our main result is Theorem \ref{mainth}. We consider a map $f:X \to Y$
from a smooth projective surface $X$ to a smooth projective curve $Y$,
with a reduced divisor $Q\subset Y$, and a (reduced) simple normal
crossings divisor $D=D_V + D_H\subset X$, consisting of horizontal and
vertical parts: $D_H$ is etale over $Y$, while $D_V = f^{-1}(Q)$.
Starting from a Higgs bundle $(E,\varphi)$ on $X \setminus D$, we
construct in \eqref{l2pardol} the $L^2$ Dolbeault complex ${\rm
  DOL}^{\rm par}_{L^2}(X/Y,E_{\alpha (a)})$. This construction is
purely algebraic.  The theorem asserts that the direct images $F^i_a:=
{\mathbb R}^if_{\ast} \left( {\rm DOL}^{\rm par}_{L^2}(X/Y,E_{\alpha
  (a)}) \right) $ are the correct direct images for the Higgs bundle
$(E,\varphi)$, in the sense that they fit together\footnote{Our
  notational convention will be that a parabolic structure is denoted
  by an underlined letter; it consists of a collection of sheaves
  indexed by the parabolic levels, but the letters for these component
  sheaves are not underlined.}  into the parabolic Higgs bundles
$\underline{F}^i$ that correspond under NAHC to the direct images of
the various objects (harmonic bundle and filtered local system)
corresponding under NAHC to the original Higgs bundle $(E,\varphi)$

As discussed in Section \ref{ancon} at the end,
one can envision analytic arguments going in the direction
of the proof. But the essential ingredient would be to know 
that the higher direct image sheaves $F^i_a$ are locally free. 
That should be viewed as some kind of strictness property,
not easy to obtain ``by hand''. 

It is therefore natural to use Sabbah's theory of
twistor $\Dd$-modules \cite{Sabbah}, generalizing the technique of
Saito \cite{SaitoMHP} for
Hodge modules. This allows furthermore to
leave the analytic considerations in the background
as they are already treated within the context of 
the general theory. 

Our strategy of proof may be described, in general terms, as follows:
Sabbah has already proven a compatibility between
$V$-filtrations and higher direct images.
The $V$-filtration uses in an intrinsic way the
structure of module over the ring $\srR$ obtained by using the Rees
construction of the standard filtration of $\Dd _X$.

The required strictness property is reflected in the
Decomposition Theorem of Sabbah and Mochizuki 
\cite{Sabbah, MochizukiPure}. 

In order to pass between the $\srR$-module picture and the parabolic
Higgs picture, we need to understand the $V$-filtration and nearby
cycles functors in a relative way over the $\lambda$-line $\srA :=
\aaaa^1$.

Let us discuss this first on $Y$. The $\srR$-module
$\widetilde{\srF}^i$, defined as the direct image of the relative de
Rham complex for $\srR$-modules, has a $V$-filtration along the
divisor $Q$. On the base curve we could restrict to an open disk so we
may think of $Q$ as consisting of a single point, in particular we use
only a single real weight $a$. We have the filtration by subsheaves
$$
V_{a-1}\widetilde{\srF}^i\subset \widetilde{\srF}^i .
$$
These are $V_0\srR$-modules, where the latter ring is the one
associated to the sheaf of rings of differential operators that is
Koszul dual to the logarithmic cotangent complex $\Omega
^{\bullet}_Y(\log Q)$.

In the range $a<1$ these sheaves fit together to form a structure of
parabolic bundle. That is to say they are locally free $\Oo
_{Y\times \srA}$-modules organized in a filtration satisfying the axioms of a
parabolic structure. One should extend by periodicity to the case
$a\geq 1$.

The restriction of this parabolic structure to any $\lambda$ is the
parabolic bundle associated to the harmonic bundle associated to the
middle higher direct image. This comes from Saito-Mochizuki-Sabbah's
theory. The reader may refer to \cite[Chapter 5]{Sabbah} for a review
of these relationships.

In order to prove our main theorem, we would like to show that
the natural isomorphism of bundles on $Y-Q$ extends to give
an isomorphism
$$
V_{a-1}\widetilde{\srF}^i|_{Y\times \{0\}} \cong 
F^i_a.
$$
The strategy to show this is to use the compatibility of
$V$-filtrations with higher direct images: we can define a
$V$-filtration on $\srE $ over $\srX = X\times \srA$, and Sabbah shows
\cite[Theorem 3.1.8]{Sabbah} following Saito \cite[Proposition
  3.3.17]{SaitoMHP} that the higher direct image of the de Rham
complex made from a level of this $V$-filtration, is the corresponding
level of the $V$-filtration on the higher direct image.

A first question for obtaining information on the Dolbeault fiber
$\lambda =0$ is that the axiomatic characterization of the
$V$-filtration uses the full $\srR$-module structure over the
$\lambda$-line $\srA$, so it is not {\em a priori} well-defined just
in terms of Dolbeault data. This aspect was pointed out to us rather early on by Sabbah and Mochizuki. 
However, it turns out that Saito and Mochizuki 
have proven useful expressions for
the $V$-filtration viewed as being generated by certain subsheaves
closely related to the parabolic structure. This is stated
in Proposition \ref{vdescrip}. Furthermore, we furnish our own proof of
how to get from the weaker to the stronger version of
the generation statement, 
in Corollary \ref{vgen2} and Lemma \ref{vgen3}. 
These expressions are
well-defined in the restriction to $\lambda =0$.

There are now two difficulties. Near a horizontal divisor (but that is
automatically away from the singularities of $f$), we would like to
further reduce from the de Rham complex for an $\srR_{\srX}$-module
which is not $\Oo _{\srX}$-coherent, to an $L^2$ de Rham complex. For
this, we can just quote the relevant statement in Sabbah, using
the weight filtration given by Corollary \ref{weightcor}.

The main difficulty is the following: 
at a singular point of $f$,
corresponding to the intersection of two vertical divisor components,
the map $f$ is not smooth so Saito and Sabbah do not directly
define the $V$-filtration along $f=0$ (here we are assuming that $Y$
is a disk and the singular point is at the origin $y=0$).

Rather, in order to define the $V$-filtration for the
function $f$, they tell us to take the graph embedding
and use the $V$-filtration for the new coordinate function. We
therefore have to prove a statement about compatibility between the
complex obtained from the graph construction, and the $L^2$ parabolic
Dolbeault complex. It is this compatibility statement that will be our
main result.

\vspace{1cm}

Here is an outline of the paper.

Section \ref{prelim} is preparatory, and includes a review of some of
the relevant issues from non-abelian Hodge theory.  We start in
subsection \ref{geom} by describing our basic geometric situation.  In
subsection \ref{sec-hb} we review the non-abelian Hodge correspondence
in both the compact and non-compact cases. We also explain the
Nilpotence Hypothesis \eqref{mainhyp} that will be assumed
throughout. There and in the next subsection we give the main
definitions: harmonic bundle (wild, tame, tame with trivial
filtrations), $\lambda$-connection.  A central theme of this work will
be the comparison of harmonic bundles with $\srR$-modules. In
subsection \ref{sec-r} we define the relevant $\srR$-modules,
especially the minimal extension.

In section \ref{setup} we state and discuss the main theorem. We start
in section \ref{sec-locstudy} by constructing the monodromy weight
filtration $W({\rm Gr}_{k,b} (\underline{E})):= W(N_{k,b})$ associated
to a parabolic Higgs bundle $(\underline{E},\varphi)$. In section
\ref{L2dol}, these weight filtrations are used to construct the $L^2$
parabolic Dolbeault complex, in equation \eqref{l2pardol}. This allows
us, in section \ref{Main}, to state the Main Theorem \ref{mainth}.

In section \ref{sec-dr}, as a step towards the proof of the Main
Theorem \ref{mainth}, we formulate its analogue, Theorem \ref{rthm},
on the level of $\srR$-modules and their $V$-filtrations.  This
result, in the $\srR$-module context, is already known, by
Sabbah's work.  So this leads us to the comparison problem of relating
these two points of view, in order to deduce our main theorem from
Theorem \ref{rthm}.  In a little more detail: we start in subsection
\ref{j} with a brief review of the Kashiwara-Malgrange $V$-filtration
of a $\Dd$-module with respect to a map or a subscheme. While the
smooth case is straightforward, the singular case requires factoring
through a graph construction, as we recall in section \ref{j}. In
section \ref{dR} we define the de Rham complex on $\srX$ in terms of
this graph factorization, take its higher direct image to $\srY = Y\times\srA$, and
then state and prove Theorem \ref{rthm} about compatibility with the
$V$-filtration.  Finally, in subsection \ref{constructmap}, we
approach the comparison problem by constructing a map $u(a)$
(cf. Lemma \ref{mapcomplexes}) from the parabolic de Rham complex to
the new de Rham complex. The proof of the Main Theorem \ref{mainth} is
then reduced, in Proposition \ref{proofmainth}, to Theorem
\ref{cpxqis} which asserts that the comparison maps $u(a), u_0(a)$ are
indeed quasi-isomorphisms. This is proved over the next two sections.

In section \ref{sec-qi}, after some preliminaries and reductions, the
desired quasi-isomorphism is established over smooth points, thus
focusing our attention, in Theorem \ref{maincalc}, on the double
points.  Some further reductions lead us to consider, in Proposition
\ref{reduction2}, a graded version ${\rm gr}_a(u_0)$ described in
equation \ref{gru}.

The needed local calculations are then completed in section
\ref{sec-nc}. These include a Tensor Product Formula and explicit
calculations with Koszul complexes.

In the final section we consider several improvements and extensions.
The first two subsections provide some details on the proofs of parts
3 and 5 of Theorem \ref{mainth}.  We start by looking at the absolute
Dolbeault complex on $X$ and use it to define the Gauss-Manin Higgs
field on the higher direct images. Next we look at the analytical
aspects of the direct image harmonic bundle. This includes a sketch of
our original strategy for proving the theorem, involving the study of
the family of $L^2$ cohomology spaces. The remaining sections consider
extensions of our basic geometric setup to allow maps between various
higher dimensional spaces.

\bigskip

\noindent
{\bf Acknowledgements}---We would like to thank Takuro Mochizuki and
Claude Sabbah for helpful discussions all along this project.  We also
thank Mochizuki for the note \cite{MochizukiNote} that he sent us,
containing his interpretation of our result. His note gives a proof
that is uniform with respect to $\lambda$.

\smallskip

\noindent
During the preparation of this work Ron Donagi was supported in part
by NSF grants DMS 1304962 and DMS 1603526 and by Simons Foundation
grant \# 390287. Tony Pantev was supported in part by NSF grant DMS
1302242 and by Simons Foundation grant \# 347070.  Carlos Simpson was
supported in part by ANR grant ANR 933R03/13ANR002SRAR
(Tofigrou). This work enters into the context of the project supported
by ANR grant ``Hodgefun'' and by Simons grant \# 390287. Carlos
Simpson would like to thank the University of Miami for hospitality
during the completion of this work. This work was presented at the
conference ``Algebraic geometry and integrable systems---Kobe 2016''
thanks to the support of JSPS Grant-in-aid (S)24224001.

\vspace{2cm}

\section{Preliminaries} \label{prelim}

This section is preparatory, and includes a review of some of the
relevant issues from non-abelian Hodge theory.  We start by describing
our geometric situation, and then give the main definitions: harmonic
bundle (wild, tame, tame with trivial filtrations),
$\lambda$-connection, minimal $\srR _{\srX}$-module.

\subsection{The underlying geometry} \label{geom}

Suppose we are given a smooth projective surface $X$ with a map
$f:X\rightarrow Y$ to a smooth projective curve.  Suppose we are given
a (reduced) simple normal crossings divisor $D\subset X$. Suppose
given a reduced divisor $Q\subset Y$ so $Q$ is just a finite
collection of points $q_i$.  We assume that $D$ decomposes as
$$
D=D_V + D_H
$$
into two simple normal crossings divisors meeting transversally,
called the {\em vertical} and {\em horizontal} divisors
respectively. We assume that
$$
D_V = f^{-1}(Q)
$$
as a divisor (that is to say, the fibers of $f$ over the points
$q_k$ are reduced with simple normal crossings), and that $D_H$ is
etale over $Y$ (so it is a disjoint union of smooth components, not
intersecting each other but they can intersect $D_V$).

When we would like to consider an irreducible component of a vertical
divisor we write $D_{v(i)}$, and when we would like to consider an
irreducible component of a horizontal divisor we write $D_{h(j)}$.  We
use $D_k$ to refer to any one of these divisors.  Each $D_k$ is smooth
and irreducible:
$$
D_V=\sum _{i=1}^{n_v} D_{v(i)},
$$
$$
D_H=\sum _{j=1}^{n_h} D_{h(j)},
$$
so that
$$
D=\sum _i D_{v(i)} + \sum _j D_{h(j)} = \sum _k D_k.
$$
We assume that $f$ is smooth away from $D_V$, so the fibers
$f^{-1}(q_l)$ are the only singular fibers.  Each one of these fibers
will, in general, be composed of several irreducible components
$D_{v(i)}$. However, since we will rather quickly reduce to the
consideration of a neighborhood of only one $q_i$, we don't need to
introduce additional notation for that.

It follows from our etaleness assumption that $D_H$ is entirely
contained in the smooth locus of $f$.

We next consider the ``$\lambda$-line''
$$
\srA := \aaaa ^1
$$
with coordinate denoted $\lambda$. This is the base of the
Deligne moduli space of $\lambda$-connections which is the
first coordinate chart of the twistor $\pp^1$.

Denote by script letters the spaces
$$
\srX := X\times \srA , \;\;\; \srY := Y\times \srA ,
$$
with their divisors denoted
$$
\srD = \srD _V + \srD _H = D_V \times \srA + D_H\times \srA
$$
and $\srQ = Q\times \srA$.  The corresponding map is
denoted
$$
f_{\srX} : \srX \rightarrow \srY
$$
or just by $f$ if there is no confusion.

Similarly bundles over $\srX$ or $\srY$ will be denoted by script
letters when possible. Their fibers over $\lambda=0$ will be denoted
by regular letters and the fiber over an arbitrary $\lambda$ will be
denoted functionally. Thus if $\srE$ is a bundle on $\srX$ we obtain
bundles $E (\lambda )$ on $X(\lambda )=X$, and $E:= E(0)$.

We have the sheaf of rings $\srR _{\srX}$ over $\srX$ defined by doing
the Rees construction to the filtration of $\Dd _X$ by order of
differential operators (see Subsection \ref{sec-r} below). Denote its
restriction to fibers by $R_X(\lambda )$, so we have
$$
R_X:= R_X(0) = {\rm Sym}^{\ast}(T_X)
$$
whereas for any $\lambda \neq 0$,
$$
R_X(\lambda ) \cong \Dd _X.
$$

\subsection{Harmonic bundles}
\label{sec-hb}

In this subsection we review the non-abelian Hodge correspondence in
both the compact and non-compact cases. We also explain the Nilpotence
Hypothesis \eqref{mainhyp} on the residues of the Higgs field that
will be assumed throughout.

Recall that a {\em harmonic bundle} over $X-D$ 
consists of the data 
$(\Ll , \scD ',\scD '',h)$ where $\Ll$ is a $\Cc ^{\infty}$ bundle
over $X-D$, with a hermitian metric $h$,
having operators
$$
\scD '=\partial + \overline{\varphi} , \; \scD '' =\delbar + \varphi: 
\; \Ll \rightarrow  {\mathcal A}^1(\Ll )
$$
such that
$\partial$ and $\varphi$ are of type $(1,0)$ and 
$\delbar , \overline{\varphi}$ are of type $(0,1)$. 
Put $\scD := \scD ' + \scD''$. These
are subject to the following conditions: 

\begin{enumerate}
\item[1.] $\partial + \delbar$ is an $h$-unitary connection;
\item[2.] $\varphi + \overline{\varphi}$ is $h$-self-adjoint; 
\item[3.] $(\scD '')^2=0$ so that $E^o=(\Ll , \delbar )$ is a holomorphic 
bundle\footnote{In order to avoid confusion we sometimes use a 
superscript $(\, )^o$ to
denote objects such as $E^o$ on the complement $X-D$.}
  and $\varphi : E^o\rightarrow E^o\otimes \Omega ^1_{X-D}$ is a holomorphic Higgs
  field; and 
\item[4.] $\scD ^2=0$,
  so that $L:= \Ll ^{\scD}$ is a local system (i.e. a locally constant sheaf of
  finite-dimensional $\cc$-vector spaces).  
\end{enumerate}

In the compact case with empty divisor ($D=0$), recall that the local
system $L$ is semisimple; the Higgs bundle $(E^o,\varphi )$ is
polystable with vanishing rational Chern classes; and for either one
of these kinds of objects there exists a unique harmonic bundle
associated to it. The metric is unique up to rescaling on each direct
factor. This sets up the {\em nonabelian Hodge correspondence} between
Higgs bundles and local systems, pioneered by Hitchin \cite{Hitchin}.

In the noncompact case, that is to say in the presence of a nontrivial
divisor at infinity $D$ that we are assuming has normal crossings,
more needs to be said about the asymptotics of the harmonic bundle
near the divisor.  For dimension $1$ that was the subject of
\cite{SimpsonHBNC}. In the higher dimensional case this discussion was
initially extended over the smooth points of $D$ by Biquard
\cite{Biquard}. The existence of a pluriharmonic metric on a local
system with monodromy eigenvalues of norm $1$, was done by Jost and
Zuo \cite{JostZuoExistence}.  The full correspondence in general was
done by Takuro Mochizuki in a series of works including
\cite{MochizukiPure,MochizukiKobayashiHitchin}, and was later extended
to cover even wild ramifications that we don't consider here.

This is the starting point for our investigation of higher direct
images. We shall make some simplifying assumptions and now review the
basic outlines. Consider a smooth point $p\in D_i$ of one of the
divisor components, and let $z_i$ be the coordinate function defining
$D_i$ near $p$.  Let $\{ r(t)\}_{t\in (0,1)}$ be a ray going towards
$p$, with $|z_i(r(t))|=t$. If $\{ u(t)\in \Ll _{r(t)}\}$ is a flat
section of the local system $L$ over the ray, we can look at the
growth rate of $\| u(t)\| _{h(r(t))}$ with respect to the harmonic
metric. We say that $u$ has {\em polynomial growth} (respectively,
{\em sub-polynomial growth}) along the ray, if for some (respectively
all) $b>0$ we have:
$$
\| u(t)\| _{h(r(t))} \leq C t^{-b}.
$$
The harmonic bundle is said to be {\em tame} if all its flat
sections have polynomial growth along rays. It is said to be {\em tame
  with trivial filtrations} if flat sections have sub-polynomial
growth along rays.

Tameness is equivalent to the condition that the eigen-one-forms of
the Higgs field $\varphi$ are multivalued sections of the logarithmic
cotangent bundle. Triviality of the filtrations on the local system is
then equivalent to the condition that the limiting values of the
eigen-one-forms (these will be the residues of the logarithmic Higgs
field) are purely imaginary.

Suppose $(\Ll , \scD ',\scD '',\varphi )$ is a tame harmonic bundle.
Using the order of growth defines a collection of filtrations on the
restrictions of the local system $L$ to punctured neighborhoods of
each of the divisor components. Let $j:X-D\hookrightarrow X$. If $\eta
= \{ \eta _i\}$ is a parabolic level, that is to say $\eta _i\in \rr$
for each $i$ corresponding to a divisor component $D_i$, then
$L_{\eta}$ is the subsheaf of $j_{\ast}(L)$ consisting of sections
that have growth rate $\leq C t^{-\eta _i-\epsilon}$ for any $\epsilon
>0$, along rays going towards smooth points of $D_i$. The collection
$\underline{L}=\{ L_{\eta}\}$ is the filtered local system associated
to $(\Ll , \scD ',\scD '',\varphi )$. The condition of trivial
filtrations is equivalent to $L_0=j_{\ast}(L)$ and $L_{\eta}=j_!(L)$
for $\eta _i<0$.

Similarly, the Higgs bundle $(E^o,\varphi )$ extends to a parabolic
sheaf $(\underline{E},\varphi )$ defined as follows. For a parabolic
level $\beta = \{ \beta _i\}$ with $\beta _i\in \rr$, let
$E_{\beta}\subset j_{\ast}(E^o)$ be the subsheaf of holomorphic
sections that locally near smooth points of $D_i$ have growth rate
$\leq C |z_i|^{-\beta _i-\epsilon}$ for any $\epsilon >0$.

Our convention is that the parabolic levels (and weights) are indexed
by $\beta = (\ldots , \beta _k , \ldots )$ with one $\beta_{k} \in
\mathbb{R}$ for each divisor component $D_{k}$.  The parabolic
filtrations are increasing: $E _{\alpha} \subseteq E _{\beta}$ if
$\alpha _i \leq \beta _i$, and if we let $\delta ^k$ denote the index
with all values $0$ except for $1$ at position $k$, then
$$
E  _{\beta + \delta ^k} =
E  _{\beta} (\srD _k).
$$
In our notation for the parabolic bundle, we don't underline when
there is a subscript: a notation of the form $E _{\beta}$ with a
subscript denotes one of the bundles in the structure of the parabolic
bundle.

The following collects some of Mochizuki's main results:

\begin{theorem}[Non-compact non-abelian Hodge correspondence]
\label{hbthm}
If $(\Ll , \scD ',\scD '',\varphi )$ is a tame harmonic bundle, then
the filtered local system is locally abelian: it is locally an
extension of standard rank $1$ filtered local systems. The parabolic
sheaf $\underline{E}$ is a locally abelian parabolic bundle, and the
Higgs field $\varphi$ is logarithmic in the sense that for any
parabolic weight $\beta$ we have
$$
\varphi : E_{\beta}\rightarrow 
E_{\beta}\otimes \Omega ^1_X(\log D).
$$
The filtered local system, and the parabolic logarithmic Higgs
bundle, are both polystable objects with vanishing Chern
classes. Furthermore, any polystable filtered local system or
parabolic logarithmic Higgs bundle with vanishing Chern classes comes
from an essentially unique harmonic bundle (the metric is unique up to
scaling on each irreducible direct factor).  This sets up one to one
correspondences between the three kinds of objects.
\end{theorem}
\begin{proof}
See \cite{MochizukiKobayashiHitchin}. 
\end{proof}

\

We recall here the correspondence between residual data for filtered
local systems and parabolic logarithmic Higgs bundles.  This was given
in the table of \cite[p. 720]{SimpsonHBNC} for the case of curves, and
it works the same way at smooth points of $D$ in higher
dimensions. Notice however that we are using the convention that
filtrations are increasing, so there is a sign change: our $\eta _i$
and $\beta_i$ have opposite signs with respect to \cite{SimpsonHBNC}.

Denote by $m_i\in \cc^{\ast}$ an eigenvalue of the monodromy
transformation of $L$ around $D_i$, in the $\eta _i$-graded piece for
the filtration. This generalized eigenspace corresponds to a
generalized eigenspace of the same dimension, with eigenvalue $r_i\in
\cc$, for the residue of the Higgs field $\varphi$ acting on the
$\beta _i$-graded piece of the parabolic structure along $D_i$. The
correspondence is:

\vspace{.5cm}

\begin{center}
\begin{tabular}{|| l || c |  c || }
\hline \hline
$ { } $ & Local system $(\lambda=1)$ & Higgs $(\lambda=0)$\\
\hline \hline
Grade & $\eta _i= 2{\rm Re}(r_i)$  
\rule{0pt}{12pt}& $\beta _i= {\rm Arg}(m_i)/2\pi $ \\
\hline
Eigenvalue &  $m_i = e^{4\pi {\rm Im}(r_i)+2\pi \sqrt{-1}\beta_i} $  & $r_i = \eta _i /2 + 
\frac{\sqrt{-1}}{4\pi}\log |m_i| $\rule{0pt}{13pt} \\

\hline \hline
\end{tabular}
\end{center}

\vspace{.5cm}

The following hypothesis, in effect throughout the paper,
will considerably simplify many parts of the discussion. 
We note that local systems coming from geometry certainly
satisfy this hypothesis, and it leaves a wide lattitude for
the choice of local system.

\begin{hypothesis}[Nilpotence]
\label{mainhyp}
Throughout the present paper, we shall assume that the filtered local
system has trivial filtrations ($\eta_i=0$) and the eigenvalues of the
monodromy are complex numbers of norm $1$ ($m_i\in S^1\subset
\cc^{\ast}$).  This corresponds for the parabolic logarithmic Higgs
bundle to the hypothesis that the residue of the Higgs field has
trivial eigenvalues ($r_i=0$). The correspondence becomes:

\vspace{.5cm}

\begin{center}
\begin{tabular}{|| l || c |  c || }
\hline \hline
$ { } $ & Local system $(\lambda=1)$ & Higgs $(\lambda=0)$\\
\hline \hline
Grade & $\eta _i=0$  & 
\rule{0pt}{12pt}$\beta _i= {\rm Arg}(m_i)/2\pi $ \\
\hline
Eigenvalue &  $m_i = e^{2\pi \sqrt{-1}\beta_i} $ 
\rule{0pt}{12pt} & $r_i = 0  $ \\

\hline \hline
\end{tabular}
\end{center}

\vspace{.5cm}

\end{hypothesis}

\subsection{The $\lambda$-connections}
\label{sec-lam}

One can extend the construction of Theorem \ref{hbthm}
to obtain a twistor family of parabolic logarithmic $\lambda$-connections, 
interpolating between the connection and the Higgs field. 

Consider a harmonic bundle $(\Ll ,\scD ',\scD '',h)$. 

Our Nilpotence Hypothesis \ref{mainhyp} is in effect; this makes it so that the whole family of parabolic logarithmic $\lambda$-connections associated to the harmonic bundle, has a collection of residual data that varies in a nice way. Complications such as were illustrated in \cite[Figure 3.A]{Sabbah}, leading for example to the need for choosing small open neighborhoods in the $\lambda$-line, are avoided.  

For any $\lambda \in \srA$, we have a vector 
bundle
with 
$\lambda$-connection $(E^o(\lambda ), \nabla (\lambda ))$
on $X-D$ where 
$$
E^o(\lambda )=(\Ll , \delbar + \lambda \overline{\varphi}),
\;\; \nabla (\lambda ) = \lambda \partial + \varphi .
$$
At $\lambda = 0$ this specializes to the Higgs bundle $(E^o(0),\nabla (0))=(E^o,\varphi )$. 
These bundles fit together to form a
holomorphic vector bundle $\srE ^o(-)$ over $\srX -\srD = (X-D)\times \aaaa ^1$, with a relative $\lambda$-connection $\nabla (-)$. Denote again by $j:\srX -\srD \hookrightarrow \srX$ the inclusion.

\begin{theorem}
\label{hblambda}
Under the Nilpotence Hypothesis \ref{mainhyp}, if we set $\srE _{\beta}\subset j_{\ast}(E^o(-))$ to be the subsheaf of sections that have order of growth near $\srD _i$ 
bounded by $C|z_i|^{-\beta _i-\epsilon}$ for any $\epsilon >0$,
then these fit together into a locally abelian parabolic bundle $\underline{\srE}=\{ \srE_{\beta}\}$ and $\nabla (-)$ extends to a relative parabolic logarithmic $\lambda$-connection 
$$
\nabla : \srE _{\beta} \rightarrow \srE_{\beta}
\otimes _{\Oo _{\srX}}
\Omega ^1_{\srX /\srA} (\log \srD ) .
$$
The fiber over $\lambda = 1$ corresponds to the 
local system $L$.
\end{theorem}
\begin{proof}
See \cite{MochizukiPure} (Proposition 1.6 and many others) 
for the
basic regularity properties of the filtered holomorphic bundles.  See \cite{MochizukiKH1} 
and \cite{MochizukiKobayashiHitchin} for the
Kobayashi-Hitchin correspondences between Higgs and flat bundles. 
\end{proof}

The Nilpotence Hypothesis \ref{mainhyp} tells us, in this context,
that the eigenvalues of the residue of the connection $\nabla (\lambda
)$ on the $\beta _i$-graded piece of the parabolic bundle are equal to
$\lambda \beta _i$. This formula is one of the things that is much
simplified by our Nilpotence Hypothesis.

(That could be seen by calculating directly for rank one
systems, or one may refer to \cite[Lemma 3.3.4(7)]{Sabbah}
as discussed in the remarks on page \pageref{genlam} below.)

As was the case for $\underline{E}$, the parabolic levels are indexed
by $\beta = (\ldots , \beta _k , \ldots )$ with one level for each
divisor component.  The parabolic filtrations are increasing: $\srE
_{\alpha} \subseteq \srE _{\beta}$ if $\alpha _i \leq \beta _i$, and
$$
\srE  _{\beta + \delta ^k} =
\srE  _{\beta} (\srD _k).
$$

\subsection{The $\srR$-module $\srE$}
\label{sec-r}

A central theme of this work will be the comparison of $\srR
_{\srX}$-modules and harmonic bundles. In this subsection we define
the relevant $\srR _{\srX}$-modules, especially the minimal extension.

Mochizuki shows in \cite{MochizukiPure} that, starting from a harmonic
bundle, we obtain a pure twistor $\Dd$-module in the sense of Sabbah
\cite{Sabbah}.  This generalizes Saito's construction of a pure Hodge
module out of a variation of Hodge structure \cite[Th\'eor\`eme
  5.4.3]{SaitoMHP}.

Recalling various parts of Sabbah's theory will take up much of our
exposition, and the reader will have to refer to \cite{Sabbah} for
many details. Here is the starting point.

The sheaf of rings $\Dd_X$ of differential operators on $X$ has a
filtration $F^{\bullet}$ by order of the operator: $F^k\Dd _X$
consists of differential operators of order $\leq k$. Thus, $F^0\Dd
_X=\Oo _X$. The Rees construction for the filtration $F^{\bullet}$
produces the sheaf of rings $\srR _{\srX}$ over $\srX =X\times
\srA$. Concretely, consider the sheaf of rings
$$
{\rm pr}_X^{\ast}(\Dd _X)= 
\Dd _{\srX /\srA}
$$
on $\srX$ and let 
$$
\srR _{\srX} := \sum_{i=0}^{\infty} {\lambda}^i {\rm
  pr}^{\ast}_XF^i\Dd_{X}\subset {\rm pr}_X^{\ast}(\Dd _X)
$$
be the subsheaf of $\Oo _{\srX}$-algebras which is 
generated in local coordinates by
$$
\partial _i := \lambda \frac{\partial}{\partial x_i}.
$$

Notice that $F^1\srR_{\srX}$ is still a direct sum $\Oo_{\srX}\oplus
T_{\srX / \srA }$, with the operators $\partial _i$ being a
basis for the relative tangent bundle $T_{\srX / \srA }$.
However, this copy of the tangent bundle acts through derivations on
$\Oo _{\srX}$ only after multiplying by $\lambda$.

According to our convention, we denote the restriction to
fibers of this sheaf of rings by $R_X(\lambda )$, so we have
$$
R_X:= R_X(0) = {\rm Sym}^{\ast}(T_X)
$$
whereas for any $\lambda \neq 0$,
$$
R_X(\lambda ) \cong \Dd _X.
$$

A $\lambda$-connection on a quasicoherent sheaf over $\srX$ is the
same as a structure of $\srR_{\srX}$-module. Our parabolic collection
of sheaves with logarithmic $\lambda$-connections yields a structure
of $\srR _{\srX}$-module on the union
$$
\widetilde{\srE}:= \bigcup _{\beta}\srE _{\beta}.
$$
This is almost, but not quite, the $\srR_{\srX}$-module associated
to our local system in the theory of Saito-Sabbah-Mochizuki.

It is too big, not being finitely generated over $\srR$ because there
is no derivation at $\lambda = 0$ that could increase the pole
order. For this reason, Saito and Sabbah introduce the following
definition, see \cite[Definition 3.4.7]{Sabbah}.

\begin{definition}
\label{minext}
The {\em minimal extension} $\srE$ is the smallest
$\srR_{\srX}$-submodule $\srE \subset \widetilde{\srE}$ containing the
$\srE_{\beta}$ for $\beta _i<1$.
\end{definition}

We recall that a notation of the form $\srE _{\beta}$ with a subscript
denotes one of the bundles in the structure of the parabolic bundle,
whereas $\srE$ without subscript is the $\srR_{\srX}$-module.

The union of parabolic components $\widetilde{\srE}$ that we defined
first, is the localization of $\srE$ along $\srD$, obtained by
tensoring with functions with poles along $\srD$.

\bigskip

\noindent
{\em Caution:} If $\beta _i<1$ then $\srE _{\beta} \subset \srE$.  But
if some $\beta _i \geq 1$ then $\srE _{\beta}$ is not necessarily
contained in $\srE$ but only in $\widetilde{\srE}$.  The $\srE
_{\beta}\cap \srE$ are characterized as the submodules obtained by
applying the appropriate number of derivatives to $\srE _{\beta '}$
with $\beta '_i<1$.

\bigskip

Note that for a given $\widetilde{\srE}$ there can be several possible
holonomic sub-$\srR _{\srX}$-modules $\srE \subset \widetilde{\srE}$
giving the same localization.  One first restricts to submodules that
are {\em strictly specializable} \cite[Definitions 3.3.8, Proposition
  3.3.11]{Sabbah}. However, these can have subobjects or quotients
supported on the divisor.

\begin{remark}
\label{s-term}
Recall that Sabbah defines the notion of
{\em strictly S-decomposable $\srR$-module}
\cite[Definition 3.5.1]{Sabbah}. 
This is a strictly specializable one satisfying the
conditions of \cite[3.11(e)]{Sabbah} (that we don't recopy here) about the morphisms
${\rm can}$ and ${\rm var}$.  Then $\srE$ has {\em strict
supports} if it is strictly S-decomposable and has no submodules supported in smaller dimension \cite[Definition 3.5.3]{Sabbah}. Now, the minimal extension may be characterized as the
unique extension with strict supports. 

In the other direction, the characterization of \cite[Corollary 3.5.5]{Sabbah} says that an $\srR$-module is strictly S-decomposable if and only if it is a direct sum of minimal extensions of smooth ones on locally closed subvarieties. 
\end{remark}

We will not need to calculate with the formulas for minimal extension;
these would come into the discussion of our quasiisomorphism along
horizontal divisors, but Sabbah has already done what we need and we
will just be quoting Sabbah's calculation in Proposition
\ref{horizontaldiv} below. Along the vertical divisors the localized
module $\widetilde{\srE}$ is just as well adapted (if not sometimes
better) for consideration of the $V$-filtrations.

Recall that the fibers over $\lambda =0$ are to be denoted by regular
letters: thus we get an $R_X$-module $E$ and a parabolic Higgs bundle
$\underline{E}$ on $X$.

One of our main concerns will be the relationship between the
parabolic Higgs bundle (resp. parabolic bundle with
$\lambda$-connection) and the $R_X$-module
(resp. $\srR_{\srX}$-module):

\vspace{.5cm}

\begin{center}
\begin{tabular}{|| l || c |  c || }
\hline \hline
 $ { } $ &  $\lambda=0$ & all $\lambda$\\
\hline \hline
Parabolic bundle (section \ref{sec-hb}) &
$\underline{E} =\{ E_{\beta}\}$  & $\underline{\srE}=\{ \srE_{\beta}\}$  \\
\hline
$R_X / \srR_{\srX}$-module  (section \ref{sec-r})&  $E$   & $\srE$ \\

\hline \hline
\end{tabular}
\end{center}

\vspace{.75cm}

\noindent
{\em Remarks:}
\label{genlam}
With the Nilpotence Hypothesis \ref{mainhyp} in effect, the residual
data vary in a nice way in the family of $\lambda$-connections
associated to our $\srR_{\srX}$-module $\srE$.  In the fiber over an
arbitrary $\lambda$, the bundle $\srE _{\beta}(\lambda )$ has a
parabolic $\lambda$-connection. Along a smooth point of some divisor
$D_k$, we obtain the associated-graded ${\rm Gr}_{k,\beta _k}(\srE
_{\beta} (\lambda ))$ with the residue endomorphism of the
$\lambda$-connection.  Hypothesis \ref{mainhyp} implies that the
eigenvalues of the residue endomorphism on the $\beta _k$-graded piece
are of the form $\lambda \beta _k$.

To put this fact into perspective, the reader may want to look at the
notations in \cite[p. 17]{Sabbah}. Here is a short guide to the
correspondence of notations: what we are denoting $\lambda$ is denoted
there by $z$. The complex numbers $\alpha$ indexing the pieces of the
nearby cycles functors are real in our case, indeed $\alpha =\beta _k$
for the $\beta_k$-th graded piece we are looking at. If $\alpha$ is
real then $\alpha \star z=z\alpha$ in the notation of
\cite[p. 17]{Sabbah}, corresponding to $\lambda \beta _k$ for
us. Similarly, $\ell _{z_0}(\alpha )= \alpha$.

In \cite[Lemma 3.3.4(6)]{Sabbah} with our hypothesis that $\alpha$ is
real, $\ell _{z_0}(\alpha )=a$ is equivalent to $\alpha =a$ and the
${\rm gr}^{V}_a$ is the same as just $\psi _{t,\alpha}$. This gives a
simple identification between the nearby cycles functor and the
associated-graded of the $V$-filtration that we'll consider below.

In view of the simplification $\alpha \star z = \lambda \beta _k$, the
formula of \cite[Lemma 3.3.4(7)]{Sabbah} says that the eigenvalue of
the residue of our $\lambda$-connection on the $\beta_k$ graded piece
of the parabolic filtration, is $\lambda \beta _k$.

\section{Main setup and results} \label{setup}

In this section we state the main theorem and make some preliminary
constructions and reductions. We start in section \ref{sec-locstudy}
by constructing the monodromy weight filtration \linebreak $W({\rm Gr}_{k,b}
(\underline{E})):= W(N_{k,b})$ associated to a parabolic Higgs bundle
$(\underline{E},\varphi)$. In section \ref{L2dol}, these weight
filtrations are used to construct the $L^2$ parabolic Dolbeault
complex, in equation \eqref{l2pardol}. This allows us, in section
\ref{Main}, to state the Main Theorem \ref{mainth}.

\subsection{Local study of the parabolic Higgs bundle} \label{sec-locstudy}

Let $(\underline{E},\varphi)$ be a parabolic Higgs bundle satisfying
the Nilpotence Hypothesis \ref{mainhyp}. Along a component $D_k$ of
the parabolic divisor $D$, our $\underline{E}$ determines a family
${\rm Gr}_{k,b}(\underline{E})$, indexed by real $b$, of parabolic
vector bundles on $D_k$. Their parabolic structure is along the
divisor $D_{\cap k}:= \overline{(D-D_k)}\cap D_k$ where $D_k$ meets
the other components. These come with endomorphisms $N_{k,b}$ induced
by the residues of the Higgs field $\varphi$. In this subsection we
discuss the monodromy weight filtration $W({\rm Gr}_{k,b}
(\underline{E})):= W(N_{k,b})$. Subsequently, these will be used to
construct the $L^2$ parabolic Dolbeault complex.

Consider one of the components $D_k$ of the divisor.  Let $D_{\cap
  k}:= \overline{(D-D_k)}\cap D_k$ be the divisor on $D_k$ induced by
the other components of $D$, and put $D_k^{\ast}:= D_k-D_{\cap k}$.

For a real number $b$, let ${\rm Gr}_{k,b}(\underline{E})$ be the
parabolic bundle on $D_k$, with respect to the divisor $D_{\cap k}$,
defined as follows. The indexing of its parabolic structure is by the
irreducible components, which in our case of $1$-dimensional $D_k$ are
just points in the zero-dimensional divisor $p\in D_{\cap k}$. Note
that for each such point there is an index $j\neq k$ such that $p\in
D_{jk}=D_j\cap D_k$.  Given a parabolic weight vector $\alpha$ for
these indices, then we obtain a bundle $E^{(D_k)}_{\alpha , b}$ in a
neighborhood of $D_k$ (say, a tubular neighborhood in the usual
topology). Near a point $p\in D_{jk}$, the divisor $D_j$ contains a
piece\footnote{Notice here that there could in principle be several
  different points $p\in D_{jk}$; but the local pieces of $D_j$
  intersected with the tubular neighborhood are disjoint so they can
  be assigned different weights. Alternately one could assume by
  further blow-up that two divisor components intersect in at most a
  single point.}  transverse to $D_k$ at $p$ and we can use the
parabolic structure of $\underline{E}$ with weight $\alpha _p$ for
such a piece of $D_j$, and weight $b$ along $D_k$.

Define
$$
{\rm Gr}_{k,b}(\underline{E})_{\alpha}:=
E^{(D_k)}_{\alpha , b}/
E^{(D_k)}_{\alpha , b-\epsilon}.
$$
Assuming that the original parabolic bundle was locally abelian
(i.e. locally a direct sum of parabolic line bundles), then ${\rm
  Gr}_{k,b}(\underline{E})$ will be a locally abelian parabolic bundle
on $D_k$ with respect to the divisor $D_{\cap k}$.

The Higgs field
$$
\varphi :
E_{\beta}\rightarrow
E_{\beta}\otimes \Omega ^1_X(\log D)
$$
induces a map
$$
N_{k,b}:={\rm res}_{b}(\varphi ) : {\rm Gr}_{k,b}(
\underline{E})\rightarrow
{\rm Gr}_{k,b}(\underline{E}).
$$
It is a map of parabolic bundles on $D_k$ since $\varphi$ respects
the parabolic structure of $\underline{E}$.  Our assumption
\ref{mainhyp} is that this residue is nilpotent.

Along smooth points of $D_k$ (which we assume irreducible) we have the
following basic fact, a special case of the more general result of
Mochizuki that we'll refer to below, but in this case it goes back to
Biquard \cite{Biquard}.

\begin{lemma}
\label{const1}
If $p,p'\in D^{\ast}_k$
then $N_{k,b}(p)$ and $N_{k,b}(p')$ are
conjugate as nilpotent endomorphisms of a vector space.
\end{lemma}

\begin{corollary}
\label{weightcor}
Over $D_k^{\ast}$ for any real number $b$, there is a weight
filtration $W(N_{k,b})$ of the vector bundle ${\rm
  Gr}_{k,b}(\underline{E})$ with respect to the endomorphism
$N_{k,b}$, a filtration by strict subbundles such that the nilpotent
endomorphism gives isomorphisms of bundles in the usual way, in
particular the restriction of this filtration to any point $p$ is the
weight filtration of $N_{k,b}(p)$.
\end{corollary}
\begin{proof}
We have a vector bundle $V:={\rm Gr}_{k,b}(\underline{E})$ over a
variety $D_k^{\ast}$, and a nilpotent endomorphism $N \in {\rm}End(V)$
whose values at all points of $D_k$ are conjugate to each other by
Lemma \ref{const1}. The claim is that the vector 
spaces $W_{\ell}$ arising
from the monodromy weight filtration of $N$ form vector subbundles of
$V$.

Without loss of generality, we may assume (by restricting to an open
subset) that $V$ is trivial, $V = V_0 \otimes {\mathcal
  O}_{D_k^{\ast}}$.  Then $N$ gives a map $X \to G:={\rm GL}(V)$.  Our
assumption is that the image of this map is contained in a single
(nilpotent) orbit $O \cong G/S \subset G$, where $S$ is the
stabilizer.  This means that we may as well replace: $D_k^{\ast}$ by
$O$, $V$ by $V_0 \otimes {\mathcal O}_O$, and $N$ by (the restriction
from $G$ to $O$ of) the tautological endomorphism.  The group $G$ now
acts transitively on $O=G/S$ preserving $V, N$, so each of the $W_{\ell}$
is now a homogeneous vector bundle on $O$.  The original 
$W_{\ell}$ on
$D_k^{\ast}$ are the pullbacks, so they still form vector bundles.
\end{proof}

This property extends to the normal crossings points too.  Suppose
$p\in D_{\cap k}$ is one of the intersection points with some other
$D_j$ so $p\in D_{jk}$.  Fix $b, \alpha$ and consider the vector
bundle
$$
V:= {\rm Gr}_{k,b}(\underline{E})_{\alpha}
$$
over $D_k$.

This bundle has a parabolic filtration at the point $p$, that is to
say we have a filtration $0=F_{-m}\subset \cdots \subset F_0=V(p)$ of
the fiber at $p$, and there are parabolic weights attached to the
pieces.  Call them $a _{-m}< \cdots <a _0$, and we attach the weight
$a _{-i}$ to the filtration element $F_{-i}/F_{-i-1}$.  Notice that $a
_{-m}$ doesn't actually occur here but it satisfies $a _{-m}=a
_0-1$. These weights come from the parabolic weights of
$\underline{E}$ along $D_j$. Precisely, the top weight is $a _0=\alpha
_j$, the weight attached to $D_j$ in our global weight $\alpha$. Then
the other ones are some of the parabolic weights attached to $D_j$
(they don't include ones for which the graded piece for $\beta _k=b$
along $D_k$ vanishes).

The bundle $V$ also has the nilpotent endomorphism $N:V \to V$, and
$N$ preserves the filtration, because of the condition that $\varphi$
acts on the full parabolic bundle $E$.

Mochizuki shows the following version of Lemma \ref{const1}
at the crossing points:

\begin{proposition}
\label{crossingpointsconstancy}
Let $Gr^F(V(p)):= \bigoplus _{i=0}^{m-1} F_{-i}/F_{-i-1}$.  Then the
pair consisting of $Gr^F(V(p))$ together with the induced endomorphism
$Gr^F(N(p))$ on this graded vector space, has the same isomorphism
type as any $(V(p'),N(p'))$ for $p'\in D_k^{\ast}$ a smooth point of
the divisor $D$.
\end{proposition}
\begin{proof}
This is \cite[Lemma 12.34]{MochizukiPure}. The basic idea is that for
$\lambda \neq 0$ the independence of the point is due to the fact that
the module with connection splits into a direct sum of pieces
according to the eigenvalues of the monodromy. Then, the main result
\cite[Lemma 12.33]{MochizukiPure}, due to strictness of the nilpotent
map with respect to a limiting mixed twistor structure, says that the
conjugacy classes are independent of $\lambda$. The independence of
the conjugacy class as a function of $p$ at $\lambda =0$ then follows.
\end{proof}

We need to strengthen this somewhat:

\begin{lemma}
\label{grconstancy}
In the above situation, 
$( V(p), N(p) )$ is also isomorphic to $( V(p'), N(p') )$.
\end{lemma}
\begin{proof}
Inclusion of orbit closures induces a partial ordering on the set of
nilpotent conjugacy classes in a given finite dimensional vector
space.  On the one hand, $( V(p), N(p) )$ is a limit of the family of
$( V(p'), N(p') )$ as $p' \to p$, and by Lemma \ref{const1}, the
conjugacy class of $( V(p'), N(p') )$ is independent of $p'$, so $(
V(p), N(p) )$ is in the orbit closure of $( V(p'), N(p') )$ for a
fixed $p'$.  On the other hand, $( Gr^F(V(p)) , Gr^F (N(p)) )$ is in
the orbit closure of $( V(p), N(p) )$, as is seen by the Rees
construction.  But by Proposition \ref{crossingpointsconstancy}, $(
Gr^F(V(p)) , Gr^F (N(p)) )$ is isomorphic, in the category of vector
spaces with nilpotent endomorphism, to $( V(p'), N(p') )$ for any $p'
\in D_k^{\ast}$.  So $( V(p'), N(p') )$, $( V(p), N(p) )$ are in each
other's orbit closure, and therefore they are conjugate.
\end{proof}

\begin{corollary}
The weight filtration extends as a strict filtration on the bundle $V$
over all of $D_k$. Furthermore, this weight filtration induces on
$Gr^F(V(p))$ the weight filtration of $Gr^F(N(p))$.
\end{corollary}
\begin{proof}
By Lemma \ref{const1} away from the intersection points, and Lemma
\ref{grconstancy} at the intersection points, the conjugacy classes of
$N(p)$ acting on $V(p)$ are the same for all points $p\in D_k$.  The
argument of Corollary \ref{weightcor} gives a weight filtration by
strict vector subbundles of $V$, defined all along $D_k$. It is
clearly the weight filtration of $N(p')$ for smooth points $p'\in
D_k^{\ast}$.

At an intersection point $p$, the filtration is also by construction
the weight filtration of $N(p)$.  In order to show that it induces the
weight filtration of $Gr^F(N(p))$, consider the Rees construction for
the filtration $F$.  This gives an analogous situation of a family
indexed by a curve (the affine line), in which we know that the
special and general points have the same isomorphism type. Applying
Corollary \ref{weightcor} we get a canonical extension of the weight
filtration along the Rees line. Since it is canonical it is
$\Gm$-invariant so the filtration on the special fiber is the
associated-graded for $F$ of the filtration on the general fiber, that
is to say the filtration on the special fiber is the one induced on
$Gr^F(V(p))$. We get that this induced filtration is the weight
filtration of $Gr^F(N(p))$.
\end{proof}

As a result, we obtain weight filtrations denoted  
$W({\rm Gr}_{k,b} (\underline{E})):= W(N_{k,b})$ 
of the parabolic vector bundles
${\rm Gr}_{k,b} (\underline{E})$ over $D_k$, with parabolic structure
along $D_{\cap k}$. These are filtrations by strict parabolic subbundles.

\subsection{The $L^2$ parabolic Dolbeault complex}
\label{L2dol}

In this subsection we use our weight filtrations to define an $L^2$
parabolic Dolbeault complex \eqref{l2pardol} coming from our parabolic
logarithmic Higgs bundle.  It plays a central role in our main result,
theorem \ref{mainth}.  As we point out below, one can similarly define
a de Rham complex for the parabolic logarithmic $\lambda$-connection.

The definition will be algebraic, involving the weight filtrations we
constructed in subsection \ref{sec-locstudy} along the horizontal
divisor components. But let us start by considering the analytic
motivation.

Consider an open fiber $X^o_y=(X-D)_y$ for $y\in Y-Q$. Give it a
metric that is asymptotically the Poincar\'e metric near puncture
points. We are interested in the $L^2$ cohomology of the harmonic
bundle $(\Ll , \scD ',\scD '', h)$ restricted to this fiber, as shall
be discussed in more detail in Subsection \ref{ancon}
later. Basically, this is the cohomology of the complex of forms with
coefficients in $\Ll$ that are $L^2$, and whose derivative is
$L^2$. There are both de Rham cohomologies with differential $\scD$,
and Dolbeault cohomologies with differential $\scD''$.

In the case of coefficients in a variation of Hodge structure, these
cohomology spaces were considered by Zucker \cite{Zucker} who proved
that they are finite-dimensional and that the de Rham and Dolbeault
cohomologies are isomorphic, both being isomorphic to the space of
harmonic forms (recall from the K\"ahler identities that $\Delta
_{\scD} = 2\Delta _{\scD''}$ and this holds for harmonic bundles
too). But in fact, Zucker's theory also applies to tame harmonic
bundles satisfying the Nilpotence Hypothesis \ref{mainhyp}, indeed
locally near a puncture point they are asymptotically the same as the
standard local models for variations of Hodge structure
\cite{SimpsonHBNC}, so the estimates needed by Zucker hold. The theory
is treated in detail by Sabbah \cite[Section 6.2]{Sabbah}.

For our current purposes, we would like to have a calculation of the
cohomology space in algebraic terms using the parabolic Higgs
bundle. To get there, the main fact is that the $L^2$ Dolbeault
cohomology is isomorphic to the hypercohomology of the complex on
$X_y$ consisting of holomorphic forms whose restriction to $X^o_y$ is
in $L^2$ and whose differential is in $L^2$. So, we would like to
write down the resulting complex in algebraic terms.

First a general estimate tells us that a holomorphic section of
$E|_{X_y^o}$ (resp. $E|_{X_y^o}\otimes \Omega ^1_{X^o_y}$) can be in
$L^2$ only if it extends to a section of the $0$-th component of the
parabolic structure $E_0$ (resp. $E_0\otimes \Omega ^1_{X_y}(\log
D_y)$). Furthermore, if it is in $E_a$ (resp.  $E_a\otimes \Omega
^1_{X_y}(\log D_y)$) for $a<0$ then it is automatically $L^2$.

Suppose $e$ is a section of $(E_0)_y$ near a point $p$ on the
horizontal divisor component $D_{k,y}$; it projects to ${\rm
  gr}_{k,0}(e)\in {\rm Gr}_{k,0}(\underline{E})_y$. Suppose this
projection is in $W_{\ell}$ but not $W_{\ell -1}$. 
Then, denoting by $z$ a
coordinate on $X_y$ vanishing at $p$, the norm of $e$ is
asymptotically
$$
|e|\sim \left| \log |z|\right| ^{\ell /2}.
$$ Calculations with the Poincar\'e metric for the norms of sections
or holomorphic $1$-forms, done in \cite[Proposition 4.4]{Zucker}, tell
us the following. First, $e$ is in $L^2$ if and only if $m\leq 0$.
Then similarly, a section $e \frac{dz}{z}$ of $E_0\otimes \Omega
^1_{X_y}(\log D_y)$ is in $L^2$ if and only if
$$
{\rm gr}_{k,0}(e)\in W_{-2}{\rm Gr}_{k,0}(\underline{E}).
$$ These observations give us an algebraic description of the complex
of holomorphic $L^2$ forms on the fibers $X_y$.

\medskip

Proceed to define a complex on $X/Y$ as follows. 
For any divisor component $D_k$, and any multiindex $\beta$,
put 
$$
{\rm Gr}_{k,\beta _k} (E_{\beta}):= 
{\rm Gr}_{k,\beta _k} (\underline{E})_{\beta (\cap k)}
$$
where $\beta (\cap k)$ consists of the parabolic weights of $\beta$
for the components of $D_{\cap \beta}$.  We have explicitly
$$
{\rm Gr}_{k,\beta _k} (E_{\beta})
= E_{\beta} / E_{\beta -\epsilon \delta _k}
$$
where $\delta _k$ is the multiindex with $1$ in position $k$ and $0$ elsewhere. 

In the previous subsection we defined the weight filtration $W({\rm
  Gr}_{k,\beta_k} (\underline{E}))$ of the parabolic bundle ${\rm
  Gr}_{k,\beta _k} (\underline{E})$ on $D_k$. By assigning parabolic
weights $\beta (\cap k)$ on $D_{\cap k}$ this gives a weight
filtration of the bundle ${\rm Gr}_{k,\beta _k} (E_{\beta})$, and we
call that $W({\rm Gr}_{k,\beta _k} (E_{\beta}))$.

Let us now denote by 
$$
W(k, E_{\beta})\subset
E_{\beta}
$$
the pullback of the weight filtration 
$W({\rm Gr}_{k,\beta _k} (E_{\beta}))$
over $D_k$, to a filtration of $E_{\beta}$ by locally free subsheaves, via the map
$$
E_{\beta}\rightarrow
{\rm Gr}_{k,\beta _k} (E_{\beta}).
$$

Let $W(H,E_{\beta})$ denote the weight filtration obtained by using
$W(h(j),E_{\beta})$ along each horizontal component $D_{h(j)}$. More
precisely, we use the weight filtration as we have defined in the
previous subsection on the parabolic bundle ${\rm Gr}_{k,\beta _k}
(\underline{E})$, and take the resulting weight filtration on the
piece ${\rm Gr}_{k,\beta _k} (E_{\beta}) = {\rm Gr}_{k,\beta _k}
(\underline{E})_{\beta (\cap k)}$ of this parabolic bundle.

For any real number $a$, let $\alpha (a)$ denote the parabolic weight
for the divisor $D$ determined by using weight $a$ along the vertical
components and weight $0$ along the horizontal components. We then
obtain the levels of the horizontal weight filtrations
$$
W_{\ell}(H,E_{\alpha (a)})\subset
E_{\alpha (a)}.
$$
Notice that along the horizontal divisor components $D_{h(j)}$ we
have $\alpha (a)_{h(j)}=0$ so the horizontal weight filtrations come
from filtrations on the parabolic weight zero graded pieces ${\rm
  Gr}_{h(j),0} (E_{\alpha (a)})$.

We may now define our relative  $L^2$ parabolic Dolbeault complex:
\begin{equation}
\label{l2pardol}
{\rm DOL}^{\rm par}_{L^2}(X/Y,E_{\alpha (a)}):=
\xymatrix@R-2pc{
\Big[ W_0(H,E_{\alpha (a)}) 
\ar[r]^-{\varphi} & 
W_{-2}(H,E_{\alpha (a)})
\otimes_{\Oo _X} \Omega ^1_{X/Y}(\log D)\Big], \\
0 & 1 }
\end{equation}
where $\Omega ^1_{X/Y}(\log D) = \Omega^{1}_{X}(\log D)/ f^{*}
\Omega^{1}_{Y}(\log Q)$ is the sheaf of relative logarithmic one forms
along the fibers of $f$. To make the notation less cumbersome we
still write $\varphi$ for the projection of the Higgs field to the
relative logarithmic forms $\Omega ^1_{X/Y}(\log D)$.

\subsection{The Main Theorem}\label{Main}

We can now state the main theorem of this paper.

\begin{theorem}
\label{mainth}
Let
$$
F^i_a:= {\mathbb R}^if_{\ast}
\left(
{\rm DOL}^{\rm par}_{L^2}(X/Y,E_{\alpha (a)}) \right) .
$$
\begin{enumerate}
\item[1.] The 
$F^i_a$ are locally free, and fit together as
$a$ varies into a parabolic bundle $\underline{F}^i$.
\item[2.] Formation of the higher direct images is compatible with
base-change, in other words 
$F^i_a(y)$ is the cohomology
of the fiber over $y\in Y$.
\item[3.] The parabolic bundle $\underline{F}^i$ has a Higgs
  field $\theta$ given by the usual Gauss-Manin construction (section
  \ref{globalcomplex}), making it into a parabolic Higgs bundle.
\item[4.] This parabolic Higgs bundle on $(Y,Q)$ is the one associated to the
middle perversity higher direct image (of degree $i=0,1,2$) of the
local system underlying our original harmonic bundle.
\item[5.] More specifically, over $Y-Q$ the bundle 
$\underline{F}^i$ has a
harmonic metric given by the $L^2$ metric on cohomology classes in the
fibers, and the parabolic Higgs structure is the one associated to
this harmonic metric.
\end{enumerate}
\end{theorem}

Clearly we are interested mostly in the case $i=1$. It is useful to
include $i=0,2$ because these facilitate using an Euler characteristic
argument to control the dimension jumps at $i=1$.

We note here that one can also define a relative version over the
$\lambda$-line, denoted as an $L^2$ de Rham complex:
$$
DR^{\rm par}_{L^2}(\srX / \srY, \srE  _{\alpha (a)}):=
\xymatrix@R-2pc{
\Big[ W_0(H,\srE  _{\alpha (a)}) \ar[r]^-{\nabla} & 
W_{-2}(H,\srE  _{\alpha (a)}) \otimes_{\Oo _{\srX}}
\Omega ^1_{\srX / \srY} (\log \srD ). \Big] \\
0 & 1
}
$$

The restriction to $\lambda =0$ is well-behaved since the complex is
flat over $\srA$ and it gives back the Dolbeault complex
\eqref{l2pardol}:
\begin{equation} \label{l2pardR}
DR^{\rm par}_{L^2}(\srX / \srY,\srE  _{\alpha (a)}) |_{X(0)} =
{\rm DOL}^{\rm par}_{L^2}(X/Y,E_{\alpha (a)}).
\end{equation} 
See \cite[Theorem 6.2.4]{Sabbah}. Part (2) of that statement gives
only a quasiisomorphism between the restriction of the $L^2$ de Rham
complex to $X(0)$, and the $L^2$ Dolbeault complex. The restriction to
$X(0)$ corresponds to the case $z_0=0$ in Sabbah's notation. In our
situation, the restriction map between complexes is an
isomorphism. Indeed, the additional terms that occur in the proof of
\cite[Theorem 6.2.4]{Sabbah} concern the case of non-real values of
the parameter $\alpha$, whereas we are assuming that the eigenvalues
of the monodromy transformations are roots of unity so that the
parameters $\alpha$ that occur are only real.  In the terms for
$\alpha $ real (taking into account the shift of $1$ from Sabbah's
$\alpha =-1$ to our $\alpha =0$), the description of the $L^2$ complex
in terms of the weight filtration coincides with what we have said
above, both before and after restricting to $X(0)$ i.e. $z_0=0$.

The relative de Rham version of our main theorem is
the following statement: 

\begin{theorem}
\label{mainthdr}
Let
$$
\srF ^i_a:= {\mathbb R}^if_{\ast}
\left(
{\rm DR}^{\rm par}_{L^2}(\srX / \srY , \srE _{\alpha (a)}) \right) .
$$
The 
$\srF ^i_a$ are locally free and the higher direct
images are compatible with base-change. These fit together into a
parabolic bundle $\underline{\srF}^i$ with logarithmic
$\lambda$-connection on $\srY$ relative to $\srA$.  It is the
parabolic $\lambda$-connection associated to the harmonic bundle of
Theorem \ref{mainth}.
\end{theorem}

\section{The de Rham complex} \label{sec-dr}

As a step towards the proof of the Main Theorem \ref{mainth}, we
formulate its analogue, Theorem \ref{rthm}, on the level of
$\srR$-modules and their $V$-filtrations.  This result, in the
$\srR_{\srX}$-module context, is already known, by Sabbah's work.  So
this leads us to the comparison problem of relating these two points
of view, in order to deduce our main theorem from Theorem \ref{rthm}.

In a little more detail: we start in subsection \ref{j} with a brief
review of the Kashiwara-Malgrange $V$-filtration of a $\Dd$-module
with respect to a map or a subscheme. While the smooth case is
straightforward, the singular case requires factoring through a graph
construction, as we recall in section \ref{j}. In subsection \ref{dR}
we define the de Rham complex on $\srX$ in terms of this graph
factorization, take its higher direct image to $\srY$, and then state
and prove Theorem \ref{rthm} about compatibility with the
$V$-filtration.

In subsection \ref{multi}, we look briefly at the
notion of multi-$V$-filtration with respect to several
smooth divisors meeting transversally. This structure
gives back our parabolic structures. It will be useful
to have a formulation of the filtered objects involved
in the parabolic structures, in the language of $V$-filtrations. 

Finally, in subsection \ref{constructmap}, we approach the comparison
problem by constructing a map (cf. Lemma \ref{mapcomplexes}) from the
parabolic de Rham complex to the new de Rham complex. The proof of the
Main Theorem \ref{mainth} is then reduced, in Proposition
\ref{proofmainth}, to Theorem \ref{cpxqis} which asserts that the
comparison map is indeed a quasi-isomorphism.

\subsection{The $V$-filtration} \label{j}

Consider first a smooth map $p:P \rightarrow Y$, with $y$ a local
coordinate on $Y$ vanishing at a point $q \in Y$, and $x$ a coordinate
along the fibers, so $x,y$ are local coordinates on $P$. (Or
$x=(x_1,\ldots, x_n)$ and $y=(y_1,\ldots,y_k)$ could be coordinate
systems).  In these local coordinates, the sheaf of rings $\Dd_P$ of
differential operators on $P$ looks like
$$
\Dd_P = \cc[x,y,\partial_x,\partial_y]
$$
The $V$-filtration on $\Dd_P$ with respect to the smooth map $p$, or
with respect to the smooth subvariety 
$Z:=p^{-1}(q) = \{y=0\}$) defined by the ideal $I:=(y)$, is given by:
$$
V_j\Dd_P := \left\{ \xi \in \Dd_{P} \ \left| \ \xi I^{i} \subset I^{i -j}  \text{ for all } i
\in \mathbb{Z}
\right. \right\},
$$
where by convention $I^{k} = \mathcal{O}_{P}$ for all $k \leq 0$. 
In local coordinates, $V_j\Dd_P$   is spanned over $\Oo_P$ by the expressions
$\partial_x^a y^b\partial_y^c$ with $c-b \leq j$. In particular,
$V_0\Dd_P$ is the sheaf of rings of differential operators Koszul dual to the
logarithmic cotangent complex $\Omega ^{\bullet}_X(\log Z)$,
i.e. generated by the log tangent vectors $\partial_x, y\partial_y$.

The $V$-filtration (still with respect to the smooth map $p$ or
subvariety $Z$) on a left-$\Dd_P$-module $M$ is an increasing
filtration of $M$ by coherent $V_0\Dd_P$-submodules $V_a$. This is
subject to several axioms \cite{Bjork,Budur}, the main one being that
on
$$
{\rm Gr}_{a}(M):= V_{a}(M)/V_{a-\epsilon}(M),
$$
the first order operator $\sum_{k} \partial _{y_{k}} y_{k}$ acts with
generalized eigenvalue $-a$, that is to say $\sum_{k} \partial
_{y_{k}} y_{k} + a$ is nilpotent. Equivalently\footnote{This notation
  is shifted by $1$ from what one finds in $\Dd$-module literature
  \cite{Bjork,Budur}, but is commonly used in \cite{Sabbah} and other
  relevant references.  The relationship with the parabolic structure
  is shifted as shall be discussed further in the next section. For
  these reasons we usually include the shift in the notation and
  consider $V_{a-1}$.}, the vector field $\sum_{k} y_{k}\partial
_{y_{k}}$ acts there with generalized eigenvalue $-(a+1)$.

It is fairly straightforward to verify that the $V$-filtration is
uniquely characterized by the axioms. Kashiwara and Malgrange show
that if $M$ is regular holonomic and quasi-unipotent, its $V$-filtration
exists.

We can include $\lambda$-connections, i.e. replace $\Dd$-modules by
$\srR$-modules, by working instead with $p:\srP \rightarrow \srY$.
Only minor changes are needed, e.g. the generalized eigenvalue becomes
$-\lambda a$ instead of $-a$.

The interesting case for the present paper is when $X$ is a smooth
complex variety but the map $f:X \to Y$ is arbitrary. Equivalently,
the subscheme $Z \subset X$ defined by the ideal $I=(f_1,\ldots,f_k)$
is allowed to be arbitrarily singular, where the $f_i$ are coordinates
of the map $f$. In that case, let $P:= X\times Y$ be the product and
set $\srP := \srX \times _{\srA}\srY$.  We have the map $g:
\srX\rightarrow \srP$ given by the graph of $f$, that is to say
$$
g(x,\lambda ):= ((x,\lambda ),(f(x),\lambda )).
$$
Let $p:\srP \rightarrow \srY$ be the projection,
and let $\srX _q\cong \srX$ denote the fiber over $q\in Q$ of the
projection.

We may assume that there is only one singular point $q$, and indeed
that $Y$ is just an open disc. In coordinate notation we shall denote
by $t$ the coordinate on $\srP$ pulled back from the coordinate
function of the disk by the projection $p$. We can use interchangeably
the notation $\srX_0=\srX_q$ for the fiber at $t=0$.

From our $\srR_{\srX}$-module $\srE$ we get an $\srR_{\srP}$-module
$g_+(\srE )$ supported along the image of $g$. We then get the
$V$-filtration $V_{a-1}(g_+(\srE ))$.

\subsection{The de Rham complex} \label{dR}

We can form the de Rham complex
of $g_+(\srE )$ on $\srP$ 
relative to $\srY$, in the world of $\srR$-modules, 
and take the higher direct image to $\srY$.
The relative de Rham complex is
$$
DR(\srP/\srY,g_+\srE ) :=
\xymatrix@R-2pc{
\Big[ 
g_+(\srE )\ar[r] & 
g_+(\srE )\otimes_{\Oo _{\srP}} \Omega ^1_{\srP /\srY}
\ar[r] & 
g_+(\srE )\otimes_{\Oo _{\srP}}
\Omega ^2_{\srP /\srY}\Big]. \\
0 & 1 & 2
}
$$ Note that this is a somewhat different kind of object than the
$L^2$ parabolic Dolbeault or the $L^2$ parabolic de Rham complex
considered above: it doesn't {\em a priori} bring the parabolic
structure into play, rather relying on the $\srR_{\srX}$-module
structure of $\srE $. Our comparison problem is to relate these two
points of view, in order to deduce our main theorem from the following
result already known by Sabbah's work in the $\srR_{\srX}$-module
context.

Let
$$
\srF^i:=
{\mathbb R}^if_{\ast}
\left(
{\rm DR}(g_+\srE ) \right) .
$$
This $\srR _{\srY}$-module was also denoted by 
$\Hh ^i f_{\dagger}\srE$ in \cite[Remark 1.4.8]{Sabbah}.

Let $\overline{\srF}^i$ denote the minimal extension to $Y$ of
$\srF^i|_{\srY -\srQ}$. The localization $\widetilde{\srF}^i$
(inverting the equations of the points of $Q$) is the same as the
localization of $\overline{\srF}^i$.

On the other hand, consider the following local systems $G^i$ on
$Y-Q$: the fiber of $G^i$ at a point $y\in Y-Q$ is the cohomology of
the middle perversity extension $j_{X_y, \ast}(L|_{X_y-D_{H,y}})$ on
the fiber $X_y$. Equivalently, it is the $L^2$ cohomology of the
harmonic bundle restricted to $X_y-D_{H,y}$.

\begin{theorem}
\label{rthm}
With the above notations, and using the $V$-filtration:
\begin{enumerate}
\item[1.] $\overline{\srF}^i$ is an $\srR _{\srY}$-module on $\srY$,
with strict supports (Remark \ref{s-term}), 
in other words it is the minimal extension
  (Definition \ref{minext}) of its restriction to the open subset.  In
  particular, the $V$-filtrations at any point $q\in Q$ satisfy
$$
V_{b-1}(\overline{\srF}^i)= V_{b-1}(\srF ^i)= V_{b-1}(\widetilde{\srF}^i)
$$
for $b<1$, and $V_{b-1}(\overline{\srF}^i)$ are locally free.
\item[2.] $\overline{\srF}^i$ is the $\srR _{\srY}$-module
  corresponding to the harmonic bundle associated to the local system
  $G^i$.
\item[3.] In particular, the parabolic Higgs bundle whose components 
near a point $q\in Q$ are:
$$
F^i_b := V_{b-1}(\widetilde{\srF}^i)(0),
$$
is the parabolic Higgs bundle associated to $G^i$.
\end{enumerate}
\end{theorem}
\begin{proof}
  Sabbah proves the decomposition theorem for the higher direct images
  of a pure twistor $\Dd$-module. Mochizuki shows that our original
  harmonic bundle on $X$ gave rise to a pure twistor $\Dd$-module, of
  which $\srE$ is the part over the chart $\srA$ for the twistor
  line. We obtain the statement that $\srF ^i$ decomposes as a direct
  sum of $\srR_{\srY}$ modules with strict support (\ref{s-term}). Therefore,
  $\overline{\srF}^i$ is a direct summand of $\srF^i$.  This gives the
  properties of the $V$-filtrations at points $q\in Q$. By the
  discussion of \cite[Chapter 5]{Sabbah}, $\overline{\srF}^i$
  corresponds to a local system on $Y-Q$, and again applying this
  correspondence in the fibers, that local system is $G^i$. That
  correspondence gives the statement (3).
\end{proof}

\

\noindent
The $V$-filtration of
the relative de Rham complex, with respect to the  
function 
$$
p:\srP \rightarrow \srY ,
$$ 
is given by:
$$
DR(V_{a-1}(g_+\srE )):= 
\Big[ V_{a-1}(g_+\srE )\rightarrow
V_{a-1}(g_+\srE )\otimes \Omega ^1_{\srP /\srY}
\rightarrow
V_{a-1}(g_+\srE )\otimes \Omega ^2_{\srP /\srY} \Big].
$$
The compatibility between $V$-filtrations and higher direct images says:
\begin{proposition}
We have
$$
V_{a-1}(\widetilde{\srF}^i)
={\mathbb R}^ip_{\ast} DR(V_{a-1}(g_+\srE )).
$$
Furthermore, these bundles are strict with respect to specialization
on the
$\lambda$-line so
$$
V_{a-1}(\widetilde{\srF}^i)|_{Y(0)} =
{\mathbb R}^ip_{\ast}
\left( DR(V_{a-1}(g_+\srE )) |_{P(0)}
\right) .
$$
\end{proposition}
\begin{proof}
  See Sabbah \cite[Theorem 3.1.8]{Sabbah} generalizing Saito
  \cite[Proposition 3.3.17]{SaitoMHP}, but in turn this main property
  goes back to the original work of Kashiwara-Malgrange on the
  $V$-filtrations for $\Dd$-modules.
\end{proof}

In view of this proposition, we would like to compare
$$
DR(V_{a-1}(g_+\srE )) |_{P(0)}
\mbox{  and  }
{\rm DOL}^{\rm par}_{L^2}(E_{\alpha (a)}).
$$

\subsection{The multi-$V$-filtration and the parabolic 
structure}
\label{multi}

Let us consider, on the other hand, the multi-$V$-filtration obtained
by combining the individual $V$-filtrations along the divisor
components.  This is easier because the divisor components are smooth
so we don't need to use the graph embedding.

For each component $D_i$, we obtain a filtration
$$
V^i_{a-1}(\srE )\subset \srE  .
$$
It may be defined locally on open subsets where $D_i$ is the
zero-locus of a function.  Now let us intersect them: if $\beta$ is a
parabolic weight vector then we can use each $\beta _i$ along the
component $D_i$.  Denote by ${\bf 1}$ the multi-index whose components
are $1$.  Then we put
$$
V^{\rm multi}_{\beta -{\bf 1}}(\srE ):= \bigcap _i V^i_{\beta _i-1}(\srE ) \subset  \srE  .
$$
These submodules are closely related to the parabolic
structure. However, that relationship only holds for negative values
of the weight, because of the constraint imposed by the notion of
``minimal extension''. This restriction can be removed by passing to
the localized sheaf $\widetilde{\srE}$ of sections with poles along
$\srD$. Note that that sheaf is no longer of finite type over $\srR$,
however as Sabbah notes \cite[Section 3.4]{Sabbah}, the $V$-filtration
still works and we obtain as above
$$
V^{\rm multi}_{\beta-{\bf 1}}(\widetilde{\srE})
\subset \widetilde{\srE}.
$$

\begin{proposition}
\label{multivcomp}
Suppose $\beta _i<1$, then there is a natural identification between
the multi-$V$-filtration and the parabolic structure of $\srE$:
$$
V^{\rm multi}_{\beta -{\bf 1}}(\srE )\cong \srE _{\beta} .
$$
Notice here that the shift by $1$ is not present in the
parabolic structure. 

After localizing this extends to all values of $\beta$:
$$
V^{\rm multi}_{\beta -{\bf 1}}(\widetilde{\srE})\cong \srE _{\beta} .
$$
In particular, for any $\beta$ we have $\srE _{\beta}\subset \widetilde{\srE}$,
and for $\beta _i<1$ we have $\srE _{\beta}\subset \srE $.
\end{proposition}
\begin{proof} 
  See \cite[Lemma 3.4.1]{Sabbah}.  Then for the identification with
  the parabolic filtration of the parabolic Higgs bundle, see
  \cite[Corollary 5.3.1]{Sabbah} in the one-dimensional case.  The
  shift between the parabolic index and the index for the
  $V$-filtration was noted in \cite[Section 5]{Sabbah}, passing
  through his introduction of an upper-indexed decreasing filtration
  that already incorporates the shift. There is a change of direction
  in the parabolic filtrations: Sabbah considers a decreasing
  parabolic filtration.  This comparison has been extended to the
  higher dimensional case in Chapter 15 of \cite{MochizukiPure} (cf
  15.1.2 for the shift).
\end{proof}

Unfortunately, the multi-$V$-filtration doesn't enter into the theory
of Saito-Sabbah about compatibility with direct images. We need
instead to look at the $V$-filtration with respect to the full
vertical divisor $\srD _V=f^{-1}(Q)$ defined above.  Along smooth
points of $\srD _V$ this is the same as the multi-$V$-filtration
because there is only one index.  The difficulty occurs at a normal
crossing point, and will be the subject of our main work to follow.

\subsection{Comparison of the two de Rham complexes}
\label{constructmap}

In order to compare the two complexes we would like to have a map
between them. However, they are not complexes on the same space.

Consider $\srE $ as a quasicoherent sheaf of $\Oo _{\srX}$-modules on
$\srX$ and as such take the direct image $g_{\ast}(\srE )$ to $\srP$.
There is an inclusion of sheaves of $\Oo
_{\srP}$-modules
$$
g_{\ast}(\srE )\hookrightarrow
g_{+}(\srE ).
$$
Locally if $t$ is a coordinate on a neighborhood in $Y$ and we let
$\partial _t$ be the corresponding tangent vector field, then we can
write
\begin{equation}
\label{expressiongplus}
g_{+}(\srE ) =
g_{\ast}(\srE )[\partial _t]
\end{equation}
(see \cite[Remark 1.4.7]{Sabbah}). The above inclusion is
just the term with the $0$-th power of $\partial _t$.

If we let $\srE_{\alpha (a)}$ denote the $\alpha (a)$ piece in the
parabolic bundle with $\lambda$-connection $\srE_{\ast}$ then for
$a<1$ we have a morphism
$$
\srE_{\alpha (a)}\rightarrow \srE .
$$

We get the composition
$$
g_{\ast}(\srE_{\alpha (a)})\rightarrow
g_{\ast}(\srE )\rightarrow
g_{+}(\srE ).
$$

Now recall that
$g_{+}(\srE )$ is an $\srR_{\srP}$-module,
in particular it has an action of $V_0\srR_{\srP}$ which is
the $V$-filtration along the divisor $\srX _q$ (here we fix a
point $q\in Q$ which we think of as the origin $t=0$ of a
disc).

Near the horizontal divisors we will be importing later the discussion
of the $L^2$ complex from Sabbah's Section 6.2, so for the moment we
work locally away from the horizontal divisors.

Saito-Mochizuki's description of the $V$-filtration may be
summed up in the following proposition:

\begin{proposition}
\label{vdescrip}
Away from the horizontal divisors, for $a<1$
the $V$-filtration $V_{a-1}$ of the module
$g_{+}(\srE )$ along $\srX _q$ is the
sub-$V_0\srR_{\srP}$-module generated by
the image of
$$
g_{\ast}(\srE_{\alpha (a)})\rightarrow
g_{+}(\srE ).
$$
In fact, more strongly 
it is the sub-$q^{-1}\srR _{\srX}$-module generated
by this image where $q:\srP \rightarrow \srX$ is the
first projection. 
\end{proposition}
\begin{proof}
  Mochizuki gives a review in \cite[12.2.2]{MochizukiWild}, refering
  to \cite[Section 16.1]{MochizukiPure} and going back to
Saito's  \cite[Theorem 3.4]{SaitoMHM}.  Note that they phrase the statement
  in terms of the multi-$V$-filtration considered before, then we use
  the compatibility of Proposition \ref{multivcomp}.
  
  The stronger statement that $V_{a-1}g_{+}(\srE )$
  is obtained from $g_{\ast}(\srE_{\alpha (a)})$ just
  by the action of $\srR _{\srX}$ is what is actually stated and shown in the references. In order to clarify the property that
this generates a $V_0\srR_{\srP}$-module, we'll show ourselves
that the two generation expressions are the same in Section 
\ref{sec-nc} below. 
\end{proof}

\

\noindent
In particular we obtain a map
\begin{equation}
\label{gstargplus}
g_{\ast}(\srE_{\alpha (a)})\rightarrow
V_{a-1}(g_{+}(\srE )).
\end{equation}
We would like to use this to obtain a map of
de Rham complexes. 
Recall that
$$
DR^{\rm par}_{L^2}(\srX/\srY,\srE _{\alpha (a)}):=
\Big[ W_0(H,\srE _{\alpha (a)}) \longrightarrow
W_{-2}(H,\srE _{\alpha (a)}) \otimes_{\Oo _{\srX}}
\Omega ^1_{\srX / \srY} (\log \srD )
\Big] .
$$
Locally away from the horizontal divisor this may be written
more simply as
$$
\Big[ \srE _{\alpha (a)} \longrightarrow
\srE _{\alpha (a)} \otimes_{\Oo _{\srX}}
\Omega ^1_{\srX / \srY} (\log \srD )
\Big].
$$
Apply $g_{\ast}$ to this complex, to get a complex
whose terms are coherent $\Oo _{\srP}$-modules.
The differential is still well-defined, since it comes from 
an action of vector fields tangent to the image $\srG =g(\srX )$
of the graph embedding. We get a complex
that we would like to map to the
de Rham complex $DR(\srP/\srY,V_{a-1}(g_+(\srE )))$
on $\srP$ relative to $\srY$. 

If $Y$ is a disk with coordinate $t$, $\Omega ^1_{\srY}$ is
trivialized with generator denoted $dt$.
Wedge with $dt$ gives maps of sheaves of differentials
$$
\Omega ^i_{\srX /\srY} (\log \srD )
\stackrel{\wedge dt}{\longrightarrow}
\Omega ^{i+1}_{\srX} . 
$$
In local coordinates, $t=xy$ and the above map for $i=1$ is
given by 
$$
\frac{dx}{x} \mapsto \frac{dx}{x} \wedge dt = 
dx \wedge dy . 
$$
Recall here that $\Omega ^1_{\srX / \srY} (\log \srD )$
is locally free of rank $1$ on $\srX$.

Take the product with a copy of $\srY$. 
Introduce the notation $T:= Y\times Y$ and
$\srT := \srY \times _{\srA}\srY$. We have a map
$$
(f,1):\srP \rightarrow \srT .
$$
Let $t_1$ and $t_2$ denote the two coordinates on $\srT$. 
Over $\srP$ the previous multiplication map,
interpreted in the first variable of $\srP = \srX \times _{\srA}\srY$, gives the multiplication map 
$$
\Omega ^i _{\srP / \srT} (\log \srD\times _{\srA}\srY )
\stackrel{\wedge dt_1}{\longrightarrow}
\Omega ^{i+1}_{\srP /\srY},
$$
where $\srD\times _{\srA}\srY \subset \srP$ is the divisor induced by $\srD \subset \srX$. 

Using this together with the previous
map \eqref{gstargplus}
gives a map between complexes that may be described
as follows.

Rewriting the terms,
we claim that
$$
g_{\ast}(\Omega ^i _{\srX / \srY} (\log \srD ))
=
g_{\ast}(\Oo _{\srX}) \otimes _{\Oo _{\srP}}
\Omega ^i_{\srP /\srT}(\log \srD\times _{\srA}\srY ),
$$
where $\srD\times _{\srA}\srY \subset \srP$ is the induced divisor.
To see this use the first projection $q:\srP \rightarrow \srX$
and note 
$$
\Omega ^i_{\srP /\srT}(\log \srD\times _{\srA}\srY )=
q^{\ast}(\Omega ^i _{\srX / \srY} (\log \srD )).
$$
Clearly
$$
g_{\ast}(\Oo _{\srX}) \otimes _{\Oo _{\srP}}
q^{\ast}(\Omega ^i _{\srX / \srY} (\log \srD ))
= g_{\ast}(\Omega ^i _{\srX / \srY} (\log \srD ))
$$
since $g$ is a section of the projection. 

Using the above formula gives
$$
g_{\ast}\left(
\srE _{\alpha (a)}\otimes _{\Oo _{\srX}}
\Omega ^i_{\srX /\srY} (\log \srD )
\right)
= (g_{\ast}\srE  _{\alpha (a)}) \otimes
_{\Oo _{\srP}}
\Omega ^i_{\srP /\srT}(\log \srD\times _{\srA}\srY ).
$$
Then use 
$$
\Omega ^i_{\srP /\srT}(\log \srD\times _{\srA}\srY )
\stackrel{\wedge dt_1}{\longrightarrow}
\Omega ^{i+1}_{\srP /\srY}
$$
to get a map
$$
g_{\ast}\left(
\srE _{\alpha (a)}\otimes _{\Oo _{\srX}}
\Omega ^i_{\srX /\srY} (\log \srD )
\right)
\longrightarrow
g_{\ast}(\srE  _{\alpha (a)})\otimes
_{\Oo _{\srP} }
\Omega ^{i+1}_{\srP /\srY} .
$$
Compose furthermore with 
$$
(g_{\ast}\srE  _{\alpha (a)})
\rightarrow
V_{a-1} (g_+(\srE ))
$$
to obtain a map
\begin{equation}
\label{map3}
g_{\ast}\left(
\srE _{\alpha (a)}\otimes _{\Oo _{\srX}}
\Omega ^i_{\srX /\srY} (\log \srD )
\right)
\longrightarrow
V_{a-1} (g_+(\srE ))\otimes
_{\Oo _{\srP} }
\Omega ^{i+1}_{\srP /\srY } .
\end{equation}

Another way of looking at things is as follows: 
let $T(\srG /\srY )$ 
be the sheaf of tangent vector fields on the graph 
$\srG := g(\srX )\subset \srP$, relative to $\srY$.  
Thinking that $\srG\cong \srX$, 
these tangent vector fields are just the tangent 
vector fields on 
$\srX$ relative to $\srY$, and that bundle is dual to 
$\Omega ^1_{\srX /\srY} (\log \srD )$.  
Now, this sheaf of tangent vector 
fields acts on $g_{\ast}(\srE _{\alpha (a)})$ and 
(away from the horizontal divisors) the parabolic de Rham complex 
$g_{\ast}DR^{\rm par}_{L^2}(\srX/\srY,\srE _{\alpha (a)})$ 
is just the de Rham complex (or perhaps more accurately 
the ``Spencer complex'' \cite[\S 0.6]{Sabbah})
associated to this action. On the other hand, the relative 
tangent vectors to the graph map to the 
tangent vectors of $\srP /\srY$:
$$
T(\srG /\srY )\rightarrow T(\srP /\srY )|_{\srG}.
$$
This map induces the desired map on de Rham complexes.

The expression \eqref{expressiongplus} 
for $g_{+}(\srE )$, if made in
the world of left modules, requires a 
choice of coordinate on $Y$. To get an invariant expression
in terms of left modules (we thank T. Mochizuki
for asking for that), let 
$\srG = g(\srX )\subset \srP $ coming from 
$G=g(X)\subset P$ denote the graph divisor, and
recall that the notation $(\ast \srG )$ means functions
with arbitrary finite order of pole along $\srG$. 
Let 
$$
\Oo _{\srP}({}^{\lambda}\!\ast \srG )\subset
\Oo _{\srP}(\ast \srG )
$$
denote the sub-$\srR_{\srP}$-module generated by 
the functions with a pole of order $1$. It is the minimal
extension (Definition \ref{minext}) 
of $\Oo _{\srP -\srG}$, and may be written down
as the sheaf of sections that are sums of $\lambda ^{m-1}$
times functions with poles of order $m$. 
Use the notation $({}^{\lambda}\!\ast\srG )$ for 
tensoring with $\Oo _{\srP}({}^{\lambda}\!\ast \srG )$. 
 
Recall
also that $q:\srP \rightarrow \srX$ denotes the 
first projection. The
canonical expression is 
\begin{equation}
\label{canexp}
g_+(\srE ) = \frac{q^{\ast}(\srE) ({}^{\lambda}\!\ast \srG )}{ 
q^{\ast}(\srE)}. 
\end{equation}
The map \eqref{gstargplus} depends on a choice of
section of $\Oo _{\srP} ({}^{\lambda}\!\ast \srG )/\Oo _{\srP}$  having
a pole of order one, that depends on our choice
of coordinate on $Y$. 

Writing down a canonical map will 
in fact give the beginning of our
map of complexes. Let 
$$
\zeta _{\srG}\in \frac{\Omega ^1_{\srP /\srY} (\log \srG )}
{\Omega ^1_{\srP /\srY}} 
\subset 
\frac{\Omega ^1_{\srP /\srY} ({}^{\lambda}\!\ast \srG )}
{\Omega ^1_{\srP /\srY}} 
$$
be the unique section given by a logarithmic differential 
with residue $1$ along $\srG$. 
Multiplication by $\zeta _{\srG}$ gives a map 
\begin{equation}
\label{cpx1}
q^{\ast}(\srE ) |_{\srG} = g_{\ast}(\srE ) 
\stackrel{\cdot \zeta _{\srG}}{\longrightarrow} 
\frac{q^{\ast}(\srE )\otimes \Omega ^1_{\srP /\srY}
({}^{\lambda}\!\ast \srG )}{q^{\ast}(\srE )\otimes 
\Omega ^1_{\srP /\srY}}
= g_+(\srE )\otimes \Omega ^1_{\srP / \srY} .
\end{equation}
For the second step of the complex, consider the
map of  multiplication by $\zeta _{\srG}$,
$$
q^{\ast}(\Omega ^1_{\srX}) = \Omega ^1_{\srP / \srY}
\stackrel{\wedge \zeta _{\srG}}{\longrightarrow} 
\frac{\Omega ^2_{\srP /\srY} 
({}^{\lambda}\!\ast \srG )}{\Omega ^2_{\srP /\srY}}.
$$
We claim that it factors through and extends
to a map
$$
q^{\ast}(\Omega ^1_{\srX / \srY} (\log \srD ))
\stackrel{\wedge \zeta _{\srG}}{\longrightarrow} 
\frac{\Omega ^2_{\srP /\srY} 
({}^{\lambda}\!\ast \srG )}{\Omega ^2_{\srP /\srY}}.
$$

Let us calcuate this locally, with $Y$ a disk having coordinate
$t_2$
and $X$ a product of two disks with coordinates $x,y$.
The map is $t_1=f(x,y)=xy$ so the graph is given by
the equation $xy-t_2$. Note that $dt_1=ydx + xdy$ so
$$
\Omega ^1_{\srX / \srY}  = \frac{\Omega ^1_{\srX}}{\langle
ydx + xdy\rangle }.
$$
On the other hand, $P$ has coordinates $x,y,t_2$ and 
$$
\zeta _{\srG} = \frac{d(xy-t_2)}{xy-t_2}
=d\log (t_1-t_2). 
$$
The factorization part of the
claim is equivalent to the equation 
\begin{equation}
\label{claimeq}
(ydx + xdy) \wedge \frac{d(xy-t_2)}{xy-t_2} = 0 \mbox{  mod }
dt_2 .
\end{equation}
This equation is clear since, modulo $dt_2$ it just
becomes $d(xy) \wedge d(xy) / (xy-t_2) =0$. 
The extension part of the claim asks that
$$
\frac{dx}{x}\wedge \frac{d(xy-t_2)}{xy-t_2}
\in \Omega ^2_{\srP /\srY} 
({}^{\lambda}\!\ast \srG ) .
$$
Again modulo $dt_2$ the expression becomes 
$$
\frac{dx}{x}\wedge \frac{ydx + xdy}{xy-t_2}
= \frac{dx\wedge dy}{xy-t_2}
$$
which is indeed in $\Omega ^2_{\srP /\srY} 
({}^{\lambda}\!\ast \srG )$. The same
holds for multiplication by $dy/y$. 
This completes the proof of the claim. 

We now get the map of the second term in our complex:
$$
q^{\ast}(\srE \otimes 
\Omega ^1_{\srX /\srY}(\log \srD )) |_{\srG} 
= g_{\ast}(\srE \otimes 
\Omega ^1_{\srX /\srY}(\log \srD )) 
$$
\begin{equation}
\label{cpx2}
\stackrel{\wedge \zeta _{\srG}}{\longrightarrow}
\frac{q^{\ast}(\srE )\otimes \Omega ^2_{\srP /\srY}
({}^{\lambda}\!\ast \srG )}{q^{\ast}(\srE )\otimes 
\Omega ^2_{\srP /\srY}}
= g_+(\srE )\otimes \Omega ^2_{\srP / \srY} .
\end{equation}

\

\begin{lemma}
\label{mapcomplexes}
The above maps \eqref{map3} 
for $i=0,1$ are the same as the
maps \eqref{cpx1} and \eqref{cpx2} respectively. 
These maps are compatible with the differentials
in the complexes, so they define a map of complexes
$$
u(a): g_{\ast}(DR^{\rm par}_{L^2}(\srX/\srY,
\srE  _{\alpha (a)}))
\rightarrow
DR(\srP/\srY,V_{a-1}(g_+(\srE ))).
$$
It is shifted, i.e. it
sends the degrees $0,1$ in the first one, to the
degrees $1,2$ in the second one.
\end{lemma}
\begin{proof}
For the compatibility statement note that 
$W_{\ell}(H,\srE _{\alpha (a)})\subset \srE$. 
The maps \eqref{cpx1} and \eqref{cpx2}  
are defined on $\srE$ but they restrict to maps
defined on the appropriate $W_{\ell}(H,\srE _{\alpha (a)})$.
Conserving the notations $t_1$ and $t_2$ for the
two maps from $\srP$ to $\srY$, with $t_1=xy$ and
$xy-t_2=t_1-t_2$ being the equation of the graph, 
the identification between \eqref{canexp} and 
\eqref{expressiongplus} is
\begin{equation}
\label{identif}
\frac{e}{(t_1-t_2)}\mapsto e.
\end{equation}
The maps \eqref{map3} for $i=0,1$ are 
$e\mapsto e\wedge dt_1$ whereas \eqref{cpx1} and
\eqref{cpx2} are
$e\mapsto e\wedge \frac{dt_1}{t_1-t_2}$. These are
the same after the identification \eqref{identif}.

The differential $\zeta _{\srG}=d\log (t_1-t_2)$ is closed,
so wedging with it commutes with the exterior derivative. 
Commutativity of the square comprising this map of complexes
can be seen more explicitly in the calculations of the
next sections. 
\end{proof}

\noindent {\em Remark:} \, Recall that the above
general discussion was happening away
from the horizontal divisors, in order to lighten the notation and not
include the weight filtrations and so forth. At the horizontal
divisors, the map from the $L^2$ complex to the $\srR$-module de Rham
complex exists, by Sabbah's discussion of \cite[Section 6.2.a, Lemma
6.2.2]{Sabbah}.  This needs to be interpreted in the following way.
In the general theory, if we try to define the de Rham complex using
the graph embedding for a smooth function, we get a different complex
but one that is quasiisomorphic. The map expressing this
quasiisomorphism will go from the de Rham complex on the smaller
variety, pushed forward to the graph, towards the de Rham complex of
the graph embedding (we will be doing a version of this at the normal
crossing points, below).  We should compose this map, with the map
from the $L^2$ complex to the full de Rham complex given in
\cite[Section 6.2.a]{Sabbah}, in order to obtain the map from our
$L^2$ Dolbeault complex to $DR(\srP/\srY,V_{a-1}g_+(\srE ))$. This will coincide
with the map we are considering here away from the horizontal
divisors.

Our main result is:

\begin{theorem}
\label{cpxqis}
The map of complexes $u(a)$ in Lemma \ref{mapcomplexes}
is a quasiisomorphism.
\end{theorem}

\begin{corollary}
The higher direct images from $\srX$ to $\srY$
of the complex
$DR^{\rm par}_{L^2}(\srX/\srY,\srE  _{\alpha (a)})$, or of any
of its restrictions to $X(\lambda )$, notably
including the Dolbeault complex at $\lambda = 0$,
are locally free
and compatible with base-change.
\end{corollary}
\begin{proof}
We know from Saito-Sabbah-Mochizuki that this is
true for $DR(\srP/\srY,V_{a-1}(g_+(\srE )))$,
then apply the quasiisomorphism of the theorem.
\end{proof}

\

\begin{proposition}
\label{proofmainth}
Theorem \ref{mainth} follows from 
Theorem \ref{cpxqis}. 
\end{proposition}
\begin{proof}
  Parts 1 and 2 of Theorem \ref{mainth} follow from Theorem \ref{rthm}
  once we identify the two de Rham complexes using the
  quasiisomorphism of Theorem \ref{cpxqis}.  Part 3 follows from the
  discussion of \cite[Chapter 5]{Sabbah} relating $\srR$-modules on a
  curve with parabolic Higgs bundles. The Higgs field may also be
  constructed in the usual way, starting from the absolute Dolbeault
  complex on $X$ (rather than the relative one on $X/Y$ that we are
  looking at)---see Section \ref{globalcomplex} for more details.

Part 4 comes from part 2 of Theorem \ref{rthm}.

For part 5, the calculation in \cite{SimpsonFam} applies to the $L^2$
metric on a family of open complete curves (we are using the
Poincar\'e metrics at the punctures of the fiber curves). Here some
analytic considerations are needed that are left to the reader: the
spaces of harmonic forms, all having the same dimension, fit together
into a bundle and we can differentiate harmonic forms just as was done
for the compact case in \cite{SimpsonFam}. That argument shows that
the $L^2$ metric is harmonic. Norm estimates show one direction of
compatibility with the parabolic structure, and the other direction
then follows because the degree of the parabolic bundle is zero. This
discussion will be expanded upon in Section \ref{ancon} below.
\end{proof}

\section{Quasiisomorphism}\label{sec-qi}

Put
$$
\srQ (a):= DR^{\rm par}_{L^2}(\srX/\srY,\srE  _{\alpha (a)})
$$
and
$$
Q(a):= \srQ (a) |_{\lambda =0} =
{\rm DOL}^{\rm par}_{L^2}(X/Y,E_{\alpha (a)}).
$$
Recall that $\srQ (a)$ is a complex of vector bundles on $\srX$ and
$Q(a)$ is a complex of vector bundles on $X$. The differentials of
$\srQ (a)$ are differential operators, but on $Q(a)$ they are
$\Oo_X$-linear. We have
$$
\srQ (a-1)= t\srQ(a) = \srQ (a) (-D_V) \subset \srQ (a)
$$
where $t$ is the coordinate defining the fiber over $t=0$ in $Y$. We
have $\srQ (a)\subset \srQ(a')$ when $a\leq a'$.  The same hold for
the $Q(a)$.

The graph embedding $g:X\rightarrow P$ induces complexes
$g_{\ast}(\srQ (a))$ on $\srP$ and $g_{\ast}(Q(a))$ on
$P$.

On the other hand, put
$$
\srK (a):= DR(\srP/\srY,V_{a-1}(g_+(\srE ))
$$
and
$$
K(a) := \srK(a) |_{\lambda = 0}.
$$
These are complexes of quasicoherent sheaves on $\srP$ and $P$
respectively. For $a<1$ these collections have the same properties as
the $\srQ (a)$ and $Q(a)$.

Our map of complexes given in Lemma \ref{mapcomplexes}
$$
u(a):g_{\ast}(\srQ (a))\rightarrow \srK (a)
$$
induces on $\lambda =0$ the map
\begin{equation}
\label{u0a}
u_0(a): g_{\ast}(Q(a))\rightarrow K(a).
\end{equation}

Consider the projection maps $p:\srP\rightarrow \srY$ and
similarly (denoted by the same letter) $p:P\rightarrow Y$.
Note that $p\circ g=f$ is our original map.

\subsection{Higher direct image of $K(a)$}

For $\srK(a)$ and $K(a)$ which are the de Rham complexes of
$\srR$-modules the general theory of Saito-Sabbah gives the
properties we state in the following proposition.

\begin{proposition}
For $a<1$, we have a decomposition 
$$
\mathbb{R}p_{\ast}(\srK (a)) \cong \bigoplus _{i=0}^2
\srF ^i(a)[-i]
$$
where
$$
\srF^i(a):=
\mathbb{R}^ip_{\ast}(\srK (a)),
$$
and these terms are vector bundles. Furthermore again assuming
$a,a'<1$ if $a\leq a'$ then $\srF^i(a)\hookrightarrow \srF^i(a')$. In
particular $\srF^i(a-1)=t\srF^i(a)$.

Let
$$
F^i(a):= \mathbb{R}^ip_{\ast}(K (a)).
$$
Then $F^i(a)$ is also equal to the restriction of $\srF^i(a)$ to
$\lambda = 0$, in particular these are vector bundles. We again have
the decomposition
$$
\mathbb{R}p_{\ast}(K (a)) \cong \bigoplus _{i=0}^2
F ^i(a)[-i]
$$
and $F^i(a-1)=tF^i(a)$ for $a<0$,
with $F^i(a)
\subset F^i(a')$ when 
$a\leq a'<0$.
\end{proposition}
\begin{proof}
  The first part comes from the decomposition theorem of Sabbah and
  Mochizuki. The other parts restate Theorem \ref{rthm} and gather
  some of the usual properties of the $V$-filtration.
\end{proof}

We now recall the main property calculating the parabolic Higgs bundle
associated to the harmonic bundle associated to the higher direct
image.

\begin{scholium}
  In the global situation over $Y$, the collection of $F^i(a)$ for
  $a<1$ coincides with the parabolic Higgs bundle over $Y$ associated
  to the harmonic bundle associated to the local system $R^if_{\ast
    !}(L)$.
\end{scholium}
\begin{proof}
This restates Theorem \ref{rthm}. 
\end{proof}

\subsection{Higher direct image of $Q(a)$}

Our complexes $\srQ (a)$ and $Q(a)$ are complexes of vector
bundles. Recall that $\mathbb{R}f_{\ast}=\mathbb{R}p_{\ast}\circ g_{\ast}$
on these complexes since $g_{\ast}$ is a closed embedding so it is
acyclic. We have the following standard property.

\begin{proposition}
The complexes
$\mathbb{R}f_{\ast}(\srQ (a))$ and $\mathbb{R}f_{\ast}(Q(a))$ are
perfect complexes on $\srY$ and $Y$ respectively.
\end{proposition}
\begin{proof}
  This is by a standard argument using the fact that $f$ is flat and
  that the complexes are complexes of vector bundles with differential
  operators as maps.
\end{proof}

\

\noindent
Put
$$
\srC (a):= {\rm Cone}(u(a)), \;\;\; C(a):=  {\rm Cone}(u_0(a)).
$$
These are complexes on $\srP$ and $P$ respectively,
supported on the graph of $f$ which is the image of $g$. Their
terms are quasicoherent sheaves on $\srP$ or $P$ respectively.

\begin{lemma}
The terms of all of our complexes are flat over $\srY$ and $Y$
respectively.
\end{lemma}
\begin{proof}
There is no $\lambda$-torsion,
then after restricting to $\lambda=0$ there is no $t$-torsion. 
\end{proof}

\subsection{The quasiisomorphism over smooth points}

The first basic result is:

\begin{lemma}
\label{smoothpoints}
If $x\in X$ is a point where $f$ is smooth, then
$u(a)$ and $u_0(a)$ are quasiisomorphisms at $g(x)=(x,f(x))$.
\end{lemma}
\begin{proof}
  This is by the independence of the $V$-filtration under choice of
  how to calculate it (by a graph embedding or not); our complex $\srQ
  (a)$ is just the de Rham complex of the $V_{a-1}$ piece, at a smooth
  point of $f$.
\end{proof}

\

\begin{corollary}
\label{finsupport}
The complexes $\srC (a)$ and $C(a)$ are cohomologically supported at
the singular points of $f$, so this support is finite over $\srY$
(resp. $Y$).
\end{corollary}

\

\begin{corollary}
\label{downstairs}
For any $a$, if we can show that the map $\mathbb{R}p_{\ast}(u_0(a))$ is
a quasiisomorphism then $C(a)$ and $\srC (a)$ are acyclic, hence
$u(a)$ and $u_0(a)$ are quasiisomorphisms.
\end{corollary}
\begin{proof}
  Suppose $C(a)$ is not acyclic. Then since it is supported at a
  finite set it would imply that $\mathbb{R}p_{\ast}(u_0(a))$ not be a
  quasiisomorphism.  From our hypothesis we therefore conclude that
  $C(a)$ is acyclic. Now it follows that $u_0(a)$ is a
  quasiisomorphism.

  Consider the map $\mathbb{R}p_{\ast}(u(a))$. It is a map between
  perfect complexes on $\srY$, and its restriction to $Y$ is
  $\mathbb{R}p_{\ast}(u_0(a))$, which we are assuming is a
  quasiisomorphism. Therefore it is a quasiisomorphism over an open
  set in $\srA$ containing $\lambda =0$.  As before, from the
  finiteness statement of Corollary \ref{finsupport} we conclude that
  $u(a)$ is a quasiisomorphism over that open set.

  We now claim that for any semisimple local system $L$ and any
  $\lambda \neq 0$, if we consider the corresponding parabolic
  $\lambda$-connection $(\underline{E}(\lambda ),\nabla )$ then the
  map $u(a)(\lambda )$ is a quasiisomorphism.  We sketch an argument
  as follows.  Look at a $1$-parameter family of local systems $L_t$
  such that the $t\lambda$-connection corresponding to $L_t$ is
  $(\underline{E}(\lambda ),t \nabla )$.  We know that
  $(\underline{E}(\lambda ),t \nabla )$ approach a limiting polystable
  parabolic Higgs bundle (one can use the same technique as in
  \cite{SimpsonIDSM}).  For this limiting Higgs bundle the map
  $u_0(a)$ is a quasiisomorphism as we have seen above, and by
  semicontinuity (using perfectness of the source and target of the
  higher direct image, and the global to local argument coming from
  Corollary \ref{finsupport} as above), it follows that the map
  $u_t(a)$ corresponding to $(\underline{E}(\lambda ),t \nabla )$ is a
  quasiisomorphism for $t$ near to $0$. However, the maps $u_t(a)$ are
  all the same as our original $u(a)$ up to scaling, so this shows
  that the original $u(a)$ is a quasiisomorphism.
\end{proof}

\subsection{The quasiisomorphism near horizontal divisors}

Next we quote the following main result from Saito and Sabbah. It
basically goes back to Zucker's paper \cite{Zucker}. This result was
of course the motivation for introducing the $L^2$ de Rham complexes.

\begin{proposition}
\label{horizontaldiv}
Suppose $x\in X$ is an intersection point of a horizontal and vertical
divisor component. Then $u(a)$ and $u_0(a)$ are quasiisomorphisms at
$g(x)=(x,f(x))$.
\end{proposition}
\begin{proof}
This is Proposition 6.2.4 of \cite{Sabbah}. 
\end{proof}

\subsection{The quasiisomorphism at double points---statement}

In this subsection we state the result of our main calculation.

\begin{theorem}
\label{maincalc}
Suppose $x\in X$ is an intersection point of two vertical divisor
components. Then the map $u_0(a)$ of \eqref{u0a} is a quasiisomorphism
at $g(x)=(x,f(x))$.
\end{theorem}

The proof will occupy the several upcoming sections.

\begin{corollary}
\label{proofcpxqis}
The maps $u(a)$ and $u_0(a)$ are quasiisomorphisms, that is to say
this gives the proof of Theorem \ref{cpxqis}.
\end{corollary}
\begin{proof}
  The statement of Theorem \ref{maincalc} gives the quasiisomorphism
  at crossing points of vertical components.  On smooth points use
  Lemma \ref{smoothpoints}, and at points where the horizontal divisor
  meets the vertical divisor, use Proposition \ref{horizontaldiv}.
\end{proof}

\begin{corollary}
For any $a$ we have
$$
\mathbb{R}f_{\ast}(Q(a))= \bigoplus R^i f_{\ast}(Q(a)).
$$
These terms are vector bundles. For $a\leq -2$ we have $R^i
f_{\ast}(Q(a))=F^i(a)$. The vector bundles $R^i f_{\ast}(Q(a))$ fit
together into a parabolic Higgs bundle, and this is the parabolic
Higgs bundle associated to the local system $R^if_{\ast !}(L)$.
\end{corollary}

\subsection{Reduction to $t=0$}

We point out here, using the above properties, that it suffices to
consider the restriction to the fiber over $t=0$. Let $Z:=
f^{-1}(0)\subset X$. Since we are localizing to a disk in $Y$ with
only one singular point, we have in fact $Z=D_V$. Let $\srZ := Z\times
\srA$ as usual.  Note that $Z=g(X)\cap (X\times \{ 0\})$.

\begin{proposition}
\label{reduction1}
In order to prove Theorem \ref{maincalc} it suffices to prove that the map
$$
u_0(a)|_{X\times \{0\} }:
g_{\ast}(Q(a))|_{X\times \{0\} } \rightarrow
K(a)|_{X\times \{0\} }
$$
is a quasiisomorphism near any double point of $Z\subset
X\times \{0\} $.
\end{proposition}
\begin{proof}
  The complexes $Q(a)$ and $K(a)$ are flat over $Y$, hence so are
  their higher direct image complexes. Let
$$
B(a):= {\rm Cone}\left(
\xymatrix@1@C+2pc{
\mathbb{R}f_{\ast}(Q(a))
\ar[r]^-{\mathbb{R}p_{\ast}(u_0(a))} &
\mathbb{R}p_{\ast}(K(a))
}
\right) .
$$
It is again a complex of sheaves flat over $Y$. The
exact triangle containing $B(a)$ restricts to an exact triangle
over the point $t=0$ and indeed
$$
B(a)|_{t=0} =
{\rm Cone}\left(
{\mathbb H}^{\bullet}(Q(a)|_Z)\rightarrow
{\mathbb H}^{\bullet}(K(a)|_{X\times \{0\} })
\right) .
$$
Suppose we know the statement that $u_0(a)|_{X\times \{0\} }$ is a
quasiisomorphism. Then it follows (using the above arguments at smooth
points of $Z$ as well as at points of $D_H\cap Z$) that $B(a)|_{t=0}$
is exact. On the other hand, both $\mathbb{R}f_{\ast}(Q(a))$ and
$\mathbb{R}p_{\ast}(K(a))$ are perfect complexes as we have pointed
out above.  Therefore the cone $B(a)$ on the map between them is a
perfect complex. It is cohomologically supported at $t=0$.  Now we may
conclude by using the following property of perfect complexes: a
perfect complex which is cohomologically supported at a point, and
whose restriction to that point is acyclic, is acyclic. Therefore,
$B(a)$ is acyclic and $\mathbb{R}p_{\ast}(u_0(a))$ is a
quasiisomrphism. By Corollary \ref{downstairs} this will show Theorem
\ref{maincalc}, completing our reduction.
\end{proof}

\

\noindent
We may further reduce using the parabolic structure.
Let $\varepsilon$ be strictly smaller than the difference
of any two parabolic weights.
For any parabolic weight $a$ we have the map of complexes
\begin{equation}
\label{gru}
{\rm gr}_a(u_0):
g_{\ast}(Q(a)/Q(a-\varepsilon ) )\rightarrow
K(a)/K(a-\varepsilon ).
\end{equation}

\begin{proposition}
\label{reduction2}
In order to prove Theorem \ref{maincalc}, it suffices to
show that for any parabolic weight $a$, the map
${\rm gr}_a(u_0)$ of
\eqref{gru} is a quasiisomorphism near double points of $Z$.
\end{proposition}
\begin{proof}
  Use the previous Proposition \ref{reduction1}. Both the source and
  target of the map $u_0(a)|_{X\times \{0\} }$ have filtrations such
  that the graded quotients are respectively
  $g_{\ast}(Q(a')/Q(a'-\varepsilon ) )$ and $K(a')/K(a'-\varepsilon
  )$, for the parabolic weights $a-1 < a' \leq a$. Hence, if we know
  that the ${\rm gr}_{a'}(u_0)$ are quasiisomorphisms, it will follow
  that $u_0(a)|_{X\times \{0\} }$ is a quasiisomorphism; then
  Proposition \ref{reduction1} leads to Theorem \ref{maincalc}.
\end{proof}

\section{Proof at a normal crossing} \label{sec-nc}

We now turn to the proof of Theorem \ref{maincalc}, using the
reductions above. By proposition \ref{reduction2} we would like to
obtain a quasiisomorphism for the graded pieces of the parabolic
structure.

\subsection{The $V$-filtration at a normal crossing}

As the remaining problem is to treat an intersection of two vertical
divisor components, let us restrict to a local situation. Thus we may
assume that $Y$ is a disk with coordinate $t$, and the singular fiber
is $t=0$. On $\srP =\srX \times _{\srA}\srY$ we think of $t$ as being
the coordinate of the second factor $\srY$.  Let $\partial _t$ denote
the vector field generating the $\srY$-direction of the tangent bundle
of $\srP$. Recall that it acts on functions by the derivation
$\lambda \partial /\partial t$.

Consider our $\srR _{\srX}$-module $\srE$. We look at $g_+(\srE )$
which is an $\srR _{\srP}$-module on $\srP$.  It is supported on the
graph $g(\srX )$.

We may identify sheaves on $\srP$ supported on $g(\srX )$, with their
pullbacks to $\srX$ via $g^{-1}$.

This yields on $\srX$ the $g^{-1}\srR _{\srP}$-module
$g^{-1}g_+(\srE)$.

\begin{lemma}
\label{ginvgplus}
We may write
$$
g^{-1}g_+(\srE )=\srE [\partial _t] .
$$
The action of $g^{-1}\srR _{\srP}$ is given by the formulae of
\cite[Equations (3.4.3)]{SaitoMHM}.
\end{lemma}
\begin{proof}
  It is possible to change coordinates on $\srP$ such that the image
  of the graph is just a coordinate hyperplane. We can therefore write
  explicitly the pushforward $g_+$ as follows:
$$
g_+(\srE ) = \bigoplus _{i=0}^{\infty}\partial _t^i\cdot (g_{\ast}(\srE )).
$$
The statement of the lemma is given by pulling back to $\srX$ along $g^{-1}$. 

The action of elements of $\srR _{\srP}$ is easy to calculate in the
changed coordinate system, in particular it means that the actions of
$\partial _t$ and $\Oo _{\srP}$ are just the usual ones.  On the other
hand, one must use a change of variables formula to get the action of
the vector fields on the original factor $\srX$.This is done by Saito
\cite[Theorem 3.4]{SaitoMHM} and Mochizuki \cite[Section
16.1]{MochizukiPure} \cite[12.2.2]{MochizukiWild}.
\end{proof}

\

\noindent
Recall that $V_0\srR _{\srP}\subset \srR _{\srP}$ is the sheaf of
subrings generated by the tangent vector fields tangent to the fiber
$t=0$; it contains in particular the tangent vector field $t\partial
_t$.  And, the $V$-filtration of $g_+(\srE )$ is characterized as the
increasing filtration of this module by finitely generated
sub-$V_0\srR _{\srP}$-modules $V_{b-1}(g_+(\srE ))$, such that on
$$
{\rm Gr}_{b-1}(g_+(\srE ))= V_{b-1}(g_+(\srE ))/
V_{v-\epsilon -1}(g_+(\srE )),
$$
the vector field $t\partial _t$ acts with generalized eigenvalue
$-\lambda b$, that is to say $t\partial _t + \lambda b$ is nilpotent.

\begin{proposition}
\label{vgen1}
Assume $b<1$.
The $V$-filtration of $g_+(\srE )$ is determined as follows:
$V_{b-1}(g_+(\srE ))$ is the $V_0\srR _{\srP}$-submodule of
$g_+(\srE )$ generated by $g_{\ast}(\srE _{b,b})$.
\end{proposition}
\begin{proof}
This restates the first part of Proposition \ref{vdescrip}.
We'll describe the argument for the second part of
that proposition in two steps,
in \ref{vgen2} and \ref{vgen3} below. 
\end{proof}

Let $q:\srP \rightarrow \srX$ denote the projection.
We have an inclusion of sheaves of rings
$$
q^{-1}\srR _{\srX} \subset V_0\srR _{\srP}.
$$
Denote by $s$ the section of $V_0\srR_{\srP}$ corresponding
to $t\partial_t$. We obtain the sheaf of rings
$$
q^{-1}\srR _{\srX}[s] \subset V_0\srR _{\srP}.
$$
Note that $g^{-1}q^{-1}\srR _{\srX}=\srR_{\srX}$.
This gives an inclusion of sheaves of rings on $\srX$,
$$
\srR_{\srX}[s] \subset g^{-1}V_0\srR _{\srP}.
$$
We have $g^{-1}g_{\ast}(\srE )=\srE$ and it has a
natural map to $g^{-1}g_+(\srE )$ corresponding to
the inclusion of the degree $0$ part of
$\srE [\partial _t]$ in the expression of Lemma
\ref{ginvgplus}.

A basic fact for our calculations is the following lemma.

\begin{lemma}
Given a collection of sections of
$g^{-1}g_{\ast}(\srE )$, the $g^{-1}V_0\srR_{\srP}$-submodule
of $g^{-1}g_+(\srE )$ that they generate is the same as the
$\srR_{\srX}[s]$-submodule they generate.
\end{lemma}
\begin{proof}
Let $x,y$ be the coordinates on $X$ with $g(x,y)=xy$. 
A section of the ring $g^{-1}V_0\srR_{\srP}$
may be written as
$$
r=\sum_{i,j,k} (t\partial _t)^i\partial _x^j \partial _y^k f_{ijk}(t,x,y).
$$
Setting $a_{ijk}(x,y):= f_{ijk}(xy,x,y)$ we have
(for some function $u_{ijk}$)
$$
a_{ijk}(x,y)-f_{ijk}(t,x,y)=u_{ijk}(t,x,y)(t-xy).
$$
If $e$ is a section of $g^{-1}g_{\ast}(\srE )$ then 
$(t-xy)e=0$. Hence 
$f_{ijk}(t,x,y)\cdot e = a_{ijk}(x,y)\cdot e$,
therefore
$$
r\cdot e = \sum_{i,j,k} (t\partial _t)^i\partial _x^j \partial _y^k a_{ijk}(x,y)\cdot e .
$$
This expression is in the submodule generated by $e$ under
the action of $\srR_{\srX}[s]$.
\end{proof}

\

\

\begin{corollary}
\label{vgen2}
The $V$-filtration of $g_+(\srE )$ is determined in negative degrees as
follows: for $b<1$,
$$
g^{-1}V_{b-1}g_+(\srE )\subset g^{-1}g_+(\srE )=\srE [\partial _t]
$$
is the $\srR_{\srX}[s]$-submodule generated by
$g^{-1}g_{\ast}(\srE  _{b,b})=
{\srE}_{b,b}$.
\end{corollary}
\begin{proof}
Combine the previous proposition and lemma. 
\end{proof}

\

\noindent
We now look more closely at the action of
$\srR_{\srX}[s]$. As was explained in
\cite[Section 3.4]{Sabbah},
it is useful to localize by inverting $t=xy$. Sabbah denoted the
localization with a tilde and we conserve that notation.

Notice that $xy$ never acts invertibly on the module $\srE$; this is
somewhat different from the case of $\Dd$-modules where we can have a
holonomic module in which $xy$ is invertible. Because of
multiplication of the derivatives by $\lambda$, a finitely generated
$\srR_{\srX}$-module will not have $xy$ acting invertibly. However, as
Sabbah points out, one may make this localization if we are interested
in the $V$-filtration.

Let $\widetilde{\srE}:= \srE [(xy)^{-1}]$.  Then
$$
g_+(\widetilde{\srE})= g_+(\srE ) [t^{-1}].
$$
Notice that $\srE \subset \widetilde{\srE}$ is a submodule so
$$
g_+(\srE )\subset g_+(\widetilde{\srE})
$$
is a sub-$\srR_{\srP}$-module.

We also have the following relationship with the parabolic structure:
$$
\widetilde{\srE} = \bigcup _{a,b} \srE _{a,b}.
$$

\begin{proposition}
  The (non-finite type) $\srR_{\srP}$-module $g_+(\widetilde{\srE})$
  also admits a $V$-filtration characterized by the same properties as
  in the holonomic case.  For $b<1$ we have
$$
V_{b-1}g_+(\srE ) = V_{b-1}g_+(\widetilde{\srE}).
$$
We have
$$
g^{-1}g_+(\widetilde{\srE}) =\widetilde{\srE} [s],
$$
although one must be careful that the action of $\srR_{\srX}$ includes
terms in $s$ as referred to in Lemma \ref{ginvgplus}.

For any $b$, $g^{-1}V_{b-1}g_+(\widetilde{\srE})$ is the
sub-$\srR_{\srX}[s]$-module of $\widetilde{\srE} [s]$
generated by $\srE  _{b,b}$.
\end{proposition}
\begin{proof}
  This is the same as the previous statements, noting that $t=xy$ is
  invertible on $\widetilde{\srE}$, allowing us to go between
  $\partial_t$ and $s=t\partial _t$ by multiplying or dividing by $t$.
\end{proof}

\

\noindent
We can now restrict to $\lambda =0$. This restriction is a quotient,
dividing everything by the submodules generated by $\lambda$. The
other operations that have intervened above, namely taking
localization and taking submodules generated by something, are all of
the form colimits. Therefore, all these operations commute.

Recall that $E$ denotes the $R_{X,0}$-module over $X$.
It is the restriction of $\srE$ to $\lambda =0$. We have
$g_+(E)$ the restriction of $g_+(\srE )$, with the
formula
$$
g^{-1}(g_+(E))=E[\partial _t].
$$
Let $\widetilde{E}$ be obtained by inverting $xy$ on $E$.
Again,
it is the restriction of $\widetilde{\srE}$ to $\lambda =0$. We have
$$
g^{-1}(g_+(\widetilde{E})) = \widetilde{E}[s]
$$
where as before $s=t\partial _t$.

\begin{proposition}
The $V$-filtrations $V_{b-1}g_+(E)$ (resp.
$V_{b-1}g_+(\widetilde{E})$) are the restrictions to $\lambda =0$ of the
$V_{b-1}g_+(\srE )$ (resp.
$V_{b-1}g_+(\widetilde{\srE})$). We have
$$
V_{b-1}(g_+(E))=V_{b-1}(g_+(\widetilde{E}))
\mbox{ for } b<1
$$
and for any $b$, $V_{b-1}(g_+(\widetilde{E}))$ is the submodule of
$\widetilde{E}[s]$ generated by $g_{\ast}E_{b,b}$ under the action of
the ring $R_{X,0}[s]$.
\end{proposition}

\subsection{Towards explicit calculations}

After the discussion from above, we are in the following situation. We
have that $X$ is a product of disks, and $D$ is the union of the two
coordinate lines $D_1$ and $D_2$. Use coordinates $(x,y)$ on $X$, with
$D_1$ given by $y=0$ and $D_2$ by $x=0$.

From now on we identify sheaves on $P$ supported along
$g(X)$ with sheaves on $X$ via $g^{-1}$ and $g_{\ast}$.
We work with sheaves on $X$.

We are given a module $\widetilde{E}$ over the ring of functions with
poles along $D$,
call it $\Oo _X[x^{-1},y^{-1}]$. It has submodules denoted
by
$E_{a,b}$ which are locally free
over $\Oo _X$
and give $\widetilde{E}$ when localized by inverting $x$ and $y$.
Define
$$
\psi _{b,b} := E_{b,b}/ E_{b-\epsilon , b-\epsilon }.
$$
The Higgs field on $\widetilde{E}$ is given by two sections
$\varphi_x$ and $\varphi _y$ of $End(\widetilde{E})$, having
logarithmic poles with respect to each of the submodules $E_{a,b}$. In
particular, we have
$$
\varphi_x(E_{b,b})\subset x^{-1}E_{b,b}, \;\;\;
\varphi_y(E_{b,b})\subset y^{-1}E_{b,b}
$$
and the same for $b-\epsilon$. It follows that $x\varphi_x$ and
$y\varphi_y$ act on $\psi _{b,b}$.

We have
$$
V_{b-1}(g_+\widetilde{E})\subset
g_+(\widetilde{E})=\widetilde{E}[s]
$$
is the submodule generated by $E_{b,b}$ under the operations of $s$
and the basis vector fields $\partial _x$ and $\partial _y$.  These
act according to the formulae of \cite[Equations (3.4.3)]{SaitoMHM} as
explained above.  In the present situation we are restricting to
$\lambda =0$ so there is no differentiation: the actions of all the
vector fields commute and they act trivially on functions. Recall,
however, that in order to characterize the $V$-filtration one needs to
use the full $\srR_{\srX}$-module $g_+(\widetilde{\srE})$ with
differentiation. After the characterization as a submodule generated
by $\srE _{b,b}$ we can then restrict to $\lambda = 0$ and have the
same characterization there.

In view of commutativity, the
actions of the vector fields are easier to write down: 
$\partial _x$ and $\partial _y$ act
respectively by endomorphisms
$$
A_x:= \varphi _x +s/x , \mbox{  and  } A_y:= \varphi _y + s/y.
$$
The same statement holds
for
$V_{b-\epsilon -1}$. 

By definition the module of nearby cycles is
$$
\Psi _{b-1} :=V_{b-1}(\widetilde{E}[s])/
V_{b-\epsilon -1}(\widetilde{E}[s]).
$$
Notice that we keep here the subscript $b-1$ in order to conform to
the usual practice in the theory of $\Dd$-modules. For brevity on the
other hand we used the notation $\psi _{b,b}$ without the $-1$'s. This
shouldn't cause too much confusion as either expression may be
considered as some kind of notation.

The following lemma completes our review of the second
part of Proposition \ref{vdescrip}, 
which we recall is due to Saito
and Mochizuki. We felt it would be useful to give here a proof adapted specifically to the Dolbeault case. 

\begin{lemma}
\label{vgen3}
The submodule 
$V_{b-1}(\widetilde{E}[s])\subset\widetilde{E}[s]$ 
can also be characterized as the submodule generated 
by $E_{b,b}$ under just $\Oo _X$ and
the operations $A_x$ and $A_y$.
\end{lemma}
\begin{proof}
Recall that $E_{b,b}$ is preserved by the logarithmic
Higgs field, so it is stable under the operations
$x\varphi _x$ and $y\varphi _y$. The same is
therefore true of $E_{b,b}[s]$. Thus
$xA_x$ also preserves $E_{b,b}[s]$ and 
$$
s^k E_{b,b} = (x\varphi_x -xA_x)^kE_{b,b} \subset E_{b,b}[s].
$$
Now $V_{b-1}(\widetilde{E}[s])$ is the sub-$\Oo _X$-module of 
$\widetilde{E}[s]$ generated by 
$E_{b,b}$ under the operations $s,A_x,A_y$. 
Hence, it is also 
the submodule generated by $E_{b,b}[s]$ under the
operations $A_x,A_y$. From the previous formula we see
that it is generated from $E_{b,b}$ by 
$A_x,A_y$. 
\end{proof}

\

\subsection{The tensor product formula}

Let the ring $\Oo _X [u,v]$ act on $\widetilde{E}[s]$ by setting the
action of $u$ equal to $A_x$ and the action of $v$ equal to $A_y$.
Define
$$
\varphi _{\log} := x\varphi _x - y\varphi _y,
$$
acting on $\widetilde{E}$, hence also on $\widetilde{E}[s]$;
and consider the endomorphism of $\widetilde{E}[s]$ given by
$$
A_{\log} := xA_x -yA_y =xu-yv.
$$
The action of this endomorphism is equal to the
action of  $\varphi _{\log}$.
In other words,
the element $xu-yv -\varphi _{\log}$ acts by $0$ on $\widetilde{E}[s]$.

Let $w:= xu-yv$ so we have a map $\Oo _X[w] \rightarrow \Oo _X[u,v]$,
and we note that for any $\Oo _X[w]$-module $M$ we have
$$
M\otimes _{\Oo _X[w]} \Oo _X[u,v] =
\frac{M [u,v]}{ (xu-yv - w)M[u,v]}.
$$
If furthermore $M$ is an $\Oo _D$-module, equivalent to saying that
$xyM=0$, then we may also write
$$
M\otimes _{\Oo _D[w]} \Oo _D[u,v] =
\frac{M [u,v]}{ (xu-yv - w)M[u,v]}.
$$
We clearly have a map $E_{b,b}\rightarrow
V_{b-1}(\widetilde{E}[s])$ giving
$$
\psi _{b,b} \rightarrow \Psi _{b-1},
$$
and on the other hand the operations $A_x$ and $A_y$ on 
$\Psi _{b-1}$ give an $\Oo _D[u,v]$-module structure.
The action of $xu-yv$ coincides with the action of $\varphi _{\log}$ on
$\psi_{b,b}$,
so we obtain a natural map
\begin{equation}
\label{tpf}
\psi _{b,b}\otimes _{\Oo _ D[w]}\Oo _D[u,v]\rightarrow 
\Psi _{b-1}.
\end{equation}
The following proposition gives a formula  for
$\Psi _{b-1}$:

\begin{theorem}[\sc{Tensor product formula}]
\label{tensorproductformula}
The map \eqref{tpf} is an isomorphism.
\end{theorem}
\begin{proof}
  From the above discussion, we may also say that we would like to
  show that the map
\begin{equation}
\label{tpf2}
\frac{\psi _{b,b}[u,v]}{
(xu-yv-\varphi _{\log})\psi _{b,b}[u,v]}
\rightarrow \Psi _{b-1}.
\end{equation}
is an isomorphism.

Notice that by the definition of $V_{b-1}$, the map
$$
E_{b,b}[u,v] \rightarrow V_{b-1}(\widetilde{E}[s])
$$
is surjective. Hence it follows that the map \eqref{tpf2}
is surjective. We would like to show that it is injective.

The right hand side of \eqref{tpf2} is
$$
\frac{E_{b,b}[u,v]}{E_{b-\epsilon , b-\epsilon }[u,v] +
  (xu-yv-\varphi_{\log} )E_{b,b}[u,v]}.
$$
We need to show that if $a\in E_{b,b}[u,v]$ and the image of $a$ in
$\widetilde{E}[s]$ is in $V_{b-\epsilon -1}(\widetilde{E}[s])$, that
is to say if the image of $a$ is in the image of $E_{b-\epsilon ,
  b-\epsilon }[u,v]$, then
$$
a\in E_{b-\epsilon , b-\epsilon }[u,v] + (xu-yv-\varphi_{\log} )E_{b,b}[u,v].
$$

Our first claim is that the map
\begin{equation}
\label{claimmap}
\frac{\widetilde{E}[u,v]}{(xu-yv-\varphi _{\log})
\cdot \widetilde{E}[u,v]}
\rightarrow \widetilde{E}[s]
\end{equation}
is injective. In this situation $x$ and $y$ are invertible and we can
write $s=xu-x\varphi_x = yv-y\varphi_y$.  Consider the following
change of variables: put $u':=x( u-\varphi _x)$ and $v':= y(v-
\varphi_y)$ acting on $\widetilde{E}$.  These formulas determine a map
$$
\widetilde{E}[u',v']\rightarrow \widetilde{E}[u,v]
$$
by sending
$$
\sum (u')^i(v')^j e_{ij} \mbox{  to  }
\sum (x( u-\varphi _x))^i (y(v- \varphi _y))^j e_{ij} .
$$
This map is an isomorphism. The composed map
$$
\widetilde{E}[u',v']\rightarrow \widetilde{E}[u,v]\rightarrow \widetilde{E}[s]
$$
maps the action of $u'$ to the action of $x(A_x-\varphi _x) =s$ and
also the action of $v'$ to the action of $y(A_y-\varphi _y) =s$, that
is to say it sends $u'$ and $v'$ to $s$. The kernel is therefore the
submodule $(u'-v')\widetilde{E}[u,v]$. Transporting back by the
isomorphism, we see that the kernel of $\widetilde{E}[u,v]\rightarrow
\widetilde{E}[s]$ is generated by $x(u-\varphi _x) - y(v-\varphi _y) =
xu-yv-\varphi _{\log}$.  This proves the claim that \eqref{claimmap}
is injective.

We next note that \eqref{claimmap} is surjective. Indeed, we saw above
that its image is the same as the image of
$\widetilde{E}[u',v']\rightarrow \widetilde{E}[s]$ and this map is
visibly surjective since $u'$ and $v'$ map to $s$.

We have shown that \eqref{claimmap} is an isomorphism, in particular
we may replace $\widetilde{E}[s]$ by the left hand side of this map.
Therefore we may consider the map
\begin{equation}
\label{bmap}
E_{b,b}[u,v] \rightarrow \widetilde{E}[u,v]/(xu-yv-\varphi _{\log} )
\widetilde{E}[u,v] \cong \widetilde{E}[s].
\end{equation}
The image is $V_{b-1}(\widetilde{E}[s])$.
Our second claim is that the kernel of \eqref{bmap}
is equal to $(xu-yv-\varphi _{\log} )E_{b,b}[u,v]$.
This claim is equivalent to the statement
$$
E_{b,b}[u,v] \cap (xu-yv-\varphi _{\log} )\widetilde{E}[u,v] =
(xu-yv-\varphi _{\log} )E_{b,b}[u,v].
$$
To prove this statement, choose a basis for $E_{b,b}$ so we may write
$E_{b,b}\cong \Oo _X ^r$.  Write $\varphi _{\log} = B_{ij}$ in terms
of this basis, with $B_{ij}\in \Oo _X$.  In this notation,
$\widetilde{E}[u,v]=\Oo _X[x^{-1},y^{-1}][u,v]^r$. Suppose we have a
vector $(f_i)$ here, so $f_i\in \Oo _X[x^{-1},y^{-1}][u,v]$. Suppose
that
$$
(xu-yv)f_i + \sum _j B_{ij} f_j \in \Oo _X[u,v].
$$
Write
$$
f_i = x^{-a}y^{-b}\sum _{k,l}g^{kl}_iu^kv^l.
$$
Assume that $a,b$ are chosen so that $g^{kl}_i\in \Oo _X$ but at least
one of them is nonzero along each of the components of $D$, i.e. $a$
and $b$ are the smallest possible. Suppose one of $a$ or $b$ is $>0$,
that is to say some $f_i$ has a pole. We will obtain a
contradiction. Suppose for example $a>0$. Then restrict the $g^{kl}_i$
to $(x=0)$, and consider the terms of maximal $k+l$, call $\hat{g}_i$
the sum of these terms of the form $g^{kl}_iu^kv^l$ .  After
multiplying by $(xu-yv)$ we get terms with strictly bigger degree in
$u,v$, but the restriction of the term $xu\hat{g}_i$ to $(x=0)$
vanishes; the restriction of $yv\hat{g}_i$ is nonzero. But it has
strictly bigger degre in $u,v$ than any possible term in the
restriction to $(x=0)$ of $\sum _jB_{ij}g_j$. Here $g_i:= \sum
_{k,l}g^{kl}_iu^kv^l$.  It follows that the restriction of
$$
(xu-yv)g_i + \sum _j B_{ij} g_j
$$
to $(x=0)$ is nonvanishing, but since $f_i=x^{-a}y^{-b}g_i$ with $a>0$
this contradicts the hypothesis that $(xu-yv)f_i + \sum _j B_{ij} f_j
\in \Oo _X[u,v]$. We conclude that $a\leq 0$ and similarly $b\leq 0$,
in other words our section $(f_i)$ is in
$\Oo_X^r[u,v]=E_{b,b}[u,v]$. This proves the second claim.

The corollary of the second claim is the formula
$$
V_{b-1}(E[s]) = E_{b,b}[u,v]/(xu-yv-\varphi _{\log})E_{b,b}[u,v].
$$
This may also be written as
$$
V_{b-1}(\widetilde{E}[s])=E_{b,b}\otimes _{\Oo _X[w]} \Oo _X[u,v].
$$
The same holds for $b+\epsilon$.
But now the statement of the proposition follows: we have
a right exact sequence
$$
E_{b-\epsilon, b-\epsilon}\rightarrow E_{b,b}\rightarrow \psi_{b,b}\rightarrow 0,
$$
and tensor product is
right exact, so we get
the exact sequence
$$
E_{b-\epsilon, b+\epsilon}\otimes _{\Oo _X[w]} \Oo _X[u,v]
\rightarrow E_{b,b}\otimes _{\Oo _X[w]} \Oo _X[u,v]
\rightarrow \psi_{b,b}\otimes _{\Oo _X[w]} \Oo _X[u,v]\rightarrow 0,
$$
In view of the previous formula for $V_{b-1}(E[s])$ and the same for $b-\epsilon$, this may be written
$$
V_{b-\epsilon -1}(\widetilde{E}[s])\rightarrow 
V_{b-1}(\widetilde{E}[s])
\rightarrow \psi_{b,b}\otimes _{\Oo _X[w]} \Oo _X[u,v]\rightarrow 0,
$$
in other words
$$
\Psi _{b-1} = V_{b-1}(\widetilde{E}[s])/ 
V_{b-\epsilon -1}(E[s]) = \psi_{b,b}\otimes _{\Oo _X[w]} \Oo _X[u,v].
$$
This proves the tensor product formula of the proposition.
\end{proof}

\subsection{Consequence for the Koszul complexes}

In general suppose $Z$ is a scheme or analytic space and $\Vv$ is a
vector bundle. Let $\text{Sym}^{*}(\Vv ^{\vee})$ denote the symmetric
algebra on the dual vector bundle $\Vv^{\vee}$.  A
$\text{Sym}^{*}(\Vv ^{\vee})$-module coherent over $Z$ is the same
thing as a coherent sheaf $\Ff$ on $Z$ together with a morphism
$$
\phi : \Ff \rightarrow \Ff \otimes _{\Oo _Z} \Vv 
$$
such that the induced map $\phi \wedge \phi : \Ff \rightarrow \Ff
\otimes _{\Oo _Z} \bigwedge ^2\Vv$ is zero. The case of a Higgs bundle
is when $\Vv = \Omega ^1_Z$. In our case, the structure of logarithmic
Higgs bundles of $E_{\beta}$ over $Z=X$ corresponds to $\Vv =
\Omega^1_{X/Y}(\log D_V)$.

In general given $(\Vv , \Ff , \phi )$ as above we get the {\em Koszul
  complex}
\begin{equation}
\label{koszul}
{\bf Kosz}(\Vv , \Ff , \phi ):= 
\xymatrix@R-2pc{ \Big[ 
\ldots \ar[r] &  \Ff \otimes _{\Oo _Z}\bigwedge ^i\Vv \ar[r]^-{\wedge
  \phi} & 
\Ff \otimes _{\Oo _Z}\bigwedge ^{i+1}\Vv
\ar[r] & \ldots \Big]. \\ 
& i & i+1 & 
}
\end{equation}
If we are working with modules over a ring rather than quasicoherent
sheaves on a scheme or space, we shall use the same notation.

Locally at a crossing point of vertical divisors we have
$$
Q(b)=
{\rm DOL}^{\rm par}_{L^2}(X/Y,E_{\alpha (b)}) 
={\bf Kosz}(\Omega ^1_{X/Y}(\log D), E_{b,b}, \varphi ).
$$
The Koszul complex is compatible with quotients in the module
variable, since tensoring with the locally free sheaves $\bigwedge
^i\Vv$ is exact.  It follows that
\begin{equation}
\label{qbkoszul}
Q(b)/Q(b-\epsilon ) = 
{\bf Kosz}(\Omega ^1_{X/Y}(\log D), \psi _{b,b}, \varphi )
\end{equation}
locally near a crossing point of two vertical divisor components
because $\psi _{b,b}=E_{b,b}/E_{b-\epsilon , b-\epsilon }$.

On the other hand, 
$$
K(b) = {\bf Kosz}(\Omega ^1_{P/Y}, V_{b-1}(g_+(\srE )), -- ), 
$$
so again 
\begin{equation}
\label{kbkoszul}
K(b)/K(b-\epsilon ) = 
{\bf Kosz}(\Omega ^1_{P/Y}, \Psi _{b-1}, -- ).
\end{equation}
We didn't give a name to the Higgs field for the sheaf 
$\Psi _{b-1}$ on $P/Y$.

The $\psi _{b,b}$ and $\Psi _{b-1}$
are $\Oo _D$-modules supported on $D\subset X$. Therefore we 
may consider the Koszul complexes as being constructed over
the divisor $Z:= D$, in the 
neighborhood of a normal crossing point of the vertical divisor. 

In view of the tensor product formula of the previous subsection, let
us consider an $\Oo _D$-module $M$ with action of an endomorphism
$\varphi$, and set
$$
N:= M[u,v]/(xu-yv-\varphi )M.
$$
In the previous notations, $M=\psi _{b,b}$ and $N=\Psi _{b-1}$. We
have shortened $\varphi _{\log}$ to just $\varphi$ here.

Over $D$ we have two distinct vector bundles that can be used to
define Koszul complexes. The first $V_1$ is the restriction to $D$ of
$\Omega ^1_{X/Y}(\log D_V)$. It has rank $1$ and a local generator for
$V_1^{\ast}$ is $x\partial _x-y\partial_y$.  The operator $\varphi =
\varphi _{\log}$ is an action of $S(V_1^{\ast})$ on $M$ and we get the
Koszul complex ${\bf Kosz}(V_1,M,\varphi )$.

On the other hand, $V_2$ is the vector bundle $\Omega ^1_X$ restricted
to $D$, of rank $2$ with local generators $\partial _x$ and $\partial
_y$. We have $S(V_2^{\ast})=\Oo _D[u,v]$, and the $\Oo _D[u,v]$-module
structure of $N$ corresponds to a map
$$
\psi : N\rightarrow N\otimes V_2.
$$
We get the Koszul complex ${\bf Kosz}(V_2,N,\psi )$.

The map of Koszul complexes that we would like to consider
is the vertical map between the two horizontal complexes:
\begin{equation}
\label{vertmap}
\begin{array}{ccccccc}
{\bf Kosz}(V_1,M,\varphi ) & = & 
& & M& \stackrel{\varphi }{\rightarrow} & M \\
\downarrow & & & & \downarrow & & \downarrow \\
{\bf Kosz}(V_2,N,\psi ) & = & N & \rightarrow & N^{\oplus 2} & \rightarrow & N
\end{array}
\end{equation}
where the differentials on the bottom are $n\mapsto (vn,un)$ and
$(n,n')\mapsto un-vn'$.  The middle vertical map sends $m$ to
$(xm,ym)$ and the right vertical map is just the standard
inclusion. In $N$ we have $(xu-yv)m = \varphi m$ so the square
commutes.

The upper complex is the Koszul complex for $M$ with the action of
$\varphi$, whereas the lower complex is the Koszul complex for $N$
with its action of $u$ and $v$.

\begin{theorem}[\sc{Koszul quasiisomorphism theorem}] \label{qi}
  Suppose $M$ is an $\Oo _D$-module with action of $\varphi :
  M\rightarrow M$ as above.  We suppose that $M$ has no elements
  annihilated by both $x$ and $y$, for example suppose it has no
  sections supported in dimension $0$.  Then the vertical map of
  complexes \eqref{vertmap} is a quasiisomorphism.
\end{theorem}
\begin{proof}
  Filter $N$ by submodules $N_{\leq a}$ where $N_{\leq a}$ is the
  image of the polynomials of degree $\leq a$ in $u,v$, under the
  surjection $M[u,v]\rightarrow N$. Each Koszul complex is turned into
  a filtered complex by shifting the filtrations so that the
  differential precisely preserves the filtrations. To be precise, let
  $C^0\rightarrow C^1\rightarrow C^2$ denote the lower complex ${\bf
    Kosz}(V_2,N,\psi )$ and define the increasing filtration
  $U_{\cdot}C^{\cdot}$ by
$$
U_aC^0 := N_{\leq a},\;\;\; U_aC^1:= (N_{\leq a+1})^{\oplus 2},\;\;\;
U_aC^2:= N_{\leq a+2}.
$$
Let $D^1\rightarrow D^2$ denote the upper complex 
${\bf Kosz}(V_1,M,\varphi )$ and define
$$
U_{-2}D^1:=0, \;\;\; U_{-1}D^1:= M,
$$
and
$$
U_{-3}D^2:=0, \;\;\; U_{-2}D^2:= M.
$$
The associated-graded of the complex $D^{\cdot}$ is
$$
{\rm Gr}_{-1}(D^{\cdot})= 0 \rightarrow M \rightarrow 0,
$$
$$
{\rm Gr}_{-2}(D^{\cdot})= 0 \rightarrow 0 \rightarrow M.
$$
In all, the associated-graded complex is isomorphic to $D^{\cdot}$
itself but with zero as differential.

We claim that the associated-graded of the complex $C^{\cdot}$ is
isomorphic to the complex
$$
\tilde{N}\rightarrow \tilde{N}^{\oplus 2} \rightarrow \tilde{N}
$$
where
$$
\tilde{N}:= M[u,v]/(xu-yv)M[u,v]
$$
and the differentials are defined in the same way as before.

It suffices to see that on each piece, let us say on $C^0$ for example.

Consider the exact sequence
$$
M[u,v] \stackrel{xu-yv-\varphi}{\longrightarrow} M[u,v] \rightarrow N
\rightarrow 0.
$$
Define filtrations on the pieces by setting $U_aM[u,v]$ equal to the
polynomials of degree $\leq a$ in $u,v$, in the middle; on the left,
shift by one, using the filtration $U'_a:= U_{a-1}$. On the right
$N_{\leq a}$ is the image of $U_a$ from the middle.  We claim that the
exact sequence of associated-graded pieces is still exact. On the
right, the map is clearly surjective because of the definition of the
filtration on $N$. In the middle, suppose we have an element $f\in
U_aM[u,v]$ which maps to zero in $N_{\leq a}/N_{\leq a-1}$. This means
that $f\in (xu-yv-\varphi) M[u,v] + U_{a-1}M[u,v]$. We would like to
show that $f\in (xu-yv-\varphi) U_{a-1}M[u,v]$ (it is somewhat similar
to the proof of the tensor formula (Proposition
\ref{tensorproductformula}) in the previous section). Write
$$
f=(xu-yv-\varphi) g + h
$$
with $g\in M[u,v]$ and $h\in U_{a-1}M[u,v]$. Suppose $g\in U_cM[u,v]$
and that $c$ is the smallest such, so the projection of $g$ in
$U_c/U_{c-1}$ is nonzero. We assume $c>a-1$ and would like to deduce a
contradiction. Write
$$
g=\sum_{k+l\leq c} g^{kl}u^kv^l
$$
with $g^{kl}\in M$. Some $g^{kl}$ is nonzero for a $k+l=c$.

Our situation is that $xyM=0$, and if $xm=0$ and $ym=0$ then $m=0$. It
follows that the map
$$
M\rightarrow (M/xM)/\mbox{tors} \oplus (M/yM)/\mbox{tors}
$$
is injective, where ``$\mbox{tors}$'' denotes the $x$ and $y$
torsion. Indeed, if $m$ is such that its projection into $M/xM$ and
$M/yM$ are torsion, then $x^nm\in yM$ and writing $x^nm=ym'$ we have
$x^{n+1}m=0$; similarly in the other direction we get some
$y^{n+1}m=0$, but then $m=0$ by the hypothesis.

We may therefore assume for example that the projection of some
$g^{kl}$ into $M_1:=(M/xM)/\mbox{tors}$ is nonzero.

Now the projection of $f$ into
$$
U_{c+1}M_1[u,v] / U_cM_1[u,v]
$$
(extending the notation $U_{\cdot}$ in the natural way) is equal to
the projection of $-yvg$ here.  Indeed, $xug$ projects to zero and
also $h$ projects to zero because it is in $U_{a-1}$ and $a-1<c$.
However, the projection of $yg^{kl}$ is nonzero because, by hypothesis
$M_1$ contains no elements annihilated by $y$.  Therefore, the
projection of $f$ into the above graded quotient of level $c+1$ is
nonzero. This contradicts the hypothesis that $f\in U_aM[u,v]$ but
$c+1>a$.

This proves the claim, which says that the sequence of
associated-graded pieces is still exact. This claim says, in other
terms, that
$$
{\rm Gr}(C^0)\cong \tilde{N}.
$$
Similarly for the other terms of the complex with the appropriate
shift of indices. We obtain the statement that the complex ${\rm
  Gr}(C^{\cdot})$ is just the same Koszul complex but for the module
$\tilde{N}$ constructed using $\varphi =0$.

Give the upper Koszul complex ${\bf Kosz}(V_1,M,\varphi )$ a
filtration compatibly with the lower one ${\bf Kosz}(V_2,N,\psi )$, so
that the differential vanishes on the associated-graded pieces.  The
associated-graded of this complex is just the same complex but with
zero as differential: $M\stackrel{0}{\rightarrow}M$.  To show a
quasiisomorphism, it suffices to show a filtered quasiisomorphism.

We have seen above that the associated-graded complexes are the same
ones, but for the endomorphism $\varphi = 0$. This reduces the theorem
to the case $\varphi = 0$. That will be the subject of the
calculations in the next section.

Note that we need to consider all possible module types for $M$. This
amounts to looking at $M=\Oo _D$ and $M=\Oo _D/x\Oo _D$ and $M=\Oo
_D/y\Oo _D$.  These calculations, done in the next section, will
complete the proof of the theorem.
\end{proof}

\

\subsection{Calculation}

We now finish the calculations needed for the above proof.  Since the
question is local, we can assume we are in a simplified global
situation of an affine variety consisting of two crossed lines in
$\aaaa^2$. Thus, work with the coordinate ring:
\[
A := \cc [x,y]/(xy)
\]
of our variety $D \subset \aaaa^2$,
and the algebra over it:
\[
N := A[u,v]/(xu-vy).
\]
Our module $M$ is just $M=A$.
We set up the Koszul complex $K$ with respect to $u,v$:
\[
N \to N^2 \to N.
\]
The first map sends 1 to $v e_1 + u e_2$,
and the second map sends $e_1$ to $u$ and $e_2$ to $-v$.
The question is to calculate the cohomology of $K$.

We will use the grading:
\[
N = \oplus_{i=0}^\infty  N_i,
\]
where $N_i$ is the $A$-submodule of $N$ involving monomials of degree
$i$ in $u,v$. The Koszul sequence $K$ is now the direct sum of graded
pieces $K_i$:
\[
 0 \to N_{i-1} \to N_{i}^{\oplus 2} \to N_{i+1} \to 0.
\]

\

\begin{proposition}
For $i\geq 1$, the Koszul complex $K$ is exact.
\end{proposition}
\begin{proof}
For exactness on the ends, note that
the two individual maps
\[
 N_{i-1} \to N_{i}
\]
given by multiplication with $u,v$ respectively are both injective.

\noindent
---Hence $N_{i-1} \to N_{i}^{\oplus 2}$ is injective, so
the cohomology on left vanishes.

\noindent
---Every monomial in $N_{i+1}$ is
the product of either $u$ or $v$ with a monomial
in $N_i$, so the cohomology on the right vanishes.

The following lemma says that
$K_i$  has no cohomology in the middle either.

\begin{lemma}
The map
$N_{i-1} \to \ker(N_{i}^{\oplus 2} \to N_{i+1})$
given by multiplication by $(v,u)$
is an isomorphism.
\end{lemma}
\begin{proof}
  The claim is essentially combinatorial - we check it by matching
  where monomials go.  With $z:=xu=yv$, $N_i$ has a monomial basis
  consisting of 4 blocks:
\[
u^i, u^{i-1}v,\dots,v^i
\]
\[
yu^i,zu^{i-1},zu^{i-2}v,\dots,zv^{i-1}, xv^i
\]
\[
y^2u^i, y^3u^i, \dots
\]
\[
x^2v^i, x^3v^i, \dots.
\]
Multiplication by either $u$ or $v$ sends each block to the
corresponding block of $N_{i+1}$, so it suffices to check the claim on
each of the four blocks separately. For each of the first two blocks,
the two multiplication maps are injective, and differ from each other
by a shift of one place, so pairs in the kernel are precisely the
images of basis elements in the corresponding two blocks of
$N_{i-1}$. For basis elements in the third block, multiplication by
$v$ vanishes, while multiplication by $u$ sends them bijectively to
basis elements in the third block of $N_{i+1}$, so the kernel again
consists of the basis elements in the third block of $N_{i-1}$.  More
explicitly, the map in this third block sends $(f,g) \to fu-gv = fu$,
so $(f,g)$ is in the kernel iff $f=0$, iff there is some $h \in
N_{i-1}$ such that $(f,g) = (0,g) = (hv,hu)$, namely $h=g/u$.  The
fourth block is obtained by symmetry.
\end{proof}

This completes the proof of the proposition.
\end{proof}

This leaves the sequences $K_i$ for $i=0$ and $i=-1$. For $i=0$ the
sequence is
\[
 0 \to A^{\oplus 2} \to N_1 \to 0.
\]
This is clearly surjective (everything in $N_1$ is divisible by $u$ or
$v$ or both), and the kernel is the submodule of $A^{\oplus 2}$
generated (over $A$) by $(x,y)$.  So the cohomology is isomorphic to
$A$, in the middle.

Finally, for $i=-1$, the sequence is $0\to 0 \to A$, so the cohomology
is $A$ on the right.

So this confirms that the cohomology is $A$, that is to say $O_D$,
occurring in two places, thus completing the proof of Therem \ref{qi}
in this case.

We need to consider one additional case.
The ring $A$ is, as before,

\[
A := \cc [x,y]/(xy)
\]
but now we also have the $A$-module:

\[
M := A/(x) = \cc [x,y]/(x) = \cc [y]
\]
(on which $x$ acts as 0), and we set:

\[
N:= M[u,v]/(xu-yv)M = M[u,v]/(yv)M = \cc [y,u,v]/(yv)
\]
As before, we set up the Koszul complex $K$ with respect to $u,v$:

\[
N \to N^2 \to N.
\]
We let $e_1, e_2$ be the generators of the two copies of $N$ in the
middle.  The first map sends 1 to $-v e_1 + u e_2$, and the second map
sends $e_1$ to $u$ and $e_2$ to $v$.  The question is to calculate its
cohomology. Again, we use the grading:

\[
N = \oplus_{i=0}^\infty  N_i,
\]
where $N_i$ is the $A$-submodule of $N$ involving monomials of degree
$i$ in $u,v$. The Koszul sequence $K$ is again the direct sum of
graded pieces $K_i$:
\[
 0 \to N_{i-1} \to N_{i}^{\oplus 2} \to N_{i+1} \to 0.
\]
For $i \geq 1$,we see that $K_i$ is exact. More precisely:

\begin{itemize}
\item of the two individual maps
\[
 N_{i-1} \to N_{i}
\]
given by multiplication with $u,v$ respectively, the first is
injective, the second is not.
\item  hence $N_{i-1} \to N_{i}^{\oplus 2}$ is still
injective, so no cohomology on left.
\item every monomial in $N_{i+1}$
is the product of either $u$ or $v$ with a monomial
in $N_i$, so there is no cohomology on the right.
\item by matching where monomials go we see that the kernel
of $N_{i}^{\oplus 2} \to N_{i+1}$ is generated as a vector space by:
$-v u^{i-1} e_1 + u^i e_2, -v^2 u^{i-2} e_1 +v u^{i-1} e_2,\dots, -v^i
e_1 +v^{i-1} u e_2, y u^i e_2, y^2 u^i e_2, y^3 u^i e_2, \dots$, so as
an $M$-module it is generated by $-v u^{i-1} e_1 + u^i e_2, -v^2
u^{i-2} e_1 +v u^{i-1} e_2,\dots, -v^i e_1 +v^{i-1} u e_2$, and all
these generators are clearly in the image of $N_{i-1}$. So no
cohomology in the middle either.
\item This leaves the sequences $K_i$ for $i=0$ and
$i=-1$. For $i=0$ the sequence is
\[
 0 \to M^{\oplus 2} \to N_1 \to 0.
\]
This is clearly surjective (everything in $N_1$ is divisible by $u$ or
$v$), and the kernel is the submodule $y M e_2$ generated (over $M$)
by $y e_2$.  So the cohomology is isomorphic to $M$, in the middle.
\item Finally, for $i=-1$, the sequence is $0\to 0 \to M$,
so the cohomology is $M$ on the right.
\end{itemize}

\

\noindent
This completes the calculation of the Koszul cohomology and shows that
the two complexes are quasi isomorphic as claimed, thus completing the
proof of Theorem \ref{qi}.

\subsection{Completion of proofs}

In this section we go back and see how this result leads
to the statement we were originally looking for.

In the previous section we have completed the proof of the Koszul
quasiisomorphism theorem \ref{qi}.

The Tensor product formula of Theorem \ref{tensorproductformula} tells
us that if we set $M$ equal to the module $\psi _{b,b}$ then
$\Psi_{b-1}$ corresponds to the module denoted by $N$ in Theorem
\ref{qi}.
 
By \eqref{qbkoszul} the Koszul complex for $M$ is the same as the
complex denoted $Q(b)/Q(b-\epsilon )$, on the left of the morphism
\eqref{gru} above Proposition \ref{reduction2}.  Similarly by
\eqref{kbkoszul}, the Koszul complex for $N$ is the same as the
complex denoted $K(b)/K(b-\epsilon )$ on the right of the morphism
\eqref{gru}) above Proposition \ref{reduction2}.

The map of Koszul complexes is the same as the map ${\rm gr}_a(u_0)$
in \eqref{gru} (setting $a=b$).  Therefore, the Koszul
quasiisomorphism theorem \ref{qi} tells us the quasiisomorphism asked
for in Proposition \ref{reduction2}. From that proposition, we now
obtain the proof of Theorem \ref{maincalc}.

Now Corollary \ref{proofcpxqis} gives the proof of Theorem
\ref{cpxqis}, and in turn Proposition \ref{proofmainth} gives the
proof of the main Theorem \ref{mainth}.

\section{Further considerations}

In this section we look in more detail at several aspects.  The first
two subsections provide some details on the proofs of parts 3 and 5 of
Theorem \ref{mainth}.

We start by looking at the absolute Dolbeault complex on $X$ and use it
to define the Gauss-Manin Higgs field on the higher direct
images. Next we look at the analytical aspects of the direct image
harmonic bundle. This includes a sketch of our original strategy for
proving the theorem, involving the study of the family of $L^2$
cohomology spaces.

Then in the last two subsections we explore generalizations to higher
dimensional cases.
 
\subsection{The absolute complex and Gauss-Manin}
\label{globalcomplex}

Recall that the Gauss-Manin connection on the relative algebraic de
Rham cohomology of a vector bundle with connection, comes from an
exact sequence of complexes. Whereas the higher direct image bundle is
defined using the relative de Rham complex, the exact sequence needed
to define the Gauss-Manin connection uses the absolute de Rham
complex. Similarly, in our case, in order to construct the Higgs field
on the parabolic bundle $\underline{F}^i$, we should look at the
absolute Dolbeault complex on $X$.

In order to prepare for the generalized situation to be considered in
subsection \ref{hdb} later, let us consider the case when $(X,D)$ is
of arbitrary dimension and let $f : (X,D)\to (Y,Q)$ be a smooth split
semistable family of curves (see Definition~\ref{splitsemistable}).
In this case we define the absolute $L^{2}$ Dolbeault complex as
follows. For every $i \geq 0 $ we have a short exact sequences of
forms on $X$:
\[
\xymatrix@1{ 0 \ar[r] & f^{*} \Omega^{i}_{Y}(\log Q)  \ar[r] &
  \Omega^{i}_{X}(\log D) \ar[r] & \Omega^{1}_{X/Y}(\log \, D)\otimes
  f^{*}\Omega^{i-1}_{Y}(\log Q) \ar[r] & 0,
}
\]
where
$\Omega^{1}_{X/Y}(\log D) =
\Omega^{1}_{X}(\log \, D)/f^{*}\Omega^{1}_{Y}(\log Q)$ is the
relative logarithmic dualizing sheaf of $f$. Consider 
the tensor product of this sequence with $W_{0}(H,E_{\alpha(a)})$:
\[
\xymatrix@R-1pc
             { 0 \ar[d] \\ W_{0}(H,E_{\alpha(a)})\otimes f^{*} \Omega^{i}_{Y}(\log Q) \ar[d]
  \\ W_{0}(H,E_{\alpha(a)}) \otimes \Omega^{i}_{X}(\log D) \ar[d] \\
  W_{0}(H,E_{\alpha(a)})\otimes \Omega^{1}_{X/Y}(\log D)
  \otimes f^{*}\Omega^{i-1}_{Y}(\log Q) \ar[d] \\ 0
  }
\]
and pull it back by the natural map
\[
W_{-2}(H,E_{\alpha(a)})\otimes \Omega^{1}_{X/Y}(\log D)
\otimes f^{*}\Omega^{i-1}_{Y}(\log Q) \to
W_{0}(H,E_{\alpha(a)})
\otimes \Omega^{1}_{X/Y}(\log \, D)\otimes f^{*}\Omega^{i-1}_{Y}(\log Q).
\]
This gives us a new object
\[
W_{-2,0}\left(H,E_{\alpha(a)}\otimes \Omega^{i}_{X}(\log \, D)\right)
\]
which fits into an extension:
{\small
\[
\xymatrix@R-1pc@C-0.5pc{
  0 \ar[d] & 0 \ar[d] \\
  W_{0}(H,E_{\alpha(a)})\otimes f^{*} \Omega^{i}_{Y}(\log Q) \ar[d] \ar@{=}[r] &
  W_{0}(H,E_{\alpha(a)})\otimes f^{*} \Omega^{i}_{Y}(\log Q) \ar[d]  \\
  W_{-2,0}\left(H,E_{\alpha(a)}\otimes \Omega^{i}_{X}(\log D)\right)
  \ar[d] \ar@{^{(}->}[r] &
  W_{0}(H,E_{\alpha(a)}) \otimes \Omega^{i}_{X}(\log D) \ar[d] \\
  W_{-2}(H,E_{\alpha(a)})\otimes \Omega^{1}_{X/Y}(\log D)
\otimes f^{*}\Omega^{i-1}_{Y}(\log Q)  \ar[r] \ar[d] &
  W_{0}(H,E_{\alpha(a)})\otimes \Omega^{1}_{X/Y}(\log D)
  \otimes f^{*}\Omega^{i-1}_{Y}(\log Q) \ar[d] \\
  0 & 0
}
\]
}
and we get a well defined absolute complex
\[
\Dol(X,E_{\alpha(a)})
 := 
\left[ \begin{array}{c}
    W_{0}(H,E_{\alpha(a)}) \\[+0.3pc]
    \downarrow \wedge \varphi \\[+0.3pc]
    W_{-2,0}\left( H, E_{\alpha(a)}\otimes \Omega^{1}_{X}(\log D) \right) \\[+0.3pc]
    \downarrow \wedge \varphi \\[+0.3pc]
    W_{-2,0}\left( H, E_{\alpha(a)}\otimes \Omega^{2}_{X}(\log D) \right) \\[+0.3pc]
    \downarrow \wedge \varphi \\[+0.3pc]
    \vdots  \\[+0.3pc]
    \downarrow \wedge \varphi \\[+0.3pc]
    W_{-2,0}\left( H, E_{\alpha(a)}\otimes \Omega^{d_{X}}_{X}(\log D) \right) 
\end{array}\right] \  \begin{array}{c}
     0 \\[+0.3pc]
    \\[+0.3pc]
    1 \\[+0.3pc]
    \\[+0.3pc]
    2 \\[+0.3pc]
    \\[+0.3pc]
    \vdots  \\[+0.3pc]
    \\[+0.3pc]
    d_{X}
\end{array}
\]
This absolute complex maps
naturally onto the vertical Dolbeault $L^{2}$ complex
\[
\Dol(X/Y,E_{\alpha(a)})
 := 
\left[ \begin{array}{c}
    W_{0}(H,E_{\alpha(a)}) \\[+0.3pc]
    \downarrow \wedge \varphi \\[+0.3pc]
    W_{-2}( H, E_{\alpha(a)})\otimes \Omega^{1}_{X/Y}(\log D) 
\end{array}\right] \  \begin{array}{c}
     0 \\[+0.3pc]
    \\[+0.3pc]
    1 
\end{array}
\]
that  we previously considered.

Using this map we get a short exact sequence of complexes
\begin{equation} \label{eq:sequence}
  \xymatrix@R-0.3pc{ 0 \ar[d] \\
    \Dol(X/Y,E_{\alpha(a)})[-1]\otimes
  f^{*}\Omega^{1}_{Y}(\log Q) \ar[d] \\
  \Dol(X,E_{\alpha(a)})/I^{2}(E_{\alpha(a)}) \ar[d] \\
  \Dol(X/Y,E_{\alpha(a)}) \ar[d] \\ 0 }
\end{equation}
where as usual the subcomplexes $I^{k}(E_{\alpha(a)})$ are defined inductively:
\[
\begin{aligned}
  I^{0}(E_{\alpha(a)}) & =  \Dol(E_{\alpha(a)}) \\[+0.5pc]
  I^{k+1}(E_{\alpha(a)}) & =  \text{image}\left[
    I^{k}(E_{\alpha})\otimes f^{*}\Omega^{1}_{Y}(\log Q)
    \longrightarrow \Dol(E_{\alpha(a)}) \right].
\end{aligned}
\]
By the usual construction \cite{SimpsonFam} the push forward of this sequence
by $f$ yields a connecting homomorphism
\[
\xymatrix@R-0.5pc{
\mathbb{R}^{i}f_{*}\Dol(X/Y,E_{\alpha(a)})  \ar[r]^-{\theta} \ar@{=}[d] &
\mathbb{R}^{i+1}f_{*}\left(\Dol(X/Y,E_{\alpha(a)})[-1]
\otimes \Omega_{Y}^{1}(\log Q)\right) \ar@{=}[d]  \\
\underline{F}^{i}_{a}  & \underline{F}^{i}_{a} \otimes f^{*}\Omega^{1}_{Y}(\log Q)
}
\]
which is a tame Higgs field on the parabolic bundle $\underline{F}_{\bullet}^{i}$.

\subsection{Analytic considerations}
\label{ancon}

This subsection treats the analytic family of $L^2$ cohomology
spaces. These considerations are certainly present in the theory of
Saito-Sabbah-Mochizuki used above, but they are used in a somewhat
roundabout way: Saito and Sabbah used adapted versions of Zucker's
theory \cite{Zucker} in order to take the higher direct image along a
family of curves.

Our original approach to our question was to look at the $L^2$ metric
on cohomology and give some estimates (to be described below) on the
order of growth of holomorphic sections of the higher direct image
bundles $F^1_a$. Using Poincar\'e duality can give estimates in the
other direction. This way of thinking can almost lead to a proof of
the main theorem. However, we need to know the local freeness of the
higher direct images, which is part 1 of Theorem \ref{mainth}. This
local freeness is, fundamentally speaking, a consequence of strictness
for mixed Hodge or twistor structures. That strictness is encapsulated
in the theory of Saito-Sabbah-Mochizuki. Once we know local freeness,
one can either appeal to the full statement of Sabbah's theorem, as we
have done to identify $F^1_a$ as the parabolic structure associated to
the higher direct image local system, or alternatively the analytic
considerations treated in this subsection can also give that
identification. In either case, we still need the present subsection
in order to get part 5 of Theorem \ref{mainth} about the metric.

Choose a K\"ahler metric on $X-D$ that has the local 
behavior of the Poincar\'e metric along the components of $D$,
and of the product of Poincar\'e metrics at crossing points.  
For each $y\in Y-Q$ we get a quasiprojective 
curve $X_y^o:= X_y -D_{H,y}$
and the induced metric is equivalent to the Poincar\'e metric
at the puncture points (i.e. the points of $D_{H,y}$). 

Now, if $(\Ll , \scD ',\scD '', h)$ 
is a harmonic bundle on $X-D$
we obtain its restriction 
$(\Ll _y, \scD'_y,\scD''_y,h)$ to $X_y^o$. 
Define 
$$
{\bf Har}^1(X^o_y,\Ll _y) \subset A^1(X^o_y, \Ll _y)
$$
to be the space of harmonic $1$-forms with coefficients in $\Ll _y$.
These are the forms in the kernel of the Laplacian. Recall
\cite{SimpsonHBLS} that the $\scD_y$-laplacian coincides with twice
the $\scD'_y$ or $\scD''_y$ laplacians.

The analogue of Zucker's theory \cite{Zucker}, for which we may refer
to \cite[Section 6.2]{Sabbah}, tells us that this space is naturally
isomorphic to the cohomology group $H^1(X_y, j_{y,\ast}(L_y))$ where
$j_{y}:X^o_y\hookrightarrow X_y$ is the inclusion and $L_y:= \Ll
_y^{D_y}$ is the restriction of our global local system to
$X^o_y$. Recall that in this case of sheaves on curves,
$j_{y,\ast}(L_y)$ is the {\em middle perversity extension} and
$H^1(X_y, j_{y,\ast}(L_y))$ is the middle perversity intersection
cohomology group of $L_y$ with respect to the compactification $X_y$
of $X_y^o$.

These cohomology spaces vary in a local system, in particular they
have the same dimension. We use without proof the corollary that the
family of spaces of harmonic forms fits together into a $\Cc^{\infty}$
vector bundle denoted
$$
\Hh ={\bf Har}^1((X-D)/(Y-Q),\Ll )
$$
over $Y-Q$, whose fibers are the ${\bf Har}^1(X^o_y,\Ll _y)$.

Furthermore, we also use without proof, the extension of the
calculations in \cite{SimpsonFam}, that were for the case of compact
fibers, showing that the naturally defined operators
$\scD'_{\Hh},\scD''_{\Hh}, \scD_{\Hh}=\scD'_{\Hh}+\scD''_{\Hh}$, and
the $L^2$ metric $h_{\Hh}$ (obained by using the $L^2$ inner product
on each fiber $\Hh _y={\bf Har}^1(X^o_y,\Ll _y)$) are $\Cc ^{\infty}$
and fit together to give a structure of harmonic bundle on $\Hh$ over
$Y-Q$. The underlying flat bundle of this harmonic bundle is
$R^1f_{\ast}(j_{\ast}L)$ where here $j:X-D\hookrightarrow X-D_V$ is
the inclusion into the partial compactification with the horizontal
divisors.

Another way of saying the previous paragraph is that the $L^2$ metric
on the fibers of the local system $R^1f_{\ast}(j_{\ast}L)$, obtained
from the identifications of these fibers with the ${\bf
  Har}^1(X^o_y,\Ll _y)$, is a harmonic metric over $Y-Q$.

We would like to understand the asymptotic behavior of the harmonic
bundle \linebreak $(\Hh , \scD'_{\Hh},\scD''_{\Hh},h_{\Hh})$ near a point $q\in
Q$.

\

\begin{proposition}
\label{tame}
This harmonic bundle is a tame harmonic bundle corresponding to a
filtered local system with trivial filtrations in the terminology of
\cite{SimpsonHBNC} applied to the curve $Y-Q$.
\end{proposition}
\begin{proof}
  Choose $q\in Q$ and consider a coordinate $t$ on a neighborhood of
  $q$ with $t(q)=0$. Along the ray $t\in \rr _{>0}$ suppose we are
  given a family of cohomology classes $\overline{\xi} (t)\in H^1(X_t,
  j_{t,\ast}(L_t))$. We would like to estimate the function
$$
t\mapsto \| \overline{\xi} (t) \| _{h_{\Hh}(t)}.
$$
The norm of the cohomology class is by definition the norm of its
harmonic representative, and this is the minimum of the norms of all
representatives. Thus, given a family of representatives $\xi (t)$ for
the classes $\overline{\xi}(t)$, we get
$$
| \overline{\xi} (t) | _{h_{\Hh}(t)} \leq
\| \xi (t) \| _{L^2, X^o_t}.
$$
We get a family of representatives by choosing a lift of the radial
vector field over the ray, into $X$, and flowing a representative on
the fiber over $t=1$ towards the singular fiber. The norm on the local
system has sub-polynomial growth as we approach the divisor, and a
standard choice of vector field satisfies a sub-polynomial estimate
just as in the case of scalar coefficients. These calculations (which
are not detailed here) give
$$
\| \xi (t) \| _{L^2, X^o_t}\leq C_{\epsilon} t^{-\epsilon}
$$
for any $\epsilon >0$, so 
$$
| \overline{\xi} (t) | _{h_{\Hh}(t)}
\leq C_{\epsilon} t^{-\epsilon} .
$$
This says that the norms of our cohomology classes have sub-polynomial
growth.

Poincar\'e duality for intersection cohomology says that starting from
the dual local system $L^{\ast}$ leads to the dual vector bundle $\Hh
^{\ast}$, and the resulting $L^2$ metric on $\Hh ^{\ast}$ is the dual
of $h_{\Hh}$.  The same estimate holds for the dual, so we get that
flat sections of $\Hh ^{\ast}$ also have sub-polynomial growth.  We
conclude that the harmonic bundle $(\Hh , \scD'_{\Hh}, \scD''_{\Hh},
h_{\Hh})$ is tame \cite{SimpsonHBNC}.
\end{proof}

\

\noindent
This tame harmonic bundle corresponds to a parabolic Higgs bundle
$\underline{G}^1=\{ G^1_a\}$ on $Y$, with parabolic structure on the
divisor $Q$.  Theorem \ref{mainth} identifies this parabolic structure
bundle with the parabolic bundle $\underline{F}^1=\{ F^1_a\}$ obtained
by higher direct image of the $L^2$ parabolic Dolbeault complex. Our
proof used Sabbah's theory in full to identify $\underline{G}^1$. We
indicate here a different proof of part of that, the present proof
being useful in order to fix the identification as stated in part 5 of
Theorem \ref{mainth}.

The fact that $\Delta _{\scD_y} = 2\Delta _{\scD''_y}$ gives an
isomorphism between $L^2$ Dolbeault cohomology and $L^2$ de Rham
cohomology on each fiber.  The analogue of Zucker's theory for our
case \cite{Sabbah,Zucker} tells us that this $L^2$ cohomology is the
same as the hypercohomology of ${\rm DOL}^{\rm par}_{L^2}(X/Y,
\underline{E})$ on each of the fibers (as discussed in
section \ref{L2dol}). In other words, we get an isomorphism
$$
\Hh _y \cong F^1_a(y)
$$
for any $y\in Y-Q$, and (as usual, by some analytic considerations
that we don't treat here) these fit together to give an isomorphism of
$\Cc ^{\infty}$ bundles
\begin{equation}
\label{isohf}
\Hh \cong F^1_a |_{Y-Q}.
\end{equation}
The holomorphic structure $\delbar _{\Hh}$ (the $(0,1)$ component of
$\scD''_{\Hh}$) corresponds to the holomorphic structure of
$F^1_a|_{Y-Q}$. The Higgs field $\varphi _{\Hh}$ corresponds to the
Higgs field on $F^1_a$ given by the Gauss-Manin construction with the
absolute Dolbeault complex discussed in the previous subsection.

On the other hand, $(\Hh , \scD''_{\Hh})$ is also isomorphic to $G^1_a
|_{Y-Q}$ since the latter is by definition the parabolic bundle
associated to $(\Hh , \scD'_{\Hh}, \scD''_{\Hh}, h_{\Hh})$.  Therefore
$$
F^1_a |_{Y-Q}\cong G^1_a |_{Y-Q}
$$
as holomorphic Higgs bundles on $Y-Q$. 
We would like to show that this extends to an isomorphism of 
parabolic bundles, in other words $F^1_a\cong G^1_a$. 

In the parabolic Higgs bundle associated to the harmonic bundle the
piece $G^1_a$ of parabolic weight $a$ at a point $q\in Q$ is the sheaf
of holomorphic sections whose norm is less than $|z|^{-a-\epsilon}$
for any $\epsilon >0$. The identification $F^1_a\cong G^1_a$ is
therefore equivalent to the following statement.

\begin{theorem}
\label{growthorder}
A holomorphic section of $F^1_a |_{Y-Q}$ in a neighborhood of $q\in Q$
is in $F^1_a$ if and only if the section of $\Hh$ corresponding to it
by \eqref{isohf}, has norm bounded by $|z|^{-a-\epsilon}$ for any
$\epsilon >0$.
\end{theorem}
\begin{proof}
  Recall that the parabolic structure on the bundle $E$ along $D_V$ is
  characterized by a similar norm estimate. Using this together with
  the definition of $F^i_a$ and a local estimate for the size of forms
  on the degenerating curves $X_y$ as $y\rightarrow q$ (the same as in
  the constant coefficient case, see \cite{Clemens} for example) we
  conclude one direction of the statement: any section of $F^1_a$ has
  norm bounded by $|z|^{-a-\epsilon}$ for any $\epsilon >0$. (In fact
  with more work one could obtain a more precise estimate of the form
  $|z| ^{-a} |\log |z||^k$ for some $k$.) This proves that
$$
F^1_a\subset G^1_a.
$$
One way to conclude from here would be to calculate by Riemann-Roch
the parabolic degree of $\underline{F}$. That will of course turn out
to be $0$, and since $\underline{G}$ also has parabolic degree $0$
these imply that the two parabolic structures are the same.

One may alternatively proceed, as in the previous proof of tameness,
by using Poincar\'e duality. Let $\ddual E$ denote the dual bundle
assoicated to the dual local system $\ddual L$, and let
$\ddual \Hh$, $\ddual F^1_a$ and $\ddual G^1_a$ be the resulting
objects. Intersection cohomology and $L^2$ cohomology of $\ddual E$
on the fibers $X_y$ are compatible with duality, so the pairing
$$
(\eta , \xi )\mapsto \int _X\eta \wedge \xi 
$$
induces a perfect pairing 
$$
\Hh \times \ddual \Hh \rightarrow \Cc ^{\infty}_{Y-Q}.
$$
It is holomorphic in the Dolbeault realization, and by looking at
parabolic growth rates we get a pairing
\begin{equation}
\label{gpairing}
I_G:G^1_a \otimes \ddual G^1_{-a}\rightarrow \Oo _Y.
\end{equation}
This gives a morphism of parabolic sheaves 
$$
\underline{G}^1\rightarrow 
\left( \ddual \underline{G} \right)^{\vee}
$$
that is an isomorphism over $Y_Q$. Since both have degree $0$ it
follows that it is an isomorphism of parabolic bundles, in other words
the pairing \eqref{gpairing} is a perfect pairing of vector bundles on
$Y$ for any $a$.

The parabolic bundle $\ddual \underline{E}$ is dual to
$\underline{E}$, and this extends to the $L^2$ Dolbeault complex. For
that, it is convenient to use the alternate version of the $L^2$
Dolbeault complex built using terms $W_1$ and $W_{-1}$ instead of
$W_0$ and $W_{-2}$. These two versions are quasiisomorphic, as Zucker
observed \cite{Zucker} (the same reasoning holds in the twistor
case). We get
$$
{\rm DOL}^{\rm par}_{L^2}(X/Y, E_{\alpha (a)})
\cong 
{\rm DOL}^{\rm par}_{L^2}(X/Y, 
\ddual E_{\alpha (-a)})^{\vee}\otimes \omega _{X/Y}[-1].
$$

Now, duality for the morphism $f:x\rightarrow Y$ gives 
a perfect pairing 
$$
I_F:F^1_a\otimes \ddual F^1_{-a}\rightarrow \Oo _Y.
$$
Here is where we appeal to the results of the calculations in the main
part of the paper, that show part 1 of Theorem \ref{mainth}: the
higher direct image sheaves $F^1_a\otimes \ddual F^1_{-a}$ are
bundles, so the duality pairing is a perfect pairing of locally free
sheaves.

These pairings coincide over $Y-Q$. Thus, the inclusions 
$F^1_a\subset G^1_a$ and $\ddual F^1_a\subset \ddual G^1_a$
give a commutative diagram 
$$
\begin{array}{ccc}
F^1_a\otimes \ddual F^1_{-a}&
\stackrel{I_F}{\rightarrow} &\Oo _Y\\
\downarrow & &\downarrow \\
G^1_a \otimes \ddual G^1_{-a}&
\stackrel{I_G}{\rightarrow} &\Oo _Y
\end{array}
$$
with perfect pairings on the top and the bottom. It now follows that
$F^1_a\rightarrow G^1_a$ is an isomorphism. We show that it induces an
injection on fibers over a point $q\in Q$. If $\eta$ is a section of
$F^1_a$ nonvanishing at $q$, then there is a section $\xi$ of
$F^1_{-a}$ such that $I_F(\eta\otimes \xi )(q)\neq 0$. But if $\eta$
maps to a section of $G^1_a$ vanishing at $q$ it would imply that
$I_G(\eta \otimes \xi )(q)=0$, contradicting the commutativity of the
diagram.  This shows that $F^1_a(q)\hookrightarrow G^1_a(q)$.  This
holds at all points of $Q$. Since both bundles have the same rank and
the map is an isomorphism on $Y-Q$ this shows that it is an
isomorphism over all of $Y$. This completes the proof of the theorem.
\end{proof}

In conclusion, the above proof provides our basic compatibility of
Theorem \ref{mainth}, provided we know that the higher direct image
sheaves $F^1_a$ are locally free.

We didn't see how to prove this local freeness property in
general. One could plan to use a strategy based on Steenbrink's
argument \cite{Steenbrink}.  It would take place in explicit normal
crossings situations using a double complex whose terms come from
multiple intersections of the divisor components.  It should be
possible to develop Steenbrink's approach for twistor connections,
using the strictness property of mixed twistor structures, but we
didn't do that. Indeed, Steenbrink's argument was never developed to
its full potential, because the advent of Saito's theory of Hodge
modules provided a very general and more powerful method. We have
taken that route by appealing to \cite{Sabbah} for the proof of local
freeness of the $F^1_a$, but then we also get the calculation of the
higher direct image as part of the same package.

\subsection{Semistable families over higher dimensional base}
\label{hdb}

In this section we take note that our main theorems lead rather
directly to the corresponding statements in the case of families of
curves over a higher dimensional base.

\begin{definition}
\label{splitsemistable}
We say that a morphism $f:(X,D)\rightarrow (Y,Q)$ 
is a {\em split smooth semistable family of curves} if:
\begin{enumerate}
\item[(1)] $X$ and $Y$ are smooth projective varieties, 
\item[(2)] $D$ and $Q$ are reduced divisors with simple normal crossings,
\item[(3)] all the fibers of $f$ are reduced curves,
\item[(4)] we have a decomposition $D=D_V+D_H$ where 
$D_V = f^{-1}(Q)$ and $D_H$ is a disjoint union of components
mapping locally isomorphically to $Y$, 
\item[(5)]the map $f$ is smooth away from $D_V$, and
\item[(6)] for $y\in Q$, the curve $X_y:= f^{-1}(y)$ is a semistable
  curve with only nodes. (It follows that  the nodes are distinct from the points
  marked by the components of $D_H$.)
\end{enumerate}
\end{definition}

\

\noindent
Suppose $f:(X,D)\rightarrow (Y,Q)$ is a split smooth semistable
family. We use notational conventions analogous to those in effect up
until now.

Suppose $L$ is a local system on $X-D$ such that the eigenvalues of
the monodromy around all components of $D$ are in $S^1\subset
\cc^{\times}$. Define as before the associated harmonic bundle $(\Ll ,
\scD ', \scD'', h)$, the associated parabolic bundle with
$\lambda$-connection $\underline{\srE}=\{ \srE_{\beta}\}$, and the
associated $\srR _{\srX}$-module $\srE$. Suppose $a$ is a parabolic
weight for the divisor $Q$, and define the parabolic weight $\alpha
(a)$ by associating the weight $a_i$ to any component $D_j$ of $D_V$
mapping to the component $Q_i$ of $Q$. The weight of $\alpha (a)$
associated to components of $D_H$ is $0$. The $0$-th associated graded
pieces along the horizontal divisor components $D_{h(j)}\subset D_H$
are defined as previously, and they vary in a locally constant family
by \cite{MochizukiPure}. Define the $L^2$ parabolic de Rham and
Dolbeault complexes
$$
\text{DR}^{\rm par}_{L^2}(\srX /\srY ; \srE  _{\alpha (a)}), 
\;\;\;\;
\text{DOL}^{\rm par}_{L^2}(X/Y ; E  _{\alpha (a)})
$$
by the same formulas \eqref{l2pardR} and \eqref{l2pardol} using the weight filtrations on horizontal
complexes as before.

On the other hand, we have the $\srR$-module de Rham complex
$$
\text{DR}(\srX /\srY ; g_+(\srE ))
$$
defined as before using the graph embedding $g: \srX \rightarrow
\srP := \srX\times _{\srA} \srY$. 

\begin{theorem}
\label{higherdim}
In the above situation, the higher direct images 
$$
\srF ^i_a:=\mathbb{R}^if_{\ast}\text{DR}^{\rm par}_{L^2}(\srX /\srY
; \srE _{\alpha (a)})
$$
are locally free on $\srY$ and they fit together to form a
parabolic vector bundle with $\lambda$-connection. Their restriction
to $\lambda =0$ is
$$
\srF^i_a (0)= F^i_a = 
\mathbb{R}^if_{\ast}\text{DOL}^{\rm par}_{L^2}(X/Y ; E  _{\alpha (a)})
$$
and these fit together to form a parabolic Higgs bundle.  These
parabolic Higgs bundles and parabolic bundles with
$\lambda$-connection are those associated to the local system $G^i$ of
middle perversity higher direct images of $L$ to $Y-Q$. The higher
direct image (under the projection $p:\srP\rightarrow \srY$)
$$
\srF^i:= \mathbb{R}^ip_{\ast}\text{DR}(\srX /\srY ; g_+(\srE ))
$$
are strictly $S$-decomposable $\srR_{\srY}$-modules 
(see Remark \ref{s-term})
whose piece of
strict support $\srY$ is equal to the main chart of the pure twistor
$\Dd$-module associated to $G^i$.
\end{theorem}

The identification of the theorem is functorial, in particular the
component sheaves of the parabolic structure reflect the growth rate
of the $L^2$ harmonic metric on the local system $G^i$.

\begin{proof}
If $C\subset Y$ is a curve immersed into $Y$ and transverse to each
boundary divisor $Q_i$ individually, then $X\times _YC$ is a smooth
surface mapping to $C$ by a split smooth semistable map. Hence, the
considerations of the main part of the paper apply. Notice that we may
choose such curves passing through any multiple intersection of $Q$
(since we only asked transversality to each component). We obtain that
the ranks of the higher direct image sheaves $\srF ^i_a$ or $F^i_a$
are constant over $Y$. Since these are cohomology sheaves of perfect
complexes it follows from semicontinuity that they are locally free,
and the higher direct image is compatible with base change. Now,
notice that we know from \cite{MochizukiPure} that the local system
$G^i$ corresponds to a harmonic bundle, which in turn corresponds to a
parabolic Higgs bundle and parabolic $\lambda$-connection. We have an
identification with $F^i_a$ or $\srF ^i_a$ over $Y-Q$ (resp. $\srY
-\srQ$). But now, these bundes over the open set have two extensions
to parabolic bundles on $Y$ with respect to $Q$, namely on one hand
the $F^i_a$ (resp. $\srF ^i_a$), and on the other hand the parabolic
bundles associated to $G^i$. For any smooth embedded curve $C$
transverse to the $Q_i$, the restrictions of the two parabolic bundles
to $C$ coincide. It follows that the two parabolic bundles are the
same.
\end{proof}

\

\subsection{Higher dimensional families over a curve}

One may similarly ask the question of how to generalize our result to
the case of higher dimensional fibers. Note first of all that a higher
dimensional map can be decomposed into a series of $1$-dimensional
fibrations, using alterations \cite{AbramovichKaru}. Therefore, in
principle Theorem \ref{higherdim} can be applied inductively to obtain
some approximation of the higher direct image. This is an
approximation because there may be extra terms along the way coming
from the birational transformations and finite coverings involved in
making the required alterations.

It is therefore natural to ask for a global formula. The semistable
reduction theorem of Abramovich and Karu in the case of higher
dimensional base and higher dimensional fibers involves reduction to toric
singularities \cite{AbramovichKaru}. It is an interesting and
important question to understand how to calculate in this situation,
but that would go way beyond the scope of the methods that we are
discussing here.

We can, nonetheless, ask about the case of higher dimensional fibers
over a $1$-dimensional base. This case also presents a certain
collection of difficulties, and we are not able to state a theorem
about it at the present time. Let us review some of these difficulties
and discuss what might be done.

Suppose first of all that there are no horizontal divisors, in other
words the map $X-D\rightarrow Y-Q$ is proper. In this case, we don't
need to consider the intersection cohomology or $L^2$ cohomology on
the fibers.  Suppose $L$ is a local system on $X-D$ whose monodromy
eigenvalues are in $S^1\subset \cc ^{\times}$, provided with trivial
filtrations, and let $(\underline{E},\varphi )$ be the parabolic
logarithmic Higgs bundle associated to it by
\cite{MochizukiKobayashiHitchin}. We would like to obtain the
parabolic Higgs bundle on $Y$ associated to the local system
$R^if_{\ast}(L)$ on $Y-Q$.

Given a parabolic weight $a$ (at a point $q\in Q$ in the base), let
$\alpha (a)$ be the parabolic weight on $D$ obtained by assigning $a$
to each divisor component. We have a bundle $E_{\alpha (a)}$ on $X$
with logarithmic Higgs field
$$
E_{\alpha (a)}\stackrel{\varphi}{\rightarrow}
E_{\alpha (a)}\otimes \Omega ^1_X(\log D)
$$
inducing the relative Higgs field with values in 
$$
\Omega ^1_{X/Y}(\log D):= \Omega ^1_X(\log D)/f^{\ast}\Omega ^1_Y(\log Q).
$$
The relative Dolbeault complex is
$$
{\rm DOL}(X/Y,E_{\alpha (a)}):= 
\left[ 
\ldots \stackrel{\varphi}{\rightarrow}
E_{\alpha (a)}\otimes \Omega ^i_{X/Y}(\log D)
\stackrel{\varphi}{\rightarrow}\ldots \right].
$$
The generalization of Theorem \ref{mainth} to this case should say that 
$$
F^i_a:= \mathbb{R}^if_{\ast} 
{\rm DOL}(X/Y,E_{\alpha (a)}) 
$$
are locally free over $Y$, compatible with base change, and they fit
together to form a parabolic Higgs bundle $\underline{F}^i$; and this
is the parabolic Higgs bundle associated to $R^if_{\ast}(L)$.

In this case, the proof should in principle follow the same outline as
what we have done here. One would need to identify the $V$-filtration
for the graph embedding, in the case of a normal crossing of several
divisors. This should be an algebraic problem similar to the one we
have treated here for the crossing of two divisors, perhaps only
requiring a more general notation.

The more difficult case is when there is a horizontal divisor with
normal crossings. The first question is how to define the appropriate
Dolbeault complex. It is natural to conjecture that it should be the
complex consisting of holomorphic $E$-valued forms that are in $L^2$
(for the Poincar\'e metric near $D$), such that their derivatives are
in $L^2$. Jost and Zuo have proven this theorem in the case of
variations of Hodge structure \cite{JostZuo}.

Next, it would be good to have an algebraic description in terms of
weight filtrations. This is bound to be considerably more complicated
than in the case of relative dimension $1$, because there are several
different weight filtrations that interact in a subtle way.

Assuming a good understanding of these issues, the resulting algebraic
structures will then interact with the parabolic weights along the
vertical divisors, at points of $D_H\cap D_V$.  Identifying the
$V$-filtration will be a significant question since normal crossing
points of $D_V$ can touch points (or even crossing points) of
$D_H$. One of the crucial simplifications in our arguments of the
present paper was that the crossing points of $D_V$ were disjoint from
$D_H$ and the two aspects could be treated separately.

This rough overview highlights the difficulties that would be involved
in obtaining directly a generalized formula for families of higher
relative dimension.

The other path is to decompose a morphism $f$ into a sequence of
one-dimensional fibrations, by making alterations at each stage, and
applying Theorem \ref{higherdim} inductively. This has the advantage
of being currently accessible, however in practice it will require the
undoubtedly complicated investigation of what happens under the
alterations that occur in the middle of the process.

Note however that if we are in the special situation where the
horizontal divisors are smooth over $Y$ and meet only smooth points of
the vertical divisors all these difficulties disapear and we can give
an algebraic description of the $L^{2}$ Dolbeault complex. 
We give this description next.

Suppose $f : X \to Y$ is a morphism between smooth projective
varieties.  Suppose also that $D \subset X$ and $Q \subset Y$ are 
reduced simple normal crosings divisors such that $D$ decomposes as $D
= D_{H} + D_{V}$, with $D_{V} = f^{-1}(Q)$ scheme theoretically, and
that $f$ is smooth away from $D_{V}$. Additionally we assume that each
component of $D_{H}$ is smooth over $Y$ and that $D_{H}\cup
D_{V}$ is contained in the smooth locus of $D_{V}$. Note that these
assumptions in particular imply that $D_{H}$ is a disjoint union of
smooth connected components, i.e. $D_{H}$ embeds in the normalization
$D^{\nu}$of $D$ as a union of connected components of $D^{\nu}$.

Let $\left(\underline{E}_{\bullet},\varphi\right)$ be a tame parabolic
Higgs bundle on $(X,D)$ and let $a$ be a parabolic level along $Q$.
By definition $a$ is an assignment of a real number to each
irreducible component of $Q$.  As before we will write $\alpha(a)$ for
the parabolic level along $D$ which assigns $0$ to each horizontal
component of $D$, while to a vertical component of $D$ it assigns the
value of $a$ on the image of this vertical component under $f$.  As
before we consider the level $\alpha(a)$ representative
$E_{\alpha(a)}$ of the parabolic bundle $\underline{E}_{\bullet}$ and
the horizontal weight filtration $W_{\ell}(H,E_{\alpha(a)})$ of
$E_{\alpha(a)}$.

Explicitly, for each component $D_{h(j)}$ of $D_{H}$
we have the associated graded bundle
\[
\text{Gr}_{h(j),0}(E_{\alpha(a)}) = E_{\alpha(a)}/E_{\alpha(a) - \epsilon\delta^{h(j)}}
\]
and the induced nilpotent endomorphism $\text{res}_{D_{h(j)}} \varphi$
 of $\text{Gr}_{h(j),0}(E_{\alpha(a)})$.  This gives rise to an
associated monodromy weight filtration $W_{\ell}
\left(\text{Gr}_{h(j),0}(E_{\alpha(a)})\right)$. Set
\[
\begin{aligned}
\text{Gr}_{H,0}(E_{\alpha(a)}) & = \bigoplus_{j}
\text{Gr}_{h(j),0}(E_{\alpha(a)}), \\
W_{\ell}(\text{Gr}_{H,0}(E_{\alpha(a)})) & = \bigoplus_{j} W_{\ell}
\left(\text{Gr}_{h(j),0}(E_{\alpha(a)})\right).
\end{aligned}
\]
Since in our setup all $D_{h(j)}$ are disjoint, we can view
$\text{Gr}_{H,0}(E_{\alpha(a)}) $ as a torsion sheaf on $X$ equipped
with a surjective sheaf map
\[
E_{\alpha(a)} \to \text{Gr}_{H,0}(E_{\alpha(a)}).
\]
Following the pattern in section~\ref{sec-locstudy} we define
$W_{\ell}(H,E_{\alpha(a)})$ as the preimage of the monodromy weight
filtration $W_{\ell}(\text{Gr}_{D_{H},0})$ associated to the action of
the nilpotent $\text{res}_{D_{H}} \varphi$.

For each form degree $i \geq 0$ we have\footnote{The standard Poincare residue 
$\text{res}_{D} : \Omega^{i}_{X}(\log D) \longrightarrow
\Omega^{i-1}_{D^{\nu}}$ maps logarithmic $i$-forms on
$(X,D)$ to holomorphic $i-1$ forms on the normalization
$D^{\nu}$. Since $D_{H} \subset D^{\nu}$ we can compose
$\text{res}_{D}$ with the projection $\Omega^{i-1}_{D^{\nu}}
\twoheadrightarrow \Omega^{i-1}_{D_{H}}$  to get a residue map $\text{res}_{D_{H}}$.}
a residue map
\[
\text{res}_{D_{H}} : \Omega^{i}_{X}(\log D) \longrightarrow
\Omega^{i-1}_{D_{H}},
\]
and after tensoring with $W_{0}(H,E_{\alpha(a)})$, a residue map
\[
\text{res}_{D_{H}} : W_{0}(H,E_{\alpha(a)})\otimes \Omega^{i}_{X}(\log D) \longrightarrow
W_{0}(H,E_{\alpha(a)})_{|D_{H}} \otimes \Omega^{i-1}_{D_{H}}.
\]
Similarly, for every $i \geq 0$ we have a residue map on $f$-relative logarithmic forms
\[
\text{res}_{D_{H}} : W_{0}(H,E_{\alpha(a)})\otimes \Omega^{i}_{X/Y}(\log D) \longrightarrow
W_{0}(H,E_{\alpha(a)})_{|D_{H}} \otimes \Omega^{i-1}_{D_{H}/Y}.
\]
Pulling back the subbundles
\[
\begin{aligned}
  W_{-2}(H,E_{\alpha(a)})_{|D_{H}} \otimes \Omega^{i-1}_{D_{H}} &
  \subset W_{0}(H,E_{\alpha(a)})_{|D_{H}} \otimes \Omega^{i-1}_{D_{H}} \\
W_{-2}(H,E_{\alpha(a)})_{|D_{H}} \otimes \Omega^{i-1}_{D_{H}/Y} &
\subset W_{0}(H,E_{\alpha(a)})_{|D_{H}} \otimes \Omega^{i-1}_{D_{H}/Y}
\end{aligned}
\]
by these residue maps yields locally free subsheaves
\[
\begin{aligned}
  W_{-2,0}\left(H,E_{\alpha(a)}\otimes \Omega^{i}_{X}(\log D)\right) &
  \subset W_{0}(H,E_{\alpha(a)})\otimes \Omega^{i}_{X}(\log D), \\
  W_{-2,0}\left(H,E_{\alpha(a)}\otimes \Omega^{i}_{X/Y}(\log D)\right) &
  \subset W_{0}(H,E_{\alpha(a)})\otimes \Omega^{i}_{X/Y}(\log D).
\end{aligned}
\]
By construction these subsheaves are preserved by $\varphi$ and so we
get absolute and relative parabolic Dolbeault 
complexes:
\begin{equation} \label{eq-absolute-general}
\Dol(X,E_{\alpha(a)})
 := 
\left[ \begin{array}{c}
    W_{0}(H,E_{\alpha(a)}) \\[+0.3pc]
    \downarrow \wedge \varphi \\[+0.3pc]
    W_{-2,0}\left( H, E_{\alpha(a)}\otimes \Omega^{1}_{X}(\log D) \right) \\[+0.3pc]
    \downarrow \wedge \varphi \\[+0.3pc]
    W_{-2,0}\left( H, E_{\alpha(a)}\otimes \Omega^{2}_{X}(\log D) \right) \\[+0.3pc]
    \downarrow \wedge \varphi \\[+0.3pc]
    \vdots  \\[+0.3pc]
    \downarrow \wedge \varphi \\[+0.3pc]
    W_{-2,0}\left( H, E_{\alpha(a)}\otimes \Omega^{d_{X}}_{X}(\log D) \right) 
\end{array}\right] \  \begin{array}{c}
     0 \\[+0.3pc]
    \\[+0.3pc]
    1 \\[+0.3pc]
    \\[+0.3pc]
    2 \\[+0.3pc]
    \\[+0.3pc]
    \vdots  \\[+0.3pc]
    \\[+0.3pc]
    d_{X}
\end{array}
\end{equation}
and
\begin{equation} \label{eq-relative-general}
\Dol(X/Y,E_{\alpha(a)})
 := 
\left[ \begin{array}{c}
    W_{0}(H,E_{\alpha(a)}) \\[+0.3pc]
    \downarrow \wedge \varphi \\[+0.3pc]
    W_{-2,0}\left( H, E_{\alpha(a)}\otimes \Omega^{1}_{X/Y}(\log D) \right) \\[+0.3pc]
    \downarrow \wedge \varphi \\[+0.3pc]
    W_{-2,0}\left( H, E_{\alpha(a)}\otimes \Omega^{2}_{X/Y}(\log D) \right) \\[+0.3pc]
    \downarrow \wedge \varphi \\[+0.3pc]
    \vdots  \\[+0.3pc]
    \downarrow \wedge \varphi \\[+0.3pc]
    W_{-2,0}\left( H, E_{\alpha(a)}\otimes \Omega^{d_{X/Y}}_{X/Y}(\log D) \right) 
\end{array}\right] \  \begin{array}{c}
     0 \\[+0.3pc]
    \\[+0.3pc]
    1 \\[+0.3pc]
    \\[+0.3pc]
    2 \\[+0.3pc]
    \\[+0.3pc]
    \vdots  \\[+0.3pc]
    \\[+0.3pc]
    d_{X}-d_{Y}
\end{array}
\end{equation}

\

\begin{remark} \label{rem-compare-curve}
Note that if we choose local coordinates $z_{i}$ on $X$ so that a
component $D_{h(j)}$ of $D_{H}$ is given by the equation $z_{1} = 0$, then on
this local chart $\Omega^{p}_{X}(\log D)$ decomposes as
\[
\Omega^{p}_{X}(\log D) =
\wedge^{p}(\oplus_{i \geq 2} \mathcal{O}dz_{i})\oplus \frac{dz_{1}}{z_{1}}
\wedge (\wedge^{p-1}(\oplus_{i \geq 2} \mathcal{O}dz_{i}).
\]
Using this decomposition and the definition of $W_{-2,0}$ 
we get an identification

\

\[
\begin{aligned}
  W_{-2,0}\left( H, E_{\alpha(a)}\otimes \Omega^{i}_{X}(\log D) \right)  & \\[0.5pc]
& \hspace{-1.4in}= 
\left[
\begin{array}{c}
  W_{0}(H,E_{\alpha(a)})\otimes \wedge^{p}(\oplus_{i \geq 2} \mathcal{O}dz_{i})\\ \bigoplus \\
(W_{-2}(H,E_{\alpha(a)}) + z_{1} W_{0}(H,E_{\alpha(a)})) \otimes \frac{dz_{1}}{z_{1}}
  \wedge (\wedge^{p-1}(\oplus_{i \geq 2} \mathcal{O}dz_{i})
\end{array}
\right].
\end{aligned}
\]
By the definition of a parabolic bundle we have that $z_{1}
W_{0}(H,E_{\alpha(a)})) \subset E_{\alpha(a) - \delta^{h(j)}}$ and so
under the natural map $W_{0}(H,E_{\alpha(a)})) \twoheadrightarrow
W_{0}(\text{Gr}_{h(j),0})$ the subsheaf $z_{1}
W_{0}(H,E_{\alpha(a)}))$ maps to zero. This implies that $z_{1}
W_{0}(H,E_{\alpha(a)})) \subset W_{-2}(H,E_{\alpha(a)}))$ and hence
\begin{equation} \label{eq-W-20}
  W_{-2,0}\left( H, E_{\alpha(a)}\otimes \Omega^{i}_{X}(\log D) \right)   = 
\left[
\begin{array}{c}
  W_{0}(H,E_{\alpha(a)})\otimes \wedge^{p}(\oplus_{i \geq 2}
  \mathcal{O}dz_{i})\\ \bigoplus \\ W_{-2}(H,E_{\alpha(a)})\otimes
  \frac{dz_{1}}{z_{1}} \wedge (\wedge^{p-1}(\oplus_{i \geq 2}
  \mathcal{O}dz_{i})
\end{array}
\right].
\end{equation}
This formula implies immediately that in the case when $f : X \to Y$
is of relative dimension one, the complexes $\Dol(X,E_{\alpha(a)})$
and $\Dol(X/Y,E_{\alpha(a)})$ we just defined coincide with the
complexes defined in section~\ref{globalcomplex}. 
\end{remark}

\

\noindent
Tautologically we again get  a short exact sequence of complexes
\begin{equation} \label{eq:sequence2}
  \xymatrix@R-0.3pc{ 0 \ar[d] \\
    \Dol(X/Y,E_{\alpha(a)})[-1]\otimes
  f^{*}\Omega^{1}_{Y}(\log Q) \ar[d] \\
  \Dol(X,E_{\alpha(a)})/I^{2}(E_{\alpha(a)}) \ar[d] \\
  \Dol(X/Y,E_{\alpha(a)}) \ar[d] \\ 0 }
\end{equation}
with $I^{k}(E_{\alpha(a)})$  defined inductively as in section~\ref{globalcomplex}.

Again pushing forward this short exact sequence by $f$ will give rise
to a parabolic Higgs sheaf on $Y$. If $(\underline{E},\varphi)$ comes
from a harmonic bundle we expect that this push forward will
correspond via the NAHC with the $L^{2}$ push forward of this harmonic
bundle. This can be verified in two important special cases.

First, Remark~\ref{rem-compare-curve} and Theorem~\ref{higherdim}
imply that this holds if $f$ is of relative dimension one. Second, the
analytic considerations in the section~\ref{ancon} imply that this
statement holds in the case when $Y$ is a point, i.e. when we are
dealing with the global cohomology of a harmonic bundle on the
complement of a smooth divisor.

\

\

\bigskip

\noindent
Ron Donagi, {\sc University of Pennsylvania}, donagi@math.upenn.edu

\smallskip

\noindent
Tony Pantev, {\sc University of Pennsylvania}, tpantev@math.upenn.edu

\smallskip

\noindent
Carlos Simpson, {\sc CNRS, Universit\'e C\^ote d'Azur, LJAD}, Carlos.SIMPSON@unice.fr

\end{document}